\documentclass{amsart}
\usepackage{amscd, amssymb, amsmath, amsthm}
\usepackage{amsfonts,enumerate, latexsym}
\usepackage{graphics,graphicx,psfrag, mathrsfs}
\usepackage[all]{xy}
\usepackage[colorlinks=true]{hyperref}

\newtheorem{theorem}{Theorem}[section]
\newtheorem{lemma}[theorem]{Lemma}
\newtheorem{proposition}[theorem]{Proposition}
\newtheorem{corollary}[theorem]{Corollary}

\newtheorem{construction}[theorem]{Construction}

\theoremstyle{definition}
\newtheorem{definition}[theorem]{Definition}
\newtheorem{convention}[theorem]{Convention}

\newtheorem{question}[theorem]{Question}
\newtheorem{remark}[theorem]{Remark}

\newcommand{\RR}{{\mathbf R}}
\newcommand{\QQ}{{\mathbf Q}}
\newcommand{\CC}{{\mathbf C}}
\newcommand{\ZZ}{{\mathbf Z}}

\newcommand{\Hyp}{{\mathbb{H}^3}}

\newcommand{\pt}{{*}}

\newcommand{\SO}{{\mathrm{SO}}}

\newcommand{\Isom}{{\mathrm{Isom}}}

\newcommand{\id}{{\mathrm{id}}}

\newcommand{\imunit}{{\mathrm{i}}}

\newcommand{\ocurves}{{\mathbf{\Gamma}}}

\newcommand{\opants}{{\mathbf{\Pi}}}
\newcommand{\vistorus}{{\mathscr{N}}}
\newcommand{\ocobordism}{{\mathbf{\Omega}}}

\newcommand{\chlen}{{\mathbf{hl}}}
\newcommand{\clen}{{\mathbf{l}}}

\newcommand{\meas}{{\mathcal{M}}}
\newcommand{\zmeas}{{\mathcal{ZM}}}
\newcommand{\bmeas}{{\mathcal{BM}}}
\newcommand{\fboundary}{{\partial^\sharp}}
\newcommand{\nboundary}{{\partial^\flat}}

\newcommand{\ini}{{\mathtt{ini}}}
\newcommand{\ter}{{\mathtt{ter}}}
\newcommand{\phasor}{{\boldsymbol{\lambda}}}

\title[Homology and QF subsurfaces]{Homology of curves and surfaces in closed hyperbolic $3$-manifolds}

\author[Y.~ Liu]{%
        Yi Liu}
\address{%
    Mathematics 253-27\\
    California Institute of Technology\\
    Pasadena, CA 91125}
\email{%
    yliumath@caltech.edu}

\author[{V.~ Markovic}]{
	Vladimir Markovic}
\address{
	Mathematics 253-27\\
	California Institute of Technology\\
	Pasadena, CA 91125}
\email{
	markovic@caltech.edu}

\thanks{Supported by NSF grant No.~DMS 1308836}

\date{%
 \today}

\begin{document}

\begin{abstract} Among other things, we prove the following two topologcal statements about closed hyperbolic $3$-manifolds. First,
every rational second homology class of a closed hyperbolic $3$-manifold  has a positve integral multiple  represented by an oriented connected closed $\pi_1$-injectively immersed quasi-Fuchsian subsurface.
Second, every rationally null-homologous,  $\pi_1$-injectively immersed oriented closed $1$-submanifold in a closed hyperbolic $3$-manifold has an equidegree finite cover which bounds an oriented connected compact 
$\pi_1$-injective immersed quasi-Fuchsian subsurface.  In part, we  exploit  techniques developed by Kahn and Markovic in \cite{KM-surfaceSubgroup,KM-Ehrenpreis}, but  we only distill geometric and topological ingredients from those papers  so no hard analysis is involved in this paper. 
\end{abstract}

\maketitle

\section{Introduction}\label{Sec-introduction}

	In this paper, we are concerned about
	the construction problem of homologically interesting
	connected  quasi-Fuchsian subsurfaces in
	closed hyperbolic $3$-manifolds.
	We show that in a closed hyperbolic $3$-manifold,
	it is always possible to construct
	an oriented compact connected $\pi_1$-injectively immersed
	quasi-Fuchsian subsurface which is virtually bounded by
	prescribed multicurves
	and which virtually represents a prescribed rational relative second homology class
	(Theorem \ref{main-qfSurface}).

 The following two results are  motivational special cases of Theorem \ref{main-qfSurface}. For simplicity we  state them first. 
		
	\begin{corollary}\label{qfSurfaceClosed}
		Every rational second homology class of a closed hyperbolic $3$-manifold 
		has a positve integral multiple 
		represented by an oriented connected closed $\pi_1$-injectively
		immersed quasi-Fuchsian subsurface.
	\end{corollary}
	
	\begin{corollary}\label{qfSurfaceBoundary}
		Every rationally null-homologous, 
		$\pi_1$-injectively immersed oriented closed
		$1$-submanifold in a closed hyperbolic $3$-manifold
		has an equidegree finite cover which bounds an oriented connected compact 
		$\pi_1$-injective immersed quasi-Fuchsian subsurface.
	\end{corollary}
	
	Here the closed $1$-submanifold being $\pi_1$-injectively immersed
	means that all components are homotopically nontrivial, and 
	a finite cover being equidegree means that the covering degree
	does not vary over different components of the $1$-submanifold. 
	However, we do not
	require the finite cover to be connected restricted to any component
	of the closed $1$-submanifold.
	
	Corollary \ref{qfSurfaceClosed} was a question that was recently (and informally) raised by William Thurston.
	Note that if not requiring the subsurface to be connected, one may easily
	obtain a componentwise quasi-Fuchsian embedded incompressible subsurface 
	representing a second homology class that is nontrivial and non-fibered,
	or obtain a componentwise $\pi_1$-injectively
	immersed quasi-Fuchsian representative subsurface, 
	using the Cooper--Long--Reid construction \cite{CLR} in the fibered case
	or the Kahn--Markovic construction \cite{KM-surfaceSubgroup} in the trivial case.
	In the paper \cite{CW}, Danny Calegari and Alden Walker show that
	in a random group at any positive density, 
	many second homology classes can be rationally represented by
	quasiconvex (closed) surface subgroups (cf.~Remark 6.4.2 of \cite{CW}).	
	Corollary \ref{qfSurfaceBoundary} answers a question
	of Calegari in the case of closed hyperbolic $3$-manifold groups.
	Calegari proved the surface group case \cite{Calegari}
	but his question remains widely open for hyperbolic groups in general.	

	Next,  we state our main result Theorem \ref{main-qfSurface}. A compact immersed subsurface $F$ of 
	a closed hyperbolic $3$-manifold $M$ is \emph{quasi-Fuchsian} if it is an 
	essential subsurface
	of a closed $\pi_1$-injectively immersed quasi-Fuchsian subsurface of $M$.
	Perhaps it would be better to call $F$ `quasi-Schottky'
	if it is quasi-Fuchsian with nonempty boundary.
	
	\begin{theorem}\label{main-qfSurface}
		Let $M$ be a
		closed hyperbolic $3$-manifold, and $L\subset M$
		be the (possibly empty) 
		union of finitely many mutually disjoint, $\pi_1$-injectively embedded
		loops. Then for any relative homology class
		$\alpha\in H_2(M,L;\,\QQ)$, there exists an oriented connected compact surface $F$,
		and an immersion of the pair
			$$j:\,(F,\partial F)\looparrowright (M,L),$$
		such that $j$ is $\pi_1$-injective and quasi-Fuchsian,
		and that $F$ represents a positive integral multiple of $\alpha$.
	\end{theorem}
	
	The reader is referred to Subsection \ref{Subsec-descriptionOfTheProblem} for
	more explanation about the formulation.
	In fact, the proof also implies that
	the claimed immersed subsurface 
	is nearly geodesic and nearly regularly panted (cf.~Section \ref{Sec-methodology}).	

	In the course of proving Theorem \ref{main-qfSurface},  we revisit
	the techniques developed in the work of
	Kahn--Markovic in \cite{KM-surfaceSubgroup,KM-Ehrenpreis},
	with an attempt to distill the topological ingredients from those papers.
	Specifically, we quote Theorems 2.1, 3.4, and 4.2 from \cite{KM-surfaceSubgroup}
	as black boxes, so details in quasi-conformal geometry or dynamics of 
	frame flow are not involved in our proofs.
	In order to prove results in this paper we are not required to generalize
	the quantitative aspect of the good correction theory, which is Thereom 3.3 (3)
	of \cite{KM-Ehrenpreis}, so randomization techniques are not
	required for our discussion.
	On the other hand, we recall the gluing construction of \cite{KM-surfaceSubgroup}, and 
	reform the topological part of the good correction theory of \cite{KM-Ehrenpreis}.
	The treatment of this paper is
	completely self-contained except for the quoted results from \cite{KM-surfaceSubgroup}.

	The connectedness of the surface $F$ in the conclusion of Theorem \ref{main-qfSurface}
	comes from improving the gluing construction of \cite{KM-surfaceSubgroup}.
	The idea of the construction of \cite{KM-surfaceSubgroup} is to build a closed $\pi_1$-injectively immersed
	quasi-Fuchsian subsurface in a closed hyperbolic $3$-manifold by gluing a sufficiently
	large finite collection of nearly regular pairs of pants with nearly evenly distributed feet.
	A crucial criterion was proved in \cite[Theorem 2.1]{KM-surfaceSubgroup}, asserting that
	a nearly unit shearing gluing yields the $\pi_1$-injectiveness and the quasi-Fuchsian property.
	In Section \ref{Sec-methodology}, we will review the program in more details with emphasis
	on the boundary operator on measures of nearly regular pairs of pants.
	However, the criterion of \cite{KM-surfaceSubgroup} does not necessarily produce a
	connected surface, so we provide a slightly stronger criterion (Theorem \ref{gluing})
	which ensures connectedness of the output. 
	The new criterion will be proved in Section \ref{Sec-quasiFuchsianConnectedGluing}
	by applying a trick called \emph{hybriding}. On the other hand, the assumptions
	of the new criterion
	are not hard to be satisfied, for instance, cf.~Theorem \ref{homologyViaPants}.
		
	The control of the homology class of the surface $F$ in the conclusion of Theorem \ref{main-qfSurface}
	comes from extending and strengthening the non-random good correction theory of \cite{KM-Ehrenpreis} in the
	$3$-dimensional case. For an oriented closed hyperbolic $3$-manifold $M$,
	we will reformulate the Good Pants Homology introduced in \cite{KM-Ehrenpreis}
	as the \emph{nearly regularly panted cobordism group} $\ocobordism_{R,\epsilon}(M)$
	 (Definition \ref{pantedCobordismGroup}). In Section \ref{Sec-pantedCobordismGroup},
	we will find a canonical isomorphism $\Phi$
	between $\ocobordism_{R,\epsilon}(M)$ and 
	the first integral homology of the special orthonormal frame bundle $\SO(M)$
	over $M$ (Theorem \ref{theoremPantedCobordism}).
	This isomorphism fully characterizes the structure of $\ocobordism_{R,\epsilon}(M)$,
	and improves the treatment of non-random good correction theory 
	of \cite{KM-Ehrenpreis} in that it accounts for the torsion part 
	which was previously ignored. In Section \ref{Sec-pantifyingSecondHomologyClasses}, we will further show that 
	any second integral homology class of $M$ can be represented by an oriented
	closed nearly regularly panted subsurface (Theorem \ref{secondHomologyClass}). With an extra
	property called \emph{nearly regularly panted connectedness}
	introduced in Section \ref{Sec-pantedConnectedness}, our study of nearly regularly panted
	cobordisms can be summarized by the following Theorem \ref{main-pantedSurface},
	stated in a form analogous to Theorem \ref{main-qfSurface} (cf.~Section \ref{Sec-methodology}
	for the notations). Note that Corollaries \ref{qfSurfaceBoundary}
	and \ref{qfSurfaceClosed} are also parallel to 
	Theorems \ref{theoremPantedCobordism} and \ref{secondHomologyClass} 
	in their statements respectively. 
	These results are all based on geometric constructions
	using $\partial$-framed segments as we will study in Section \ref{Sec-basicConstructions}.
	
	\begin{theorem}\label{main-pantedSurface}
		Let $M$ be closed hyperbolic $3$-manifold. 
		For any universally small positive constant $\epsilon$ 
		and any sufficiently large positive constant $R$ depending only on
		$M$ and $\epsilon$, the following holds.
		There exists a nontrivial invariant
		$\sigma(L)$ valued in $\ZZ_2$, defined for all null-homologous 
		oriented $(R,\epsilon)$-multicurve $L$ in $M$, satisfying
		the following.
		\begin{enumerate}
			\item For any null-homologous oriented $(R,\epsilon)$-multicurve $L_1,L_2$,
				$$\sigma(L_1\sqcup L_2)=\sigma(L_1)+\sigma(L_2).$$
			\item The invariant $\sigma(L)$ vanishes if and only if $L$ bounds a connected 
			compact oriented $(R,\epsilon)$-panted subsurface $F$ immersed in $M$.
			\item When $\sigma(L)$ vanishes, every relative
			homology class $\alpha\in H_2(M,L;\ZZ)$	with $\partial\alpha$ equal to 
			the fundamental class $[L]\in H_1(L;\ZZ)$
			is represented by a connected 
			compact oriented $(R,\epsilon)$-panted immersed subsurface $F$ bounded by $L$.
		\end{enumerate}
	\end{theorem}
	
	In fact, the universal bound for $\epsilon$ can be explicitly taken to be $10^{-2}$. 
	Note that the $(R,\epsilon)$-panted immersed subsurface $F$ is not required to be $\pi_1$-injective
	or quasi-Fuchsian.
	
	The invariant $\sigma(L)$ is defined as $\Phi([L]_{R,\epsilon})$,
	where $[L]_{R,\epsilon}$ is the $(R,\epsilon)$-panted cobordism class
	of $L$, so $\sigma(L)$ lies in a canonical submodule of
	$H_1(\SO(M);\ZZ)$ isomorphic to $\ZZ_2$. 
	The proofs of Theorems \ref{main-qfSurface} and \ref{main-pantedSurface} will be completed
	in Section \ref{Sec-boundedQuasiFuchsianSubsurfaces}. A few further questions will be
	proposed in Section \ref{Sec-conclusions}.

	\bigskip\noindent\textbf{Acknowledgement}. The authors thank Hongbin Sun for pointing out
	errors in a previous draft of this paper, and Danny Calegari
	for valuable comments. The authors also thank the anonymous referees 
	for suggestions and corrections.

\section{Methodology}\label{Sec-methodology}
    For a typical construction problem of quasi-Fuchsian subsurfaces in a closed hyperbolic
    $3$-manifold, such as addressed in Theorem \ref{main-qfSurface}, one may generally follow two steps: first,
    decide a suitable finite collection
    of (oriented) nearly geodesic pairs of pants whose cuff lengths are nearly equal; secondly,
    glue these pairs of pants up along boundary in a well controlled fashion to output a connected quasi-Fuchsian
    subsurface. The second step is supposed to be automatic
    once we have fed in the collection of pairs of pants as initial data, so the real task is to provide
    such a collection. Regarding the collection as a finite measure over the set of pairs of pants, we will
    translate the compatibility condition for the gluing into a linear system of
    equations of that measure, and we will introduce properties 
    on solutions to ensure a suitable gluing.
    In other words, we will be interested in certain solutions 
    of the linear system of equations associated to
    a \emph{boundary operator} on measures of pants.
    The purpose of this section is to set up the framework,
    and to divide the discussion into several aspects
    that can be treated separatedly in the rest of this paper.
	
	\subsection{Measures of pants}\label{Subsec-boundaryOfPants}
        Throughout this
        subsection, $M$ will be a closed hyperbolic $3$-manifold. Identifying the universal cover
        $\widehat{M}$ of $M$ as the $3$-dimensional hyperbolic space $\Hyp$, we will regard the deck transformation
        group $\pi_1(M)$ as a torsion-free cocompact discrete subgroup of the group of isometries $\Isom(\Hyp)$.

        \subsubsection{Curves and pants} 
        	Let $S^1$ be the topological circle with a fixed orientation.
        	An \emph{oriented curve} in $M$, or simply a \emph{curve},
			is the free homotopy class of a $\pi_1$-injective immersion
            $\gamma:S^1\looparrowright M$.
            We often abuse the notations for curves and their representatives, and write a curve
            as 
            $$\gamma\looparrowright M.$$ 
            Every curve can be homotoped to a unique oriented closed geodesic in $M$ with the length
            parametrization up to a rotation, so we define the \emph{visual torus} $\vistorus_\gamma$
            of $\gamma$ to be the unit normal vector bundle of the geodesic representative.
            We think of the visual torus to be a holomorphic torus, and the name comes from the fact
            that we may alternatively define $\vistorus_\gamma$ as follows.
            Let $\hat\gamma$ be any elevation of a curve $\gamma$ in $\Hyp$. 
            As $\hat\gamma$ is a quasi-geodesic with endpoints $p,q$ on the sphere at infinity $\hat\CC$,
            we may define $\vistorus_\gamma$ to be the holomorphic cylinder $\hat\CC\setminus\{p,q\}$
            quotiented by the stabilizer $\mathrm{Stab}_{\pi_1(M)}(\hat\gamma)$. The two definitions
            of $\vistorus_\gamma$ are certainly equivalent, but the latter might be more natural from a perspective of geometric
            group theory.
            
            The collection of curves in $M$ will be denoted as $\ocurves(M)$, or simply $\ocurves$.
            The quotient of $\ocurves$ under the free involution induced by orientation reversion
            of curves is the collection of \emph{unoriented curves}
            in $M$, and we will denote it as $|\ocurves|$.
            
        	Let $\Sigma_{0,3}$ be a topological pair of pants, namely,
			a compact three-holed sphere. For convenience, we will fix
			an orientation of $\Sigma_{0,3}$.
			An \emph{unmarked oriented pair of pants} in $M$, or simply a \emph{pair of pants},
			is the homotopy class of a $\pi_1$-injective immersion $\Pi:\Sigma_{0,3}\looparrowright M$,
			up to orientation-preserving self-homeomorphisms of $\Sigma_{0,3}$.
			We often abuse the notations for homotopy classes and their representatives, and write a pair of
			pants as $$\Pi\looparrowright M.$$
			
			The \emph{cuffs} of $\Sigma_{0,3}$ are the three boundary curves of $\Sigma_{0,3}$, and the \emph{seams}
			of $\Sigma_{0,3}$ are three mutually disjoint, properly embedded arcs connecting the three pairs of cuffs,
			which are unique up to orientation-preserving self-homeomorphisms of $\Sigma_{0,3}$.
			Every pair of pants can be homotoped so that the cuffs are the unique geodesic closed curves,
			and that the seams are the unique geodesic arcs orthogonal to the adjacent cuffs, or possibly points
			in the degenerate case.
			We say a pair of pants in $M$
			is \emph{nonsingular} if no seam degenerates to a point under the straightening as above.
			
			The collection of nonsingular pants in $M$ will be denoted as
			$\opants(M)$, or simply $\opants$. The quotient of $\opants$ under 
			the free involution induced by orientation reversion
            of pants is the collection of \emph{unoriented nonsingular pants}
            in $M$, and we will denote it as $|\opants|$.
		
			Suppose $\Pi\looparrowright M$ is a nonsingular pair of pants, straightened so that the cuffs
			are geodesic and the seams are geodesic and orthogonal to the cuffs. For every pair of cuffs $\gamma$ and $\gamma'$,
			the seam $\eta$ from $\gamma$ to $\gamma'$ defines a unit normal vector $v$ at $\gamma$, pointing
			along $\eta$ towards $\gamma'$.
			We call $v\in\vistorus_\gamma$ the \emph{foot} of $\Pi$ at $\gamma$ toward $\gamma'$, and it is the
			`visual direction of the nearest point' as we observe $\gamma'$ from $\gamma$. There are exactly six
			feet of $\Pi$, two at each cuff toward the other two cuffs respectively.

        \subsubsection{Boundary operators}
        	
        	Throughout this paper, a measure is always considered to be nonnegative.
        	Let $\meas(\opants)$ denote all finitely-supported finite measures 
        	on the set of nonsingular pants $\opants$ in $M$.
           	We usually write a nontrivial
        	element of $\meas(\opants)$ as a finite formal sum of elements of $\opants$ with positive
        	coefficients. Similarly, let $\meas(\ocurves)$ denote all finitely-supported 
        	finite measures on the set of curves
        	$\ocurves$ in $M$. There is a natural \emph{boundary operator}
				$$\partial:\,\meas(\opants)\,\to\,\meas(\ocurves),$$
			defined by assigning $\partial\Pi$ to be the sum of the cuffs of $\Pi$.
			We will consider two related notions: the footed boundary $\fboundary$, which is a geometric refinement
			of $\partial$ remembering the feet; and the net boundary $\nboundary$, which is an algebraic reduction of $\partial$
			forgetting the orientation.
        	    	
        	\begin{definition}
        		Let $\meas(\vistorus_\gamma)$ denote all Borel measures on the visual torus of any curve $\gamma$ in $M$,
        		and let $\meas(\vistorus_\ocurves)$ denote the direct sum of $\meas(\vistorus_\gamma)$ as $\gamma$ runs over
       			all curves $\ocurves$.
        		The \emph{footed boundary operator} is the homomorphism:
        			$$\fboundary:\,\meas(\opants)\,\to\,\meas(\vistorus_\ocurves),$$
        		defined by assigning $\fboundary \Pi$ 
        		to be one half of the sum of the atomic measures supported 
        		at the six feet of $\Pi$, where $\Pi\in\opants$ is any
        		nonsingular pants.
        	\end{definition}
        	
        	\begin{remark}
        		The normalization coefficient $\frac12$ has been chosen so that the total measure
        		on $\vistorus_\gamma$ of each cuff $\gamma$ is equal to $1$.
        	\end{remark}
        	
        	\begin{definition}
        		Let $\meas(|\ocurves|)$ denote all finite measures on the set of unoriented curves $|\ocurves|$ in $M$.
        		We identify $\meas(|\ocurves|)$ as the subspace of $\meas(\ocurves)$ fixed under the free involution
        		induced by the orientation reversion $\gamma\mapsto\bar\gamma$, 
        		in other words, regard the atomic measure 
        		supported on the unoriented class $\{\gamma,\bar\gamma\}$ as
        		the measure $\frac12(\gamma+\bar\gamma)$.
        		The \emph{net boundary operator} is the homomorphism:
        			$$\nboundary:\,\meas(\opants)\,\to\,\meas(|\ocurves|),$$
        		defined by
        			$$\nboundary\mu\,=\,\frac12\left|\partial\mu-\overline{\partial\mu}\right|.$$        		
        	\end{definition}
			
			\begin{remark} 
				If we regard $\meas(|\opants|)$ as the subspace of $\meas(\opants)$ fixed under the
				orientation reverion, then $\meas(|\opants|)$ lies in the kernel of $\nboundary$.
			\end{remark}

			We have the following commutative diagram relating various operators:				
			\begin{displaymath}
			\xymatrix{
					&\meas(\vistorus_\ocurves) \ar[d]^{\mathrm{Tot}}\\
				\meas(\opants)	\ar[ru]^{\fboundary} \ar[r]^{\partial} \ar[rd]_{\nboundary}
					&\meas(\ocurves) \ar[d]^{\mathrm{Net}} 
				\\
					&\meas(|\ocurves|)
			}
			\end{displaymath}
			Here $\mathrm{Tot}$ is the componentwise total
			$\mathrm{Tot}(\mu)\,=\,\sum_{\gamma\in\ocurves}\,\mu(\vistorus_\gamma)\,\gamma$, and
			$\mathrm{Net}$ is the unorientation reduction defined by linearly extending 
			$\mathrm{Net}(\gamma)=\frac12|\gamma-\bar\gamma|$ for all $\gamma\in\ocurves$.
%
		\subsubsection{Shape controlling}
			For most of our treatment we will focus on pairs of pants in $M$ that are nearly geodesic 
			with cuffs of nearly equal length, or \emph{nearly regular pants} as we will introduce below.
			
			First recall that
			for a boundary-framed segment in $\Hyp$
			(with the canonical orientation), 
			the (oriented geometric) \emph{complex length} of it can be
			defined as a complex value in
				$$(0,+\infty)\,+\,(-\pi,\pi]\,\imunit.$$
			More precisely, an oriented \emph{$\partial$-framed segment} is an
			oriented geodesic arc with a unit normal vector at each endpoint,
			so the real part of the complex length is the usual length of the geodesic arc,
			and the imaginary part is the signed angle from the initial normal vector
			to the parallel transportation of the terminal normal vector to the
			initial point of the geodesic arc, with respect to the initial tangent vector.
			The complex length does not change if we reverse the orientation
			of the $\partial$-framed segment. 
			It is clear that the complex length of $\partial$-framed segments 
			also makes sense in any oriented hyperbolic $3$-manifold $M$. 
			For a geodesic loop in $M$, we may pick a normal vector at a point, and define
			the \emph{complex length} of the geodesic loop 
			as the complex length of the boundary-framed
			segment obtained from cutting the geodesic loop 
			along the chosen point, and endowing both the
			end-points with the same chosen normal vector.
			The definition is clearly independent of the choices of the point or the normal vector.
			
			Let $M$ be a closed hyperbolic $3$-manifold. Suppose $\Pi\looparrowright M$ be a nonsingular
			pair of pants, straightened by homotopy so that the cuffs and seams are geodesic and orthogonal
			as before. Observe that each cuff $\gamma$ of $\Pi$ is bisected into two boundary-framed segments
			with the boundary framing given by the two feet. In fact,
			these two boundary-framed segments, called \emph{half cuffs},
			have the same complex length. We define the \emph{complex half length}
			of the cuff $\gamma$ of $\Pi$ to be the complex length of either of 
			the half cuffs, denoted as $\chlen_\Pi(\gamma)$.
			We denote the complex length of $\gamma$ as $\clen(\gamma)$.
			When the imaginary part of $\chlen_\Pi(\gamma)$ is at most the right angle,
			$\clen(\gamma)$ is equal to twice $\chlen_\Pi(\gamma)$.
			
			\begin{definition}\label{nearlyRegular}
				Let $M$ be a closed hyperbolic $3$-manifold. Suppose
				$(R,\epsilon)$ is any pair of positive constants.
				\begin{enumerate}
					\item We say that a curve $\gamma\looparrowright M$ is \emph{$(R,\epsilon)$-nearly hyperbolic}, if
						$$\left|\clen(\gamma)-R\right|\,<\,\epsilon.$$
					The subcollection of $(R,\epsilon)$-nearly hyperbolic curves in
					$M$ will be denoted as $\ocurves_{R,\epsilon}\subset\ocurves.$
					\item We say that a nondegenerate pair of pants
					$\Pi\looparrowright M$ is \emph{$(R,\epsilon)$-nearly regular}, if for each cuff $\gamma$ of $\Pi$,
						$$\left|\chlen_\Pi(\gamma)-\frac{R}2\right|\,<\,\frac\epsilon2.$$
					The subcollection of $(R,\epsilon)$-nearly regular pants 
					in $M$ will be denoted as $\opants_{R,\epsilon}\subset\opants$.
				\end{enumerate}
				We often simply say \emph{nearly hyperbolic} or \emph{nearly regular} with the usage
				explained in the following Convention \ref{nearly}.
			\end{definition}
			
			\begin{convention}\label{nearly}
				When ambiguously saying \emph{nearly} instead of \emph{$(R,\epsilon)$-nearly},
				we suppose that $(R,\epsilon)$ are
				understood from the context. Presumably,
				$\epsilon$ will be universally small, and $R$ will be sufficiently large, depending on $M$ and
				$\epsilon$. This precisely means that
				for some universal constant $\hat\epsilon>0$ to be determined, $\epsilon$ is assumed to
				satisfy $0<\epsilon<\hat\epsilon$, and that for any given closed hyperbolic 
				$3$-manifold $M$, and for some constant $\hat{R}=\hat{R}(M,\epsilon)>0$ to be determined, 
				$R$ is assumed to satisfy $R>\hat{R}$.
			\end{convention}
			
			From Definition \ref{nearlyRegular}, it follows that the restriction of the boundary operator yields:
				$$\partial:\,\meas(\opants_{R,\epsilon})\to\meas(\ocurves_{R,\epsilon}),$$
			and similarly for $\partial^\sharp$ and $\partial^\flat$.
		
	\subsection{From pants measures to panted subsurfaces}
	
		For any finite collection of pairs of pants in a closed hyperbolic $3$-manifold
		$M$, we can try to glue them along common cuffs with opposite induced orientations,
		and this will give rise to a panted subsurface in $M$, precisely as follows.
						
		\begin{definition}\label{pantedSubsurfaceDefinition}
			An \emph{$(R,\epsilon)$-nearly regularly panted subsurface}
			of $M$, or simply an \emph{$(R,\epsilon)$-panted subsurface},
			is a (possibly disconnected) compact oriented surface
			$F$ with a pants decomposition, 
			and with an immersion $j:F\looparrowright M$ into $M$
			such that the restriction of $j$ to each component 
			pair of pants is $(R,\epsilon)$-nearly regular.
			Let $\mu\in\meas(\opants_{R,\epsilon})$ be the integral 
			measure	such that for each $\Pi\in\opants_{R,\epsilon}$, there are 
			exactly $\mu(\{\Pi\})$ copies of $\Pi$ in
			all component pairs of pants of $F$ immersed via $j$.
			Then we say that the $(R,\epsilon)$-panted
			subsurface is \emph{subordinate to} $\mu$.
		\end{definition}

		In general, the panted subsurface 
		would be neither $\pi_1$-injective quasi-Fuchsian nor connected.
		However, we wish to introduce conditions on $\mu$ 
		to ensure that some quasi-Fuchsian connected panted subsurface
		therefore exists and is subordinate to $\mu$.
		
		Recall that for a metric space $(X,d)$, and
		for a positive number $\delta$, two Borel measures 
		$\mu,\mu'$ are said to be \emph{$\delta$-equivalent}, 
		if $\mu(X)$ equals $\mu'(X)$, and if for every Borel subset $A$ of $X$, $\mu(A)\leq\mu'(\mathcal{N}_\delta(A))$, 
		where $\mathcal{N}_\delta(A)\subset X$ is the $\delta$-neighborhood of $A$. 
		Note that $\delta$-equivalence is a symmetric relation.
		For any nonvanishing finite Borel measure $\mu$ on $X$, we may speak of $\delta$-equivalence
		after normalization, namely, after dividing $\mu$ by $\mu(X)$.

		\begin{definition}\label{gluingConditions}
			Let $M$ be a closed hyperbolic $3$-manifold. 
			Let $\mu\in\meas(\opants_{R,\epsilon})$ be a measure of nearly regular pants.
			\begin{enumerate}
				\item We say that $\mu$ is \emph{ubiquitous},
				if $\mu$ is positive at every $\Pi\in\opants_{R,\epsilon}$, and if
				$\partial\mu$ is positive at every $\gamma\in\ocurves_{R,\epsilon}$.
				\item We say that $\mu$ is \emph{irreducible}, 
				if for any nontrivial decomposition	$\mu=\mu'+\mu''$,
				$\mu'$ and $\mu''$ have adjacent supports, namely, 
				that there is a curve $\gamma\in\ocurves_{R,\epsilon}$ 
				which lies in the support of $\partial\mu'$ and 
				the orientation-reversal of which lies in the support of
				$\partial\mu''$.
				\item We say that $\mu$ is \emph{$(R,\epsilon)$-nearly evenly footed},
				if for every curve $\gamma\in\ocurves_{R,\epsilon}$ on which
				$\partial\mu$ is nonvanishing,
				the normalization of $(\fboundary\mu)|_{\vistorus_\gamma}$
				is $(\frac\epsilon{R})$-equivalent to the normalization of the Lebesgue measure,
				with respect to the Euclidean metric
				on the visual torus $\vistorus_\gamma$
				induced from the unit normal vector bundle of the geodesic representative of $\gamma$. 
				We often simply say \emph{nearly evenly footed}
				following Convention \ref{nearly}.
				\item We say that $\mu$ is \emph{rich}, if the net boundary
				of $\mu$ at any unoriented curve is 
				a relatively small portion compared to the cancelled
				part, or specifically for our application, 
				that $\nboundary\mu(|\gamma|)\,\leq\,\frac15\,\partial\mu(\{\gamma,\bar{\gamma}\})$.
				Here $|\gamma|$ means the unoriented class $\{\gamma,\bar\gamma\}$ for any curve 
				$\gamma\subset\ocurves_{R,\epsilon}$. 
			\end{enumerate}
		\end{definition} 
		
		The following criterion about connected quasi-Fuchsian gluing will
		be proved in Section \ref{Sec-quasiFuchsianConnectedGluing}.
		
		\begin{theorem}\label{gluing}
			Let $M$ be a closed hyperbolic $3$-manifold. 
			For any universally small positive $\epsilon$
			and any sufficiently large positive $R$ depending on $M$ and $\epsilon$, 
			the following statement holds.
			For any nontrivial rational measure
			$\mu\in\meas(\opants_{R,\epsilon})$, if $\mu$ is irreducible,
			$(R,\epsilon)$-nearly evenly footed, and rich, then 
			there exists an oriented, connected, compact,
			$\pi_1$-injectively immersed quasi-Fuchsian subsurface: 
				$$j:\,F\looparrowright M,$$
			which is $(R,\epsilon)$-nearly regularly
			panted subordinate to a positive integral multiple of $\mu$.
		\end{theorem}
							
	\subsection{Homology via pants measures}\label{Subsec-homologyViaPants}
		For a closed hyperbolic $3$-manifold $M$, 
		we wish to understand the structure of the boundary operator 
		$\partial:\meas(\opants_{R,\epsilon})\to\meas(\ocurves_{R,\epsilon})$.
		More specifically, the following Theorem \ref{homologyViaPants} should be
		viewed from this perspective. 
		
		Let $\mathcal{L}\subset\ocurves_{R,\epsilon}$ be a collection
		of distinct curves, invariant under orientation reversion. We write $|\mathcal{L}|\subset|\ocurves_{R,\epsilon}|$
		for the corresponding unoriented curves, namely, the quotient of $\mathcal{L}$ by orientation reversion.		
		Let 
			$$\zmeas(\opants_{R,\epsilon},|\mathcal{L}|)$$
		denote the subset of $\meas(\opants_{R,\epsilon})$ consisting of measures $\mu$
		with the net boundary $\nboundary\mu$ supported on (possibly a proper subset of)
		the unoriented curves $|\mathcal{L}|$.
		Choosing a collection of mutually disjoint, embedded unoriented loops 
		$k_1,\cdots,k_r$ representing elements
		of $|\mathcal{L}|$, we write $H_2(M,|\mathcal{L}|;\RR)$ for $H_2(M,k_1\cup\cdots \cup k_r;\RR)$.
		Note that $H_2(M,|\mathcal{L}|;\RR)$ is well defined up to natural isomorphisms for different choices
		of the loops $k_i$. Thus there is a natural homomorphism between semimodules over 
		the semiring of nonnegative real numbers:
			$$[\cdot]\,:\,\zmeas(\opants_{R,\epsilon},|\mathcal{L}|)\to H_2(M,|\mathcal{L}|;\,\RR).$$
		
		In particular, when $\mathcal{L}$ is empty, we denote the kernel of the homomorphism $[\cdot]$ above
		as
			$$\bmeas(\opants_{R,\epsilon})\,\subset\,\zmeas(\opants_{R,\epsilon},\emptyset).$$
		Nevertheless, $\bmeas(\opants_{R,\epsilon})$ is naturally contained in 
		$\zmeas(\opants_{R,\epsilon},|\mathcal{L}|)$ for any $\mathcal{L}$
		as well.

		\begin{theorem}\label{homologyViaPants}
			Let $M$ be a closed hyperbolic $3$-manifold. For any sufficiently small positive $\epsilon$ depending on $M$,
			and any sufficiently large positive $R$ depending on $M$ and $\epsilon$, the following
			statements hold. Suppose $\mathcal{L}\subset\ocurves_{R,\epsilon}$
			is a collection of distinct curves invariant under orientation reversion.
			\begin{enumerate}
				\item There is a short exact sequence of semimodules over 
				the semiring of nonnegative real numbers:
				$$0\longrightarrow\bmeas(\opants_{R,\epsilon})\longrightarrow
				\zmeas(\opants_{R,\epsilon},|\mathcal{L}|)\longrightarrow
				H_2(M,|\mathcal{L}|;\,\RR)\longrightarrow0.$$
				\item There exists a nontrivial measure 
				$\mu_0\in \bmeas(\opants_{R,\epsilon})$ which is ubiquitous,
				irreducible, $(R,\epsilon)$-nearly evenly footed, and rich.
				Moreover,
				every measure in $\zmeas(\opants_{R,\epsilon},|\mathcal{L}|)$ can be adjusted to satisfy
				the same properties, by adding some measure in $\bmeas(\opants_{R,\epsilon})$.
			\end{enumerate}
			Furthermore, the same statements hold for rational coefficients instead of real coefficients as well.
		\end{theorem}
		
		In fact, the constant $\epsilon$ is required to be bounded by the injectivity radius of $M$
		so as to guarantee the existence of $\mu_0$.
		Theorem \ref{homologyViaPants} will be proved in Section \ref{Sec-boundedQuasiFuchsianSubsurfaces}.
		As Theorem \ref{homologyViaPants} feeds Theorem \ref{gluing} with workable input, homologically
		interesting connected quasi-Fuchsian subsurfaces can be produced in closed hyperbolic $3$-manifolds
		under fairly general conditions.

\section{Quasi-Fuchsian connected gluing}\label{Sec-quasiFuchsianConnectedGluing}
	In this section, we prove Theorem \ref{gluing},
	restated as Proposition \ref{gluingProp} in terms of gluing.
	Let $M$ be a closed hyperbolic
	$3$-manifold, and let $(R,\epsilon)$ be a pair of
	positive constants.
	
	Given a panted surface $F$ of which the pants structure is given
	by a union of disjoint simple closed curves $C\subset F$, we may cut $F$ along $C$
	to obtain a disconnected surface $\mathcal{F}$ whose components 
	are all pairs of pants. Denote the union
	of all the new boundary components of $\mathcal{F}$
	coming from the cutting as $\mathcal{C}\subset \partial\mathcal{F}$. Then the panted surface $F$ can 
	be recovered as the quotient of $\mathcal{F}$ by an 
	orientation-reversing involution $\phi:\mathcal{C}\to \mathcal{C}$,
	which sends any preimage component of $C$ 
	to its opposite boundary component. With this
	in mind, we introduce the notion of gluing as follows.
	
	\begin{definition}
		For any integral measure $\mu\in\meas(\opants_{R,\epsilon})$,
		let $\mathcal{F}$ be the finite disjoint union of
		copies of nearly regular pants prescribed by $\mu$, namely, such that for
		any $\Pi\in\opants_{R,\epsilon}$, there are exactly $\mu(\{\Pi\})$ copies
		of $\Pi$ in $\mathcal{F}$. 
		By a \emph{gluing} of $\mathcal{F}$, we mean a pair $(\mathcal{C},\phi)$, such that
			$$\mathcal{C}\subset\partial\mathcal{F}$$
		is a subunion of cuffs, and that
			$$\phi:\,\mathcal{C}\to\mathcal{C}$$
		is a free involution which sends each cuff $c\subset\mathcal{C}$
		to its orientation-reversal, regarded as 
		in $\ocurves_{R,\epsilon}$. We say that $(\mathcal{C},\phi)$ is \emph{maximal}
		if $\phi$ cannot be extended to any subunion of 
		cuffs $\mathcal{C}'\subset\partial\mathcal{F}$
		larger than $\mathcal{C}$. Since 
		the quotient of $\mathcal{F}$ by $\phi$ yields
		a compact oriented $(R,\epsilon)$-panted subsurface
		$j:\,F\looparrowright M$,
		the quotient image of any cuff $c\subset\mathcal{C}$ in $F$
		will be called a \emph{glued cuff}.
	\end{definition}
		
	\begin{proposition}\label{gluingProp}
		Let $M$ be a closed hyperbolic $3$-manifold.
		For any universally small positive $\epsilon$
		and any sufficiently large
		positive $R$ depending on $M$ and $\epsilon$, the following holds.
		If a rational nontrivial pants measure
		$\mu\in\meas(\opants_{R,\epsilon})$ is irreducible,
		$(R,\epsilon)$-nearly evenly footed, and rich,
		then possibly after passing to a positive integral multiple
		of $\mu$, the prescribed oriented compact surface
		$\mathcal{F}$ admits a maximal gluing $(\mathcal{C},\phi)$, which
		yields a $\pi_1$-injectively immersed, quasi-Fuchsian, 
		and connected subsurface $j:F\looparrowright M$.
	\end{proposition}
	
	The key technique to ensure the connectedness of the resulting surface
	is a trick called \emph{hybriding}.
	To illustrate the idea, the reader may assume for simplicity that $\nboundary\mu$
	is zero, so that any panted surface $F$	resulted from a maximal gluing is closed.
	We say that a gluing is \emph{nearly unit shearing}, if 
	for any glued cuff $c$ on the resulting $(R,\epsilon)$-panted surface $F$,
	the feet of the pair of pants on one side of $c$ is almost
	exactly opposite to the feet of the pair of pants on the other side of
	$c$ after a parallel transportation along $c$ of distance $1$
	(Definition \ref{nearlyUnitShearing}).
	By the construction of \cite{KM-surfaceSubgroup},
	the assumption that $\mu$ is $(R,\epsilon)$-nearly evenly
	footed implies that such a maximal gluing always exists,
	resulting in a $\pi_1$-injectively immersed surface 
	which is quasi-Fuchsian. Since $F$ might be disconnected,
	we wish to slightly modify the gluing without affecting the nearly
	unit shearing property, nevertheless the number of components
	of $F$ can be decreased in that case.
	Denote the components of $F$ as $F_1,\cdots,F_r$, where $r$ is
	at least two. If two components of $F$, say $F_1$ and $F_2$,
	has glued cuffs $c_1\subset F_1$ and $c_2\subset F_2$ that 
	are homotopic to each other, supposing that $c_i$ is nonseparating on $F_i$,
	we may modify the gluing by cutting $F_i$ along $c_i$, and regluing
	in a cross fashion. Then the new resulting surface $F'$ has a connected
	component $F_{12}$ instead of the previous two components $F_1$ and $F_2$.
	We say that $F_{12}$ is obtained by \emph{hybriding} $F_1$
	and $F_2$ along $c_1$ and $c_2$. 
	To preserve the nearly unit shearing property,
	we need to require that the feet of pants on one side of 
	$c_1$ is almost the same as the feet of pants on the same side of
	$c_2$. 
	Such $F_i$ and $c_i$ can be found by the following argument.
	First, the irreducibility of $\mu$ implies that there is some 
	curve class $\gamma\in\ocurves_{R,\epsilon}$, such that
	there are at least two distinct components of $F$ that have glued cuffs 
	homotopic to $\gamma$.
	Because the footed boundary $\partial^{\sharp}\mu$ restricted to
	the unit normal vector bundle $\vistorus_\gamma$ 
	over $\gamma$ is nearly evenly distributed,
	the connectedness of $\vistorus_\gamma$ implies that at least
	distinct two components $F_1$ and $F_2$ of $F$ (not necessarily the components
	that we started with) have glued cuffs $c_1$ and $c_2$ homotopic to $\gamma$
	with their feet on the same side 
	very close to those of each other. Therefore, performing the hybriding
	on these $F_i$ along $c_i$ will decrease the number of 
	components of the resulting surface, preserving the nearly unit shearing
	property. Iterating the process until the resulting surface become
	connected, then we are done. A minor point here
	is that the hybriding trick
	also require that $c_i$ be nonseparating on $F_i$.
	In fact, with the somewhat technical assumption
	that $\mu$ is rich, 
	the nonseparating property of glued cuffs 
	can be satisfied if we pass to a cyclic finite cover
	of $F$, and hence pass to positive multiple
	of $\mu$. Note also that the arguments above certainly
	works as well when $F$ has boundary.
	In practice, one need to be slightly careful to control the error
	so that the new resulting surface remains $(R,\epsilon)$-panted,
	but the general idea of hybriding follows the outline above.
	
	Roughly speaking, the assumption that $\mu$ is nearly evenly footed allows us to control the shape
	of $F$ along each glued cuff, 
	which ensures the $\pi_1$-injectivity and the quasi-Fuchsian property;
	the assumption that $\mu$ is rich allows us to construct $F$ so that glued
	cuff is nonseparating in $F$, so combined with the assumption that $\mu$ is 
	irreducible, we may perform a hybriding trick to obtain a connected $F$, 
	possibly after passing to a further positive integral multiple of $\mu$. 

	In the rest of this section, we prove Proposition \ref{gluingProp}.
	In Subsection \ref{Subsec-nonseparatingGluedCuffs} we explain how to control
	the gluing so that the glued cuffs are nonseparating; in Subsection \ref{Subsec-gluingWithNearlyUnitShearing},
	we review the nearly unit shearing condition that is used in \cite{KM-surfaceSubgroup};
	Subsection \ref{Subsec-hybridingDisconnectedComponents} is the hybriding argument;
	Subsection \ref{Subsec-proofOfPropositionGluingProp} summarizes
	the proof of Proposition \ref{gluingProp}.
	
	It will be convenient to introduce a measure
			$$\nu^\sharp_{\mathcal{C}}\,\in\,\meas(\vistorus_{\ocurves_{R,\epsilon}}),$$
	naturally
	associated to any subunion of cuffs $\mathcal{C}\subset\partial\mathcal{F}$. 
	This measure records the contribution
	to the footed boundary $\fboundary\mu$ from those
	pairs of pants which contain components of $\mathcal{C}$.
	More concretely, each component $c\subset \mathcal{C}$
	lies in a unique pair of pants $P_c\subset\mathcal{F}$.
	If $c\subset\mathcal{C}$ is a copy of $\gamma\in \ocurves_{R,\epsilon}$, 
	and if $P_c$ is a copy of $\Pi\in \opants_{R,\epsilon}$, we define 
	the \emph{marked footed boundary} $\partial^\sharp_c(P_c)\in\meas(\vistorus_\gamma)$
	to be sum of the two feet (as atomic measures) of $\Pi$ at the cuff corresponding to
	$c\subset \partial P_c$. Note that potentially $P_c$ could have other cuffs
	which are copies of $\gamma$ but which might not come from $\mathcal{C}$, so we need to
	specify $c$ rather than just mentioning $\gamma$. 
	We define
		$$\nu^\sharp_{\mathcal{C}}\,=\,\sum_{c\subset\mathcal{C}}\,\partial^\sharp_c(P_c).$$
	In particular, we also write
		$$\nu^\sharp_{c}\,=\,\partial^\sharp_c(P_c).$$
			
	\subsection{Nonseparating glued cuffs}\label{Subsec-nonseparatingGluedCuffs}
		The lemma below essentially follows from the condition that $\mu$ is rich.
				
		\begin{lemma}\label{gluedCuffs}
			Possibly after passing to a positive integral
			multiple of $\mu$, we may assume that the prescribed disjoint union of pants $\mathcal{F}$ admits
			a subunion of cuffs $\mathcal{C}\subset\partial\mathcal{F}$, satisfying the following:
			\begin{itemize}
				\item Any gluing $(\mathcal{C},\phi)$ of $\mathcal{F}$ along $\mathcal{C}$ is maximal;
				\item Restricted to any $\vistorus_\gamma$, the measure $\nu^\sharp_\mathcal{C}$ is a positive rational 
				multiple of $\fboundary\mu$; 
				\item Any pair of pants $P\subset\mathcal{F}$ contains at least two cuffs from $\mathcal{C}$.
			\end{itemize}
		\end{lemma}
		
		\begin{proof}
			For simplicity, we write $m_\gamma$ for $(\partial\mu)(\{\gamma\})$, and $n_\Pi$ for
			$\mu(\{\Pi\})$. Let $k_{\gamma,\Pi}\in\{0,1,2,3\}$ be the number of times
			that a curve $\gamma$ occurs as the cuff of a pair of pants $\Pi$.
			For any curve $\gamma\in\ocurves_{R,\epsilon}$,
			let 
				$$\mu_\gamma\,=\,\sum_{\Pi\in\opants_{R,\epsilon}}\,n_\Pi\,k_{\gamma,\Pi}\cdot\Pi.$$
			Let
				$$\tilde{\mu}_\gamma\,=\,\frac{m_\gamma-m_{\bar\gamma}}{m_\gamma}\cdot\mu_\gamma,$$
			if $m_\gamma\,>\,m_{\bar\gamma}$; otherwise, let $\tilde{\mu}_\gamma\,=\,0$.
			Since $\mu$ is rich, it is clear by Definition \ref{gluingConditions}
			that
				$$\frac{m_\gamma-m_{\bar\gamma}}{m_\gamma}\leq\frac13,$$
			whenever $m_\gamma\,>\,m_{\bar\gamma}$. It follows that
				$$\sum_{\gamma\in\ocurves_{R,\epsilon}}\,\tilde{\mu}_\gamma\,\leq\,\mu,$$
			because every pair of pants has only three cuffs. Furthermore, possibly after passing to
			a positive integral multiple of $\mu$, we may assume that 
			$\mu$ and all $\tilde{\mu}_\gamma$ are
			integral. Therefore, in the disjoint union of pairs of pants $\mathcal{F}$ 
			prescribed by $\mu$, we may find mutually disjoint subunions $\mathcal{F}_\gamma$ prescribed
			by $\tilde{\mu}_\gamma$, and for each component $P\subset\mathcal{F}_\gamma$, we may mark one cuff
			$c\subset P$ which is a copy of $\gamma$. When $P$ has more than one cuffs homotopic to $\gamma$,
			we require the marking to be evenly weighted on these cuffs: In other words,
			if $P$ has only two cuffs $c_1$ and $c_2$ homotopic to $\gamma$, 
			among all the pairs of pants in $\mathcal{F}_\gamma$ that are copies of $P$,
			there will be half of them having the copy cuff $c_1$ marked, and the other half having the copy cuff 
			$c_2$ marked; similarly will we mark those pairs of pants $P$ with all the cuffs homotopic to $\gamma$.
			Let
				$$\mathcal{C}\subset\partial\mathcal{F}$$
			be the union of all the unmarked cuffs. 
			
			It is straightforward to check that 
			the three listed properties about $\mathcal{C}$ are satisfied by our construction.
			In fact, for any $\gamma,\bar{\gamma}\in\ocurves_{R,\epsilon}$, 
			suppose without loss of generality that $m_\gamma\geq m_{\bar{\gamma}}$. Because
			$\mathcal{C}$ has exactly $m_{\bar{\gamma}}$ components homotopic to $\gamma$ 
			and exactly $m_{\bar{\gamma}}$ components
			homotopic to $\bar{\gamma}$, any gluing $(\mathcal{C},\phi)$ is maximal.
			The measure $\nu^\sharp_{\mathcal{C}}$ 
			restricted to $\vistorus_\gamma$ equals $\frac{m_{\bar\gamma}}{m_\gamma}$
			times $\partial^\sharp\mu$ and restricted to $\vistorus_{\bar{\gamma}}$
			equals $\partial^\sharp\mu$,
			both proportional to $\partial^\sharp\mu$. 
			For any $P\in\mathcal{F}$, we marked at most one cuff $c\subset P$ in the construction above,
			so it contains at least two cuffs from $\mathcal{C}$.		
		\end{proof}
		
		Suppose $(\mathcal{C},\phi)$ is a gluing of $\mathcal{F}$ prescribed by $\mu$. 
		For a disjoint union of pants $\tilde{\mathcal{F}}$ prescribed by a
		positive integral multiple of $\mu$,
		we say that a gluing $(\tilde{\mathcal{C}},\tilde{\phi})$ \emph{covers} $(\mathcal{C},\phi)$,
		if $\tilde{\mathcal{C}}$ is the preimage of $\mathcal{C}$ under the natural covering
		$\kappa:\tilde{\mathcal{F}}\to\mathcal{F}$, and if the following diagram commutes:
		\begin{displaymath}
			\xymatrix{
				\tilde{\mathcal{C}}	\ar[r]^{\tilde{\phi}} \ar[d]^{\kappa} &\tilde{\mathcal{C}}\ar[d]^{\kappa}\\
				\mathcal{C} \ar[r]^{\phi}	&\mathcal{C}
			}
		\end{displaymath}
		In this case, the associated surface $\tilde{F}$ naturally covers $F$ as well.
		
		\begin{lemma}\label{nonseparatingGluedCuffs}
			Let $\mathcal{C}\subset\mathcal{F}$ be a subunion of cuffs satisfying the conclusion of Lemma
			\ref{gluedCuffs}. Suppose that $(\mathcal{C},\phi)$ is a gluing of 
			$\mathcal{F}$ prescribed by $\mu$. Then
			the disjoint union of pants $\tilde{\mathcal{F}}$ 
			prescribed by $2\mu$ admits a gluing $(\tilde{\mathcal{C}},\tilde{\phi})$
			covering $(\mathcal{C},\phi)$, such that every glued cuff in the resulted surface
			$\tilde{F}$ is nonseparating.
		\end{lemma}
		
		\begin{proof}
			The glued cuffs induces a decomposition
			of $F$ into pairs of pants, and let $\Lambda$ be the (possibly disconnected)
			dual graph. By Lemma \ref{gluedCuffs}, the valence of any vertex of
			$\Lambda$ is at least two. It follows from an easy construction
			that $\Lambda$ admits a double cover
			$\tilde{\Lambda}$ in which every edge is nonseparating. 
			In other words, there is a double cover $\tilde{F}$ of $F$ induced
			from a cover $(\tilde{\mathcal{C}},\tilde{\phi})$
			of the gluing $(\mathcal{C},\phi)$, in which every glued cuff is nonseparating.
			Note also that $\tilde{\mathcal{F}}$ can be identified with $\tilde{F}$ cut along the
			glued cuffs, so it is prescribed by $2\mu$.
		\end{proof}
			
	\subsection{Gluing with nearly unit shearing}\label{Subsec-gluingWithNearlyUnitShearing}
		To control the shape of $F$ along a glued cuff, we will require the gluing $\phi$
		to be \emph{nearly unit shearing}, which can be described on the visual tori as follows.
		
		Observe that for any nearly purely hyperbolic curve $\gamma\in\ocurves_{R,\epsilon}$ 
		(or indeed for any curve), there is a natural action of 
		the additive group of complex numbers on $\vistorus_\gamma$. More precisely, for any $\zeta\in\CC$,
		there is an isomorphism between holomorphic tori:
			$$A_\zeta:\,\vistorus_\gamma\to\vistorus_\gamma,$$
		satisfying that for any $r\in\RR$, 
		$A_r$ parallel transports any normal vector along $\gamma$ by signed distance $r$,
		and that for any $\theta\in\RR$, $A_{\theta\imunit}$ rotates the direction
		of any normal vector by a signed angle $\theta$.
		It is clear that the kernel of the action is the lattice in $\CC$ generated by
		$2\pi\imunit$ and the complex length $\clen(\gamma)$ of $\gamma$.
		There is also a canonical anti-isomorphism
			$$\bar{}\,:\, \vistorus_\gamma\to\vistorus_{\bar\gamma},$$
		taking any unit normal vector $(p,v)$
		to the opposite vector $\overline{(p,v)}=(p,-v)$ at the same point $p\in |\gamma|$, where $|\gamma|$ 
		is regarded as the unoriented geodesic representative.
		Note that an anti-isomorphism
		is orientation-reversing.
		We will think of the composition of the bar anti-isomorphism with $A_\zeta$ 
		as the model of a $\zeta$-shearing gluing along $\gamma$,
		denoted as:
			$$\overline{A_\zeta}:\,\vistorus_\gamma\to\vistorus_{\bar\gamma}.$$
		In other words, a 
		nearly unit shearing gluing $(\mathcal{C},\phi)$ should 
		behave very much like $\overline{A_1}$ 
		along each glued cuff.
%
		\begin{definition}\label{nearlyUnitShearing}
			A gluing $(\mathcal{C},\phi)$ of $\mathcal{F}$ 
			is said to be \emph{$(R,\epsilon)$-nearly unit shearing}, if for every pair of cuffs $c,c'\subset\mathcal{C}$ 
			with $c'$ equal to $\phi(c)$, the feet measure $\nu^\sharp_{c'}$ is $(\frac\epsilon{R})$-equivalent
			to $(\overline{{A}_1})_*(\nu^\sharp_c)$ on $\vistorus_{\bar\gamma}$,
			where $\gamma\in\ocurves_{R,\epsilon}$ is the curve class of $c$.
		\end{definition}
		
		\begin{remark}
			A reader familiar with the Kahn--Markovic construction 
			should recognize the definition above as
			equivalent to the condition 
				$$|s(c)-1|\,<\,\frac{\epsilon}R$$
			in \cite[Theorem 2.1]{KM-surfaceSubgroup}.
		\end{remark}
		
		The lemma below is a consequence of the condition
		that $\mu\in\meas(\opants_{R,\epsilon})$ is $(R,\epsilon)$-nearly evenly footed.
		
		\begin{lemma}\label{nearlyUnitShearingGluing}
			With the notations of Proposition \ref{gluingProp}, there is
			a gluing $(\mathcal{C},\phi)$ of $\mathcal{F}$, which is $(R,\epsilon)$-nearly unit shearing.
		\end{lemma}
		
		\begin{proof}
			Let $\mathcal{C}\subset\partial\mathcal{F}$ be a subunion of cuffs as ensured by the conclusion
			of Lemma \ref{gluedCuffs}. The lemma follows from the Hall Marriage argument, cf.~
			\cite[Theorem 3.2 and Subsection 3.5]{KM-surfaceSubgroup}.
		\end{proof}
		
		\begin{lemma}\label{resultingQF}
			There exists $\hat\epsilon>0$ depending on $M$, and for any $0<\epsilon<\hat\epsilon$, there exists
			$\hat{R}>0$ depending on $M$ and $\epsilon$, such that for any $R>\hat{R}$,
			the following holds.
			If a gluing $(\mathcal{C},\phi)$ of $\mathcal{F}$ is 
			$(R,\epsilon)$-nearly unit shearing, then the resulting surface
			$j:F\looparrowright M$ is $\pi_1$-injectively immersed and quasi-Fuchsian.
		\end{lemma}
		
		\begin{proof}
			This is exactly \cite[Theorem 2.1]{KM-surfaceSubgroup} if $F$ is closed.
			In the general case, recall that a compact immersed subsurface $F$ of 
			$M$ is \emph{quasi-Fuchsian} in our sense if it is an 
			essential subsurface of a closed immersed quasi-Fuchsian subsurface $F'$ of $M$
			(Section \ref{Sec-introduction}).
			We may take a ubiquitous $(R,\epsilon)$-nearly evenly footed
			measure $\mu_0\in\meas(\opants_{R,\epsilon})$, for instance, as
			guaranteed by Theorem \ref{homologyViaPants} (2). Then for
			a sufficiently large integer $N$, we may assume that $N\mu_0-\mu$
			is still ubiquitous and $(R,\epsilon)$-nearly evenly footed.
			Let $\mathcal{F}'$ be the disjoint union of $(R,\epsilon)$-nearly regular
			pants prescribed by $N\mu_0$. We may identify $\mathcal{F}$ as a subunion
			of components of $\mathcal{F}'$. It is not hard to see that
			the gluing $(\mathcal{C},\phi)$ can be extended to be a gluing 
			$(\partial\mathcal{F}',\phi')$ which is still $(R,\epsilon)$-nearly unit shearing,
			provided $N$ sufficiently large.
			Then the gluing $(\partial\mathcal{F}',\phi')$ yields a possibly disconnected,
			componentwise $\pi_1$-injectively immersed quasi-Fuchsian 
			closed subsurface $F'\looparrowright M$ by \cite[Theorem 2.1]{KM-surfaceSubgroup}.
			The subsurface $F$ obtained via the gluing $(\mathcal{C},\phi)$ of $\mathcal{F}$
			is an essential subsurface of $F'$, so it is $\pi_1$-injectively immersed
			and quasi-Fuchsian.
		\end{proof}

	\subsection{Hybriding disconnected components}\label{Subsec-hybridingDisconnectedComponents}
	The following lemma
	uses the condition that $\mu$ is irreducible.
	
	\begin{lemma}\label{hybriding}
		Suppose that $(\mathcal{C},\phi)$ is a gluing of 
		$\mathcal{F}$ prescribed by a positive integral multiple of $\mu$ which is
		$(R,\epsilon)$-nearly unit shearing with all glued cuffs nonseparating
		on the resulting surface $F$. Then possibly after
		passing to a further positive multiple of $\mu$,
		there is a gluing $(\mathcal{C},\phi')$,
		which is $(R,2\epsilon)$-nearly unit shearing with all glued cuffs nonseparating
		on the resulting surface $F'$, and moreover,
		$F'$ is connected.
	\end{lemma}
	
	\begin{proof}
		For simplicity, we rewrite the positive integral multiple
		of $\mu$ prescribing $\mathcal{F}$ as $\mu$. Note that
		we may still assume $\mu$ to be 
		irreducible, $(R,\epsilon)$-nearly evenly footed, and rich.
		We also observe that if $F$ has $r$ components $F_1,\cdots,F_r$,
		then for any positive integral multiple $m\mu$,
		there is a gluing $(\tilde{\mathcal{C}},\tilde{\phi})$
		covering $(\mathcal{C},\phi)$
		such that the resulting surface $\tilde{F}$ 
		is an $m$-fold cover of $F$ with $r$ components as well.
		Indeed, each component $\tilde{F}_i$ of $\tilde{F}$ can be
		chosen as the $m$-fold cyclic cover of $F_i$ dual to
		a glued cuff $c_i\subset F_i$, and $\tilde{F}$ has
		an induced pants decomposition that describes 
		$(\tilde{\mathcal{C}},\tilde{\phi})$.
		Moreover, $(\tilde{\mathcal{C}},\tilde{\phi})$
		is clearly $(R,\epsilon)$-nearly unit shearing with all 
		glued cuffs nonseparating as well.
		Therefore, possibly after passing to a positive
		integral multiple of $\mu$, and
		considering the gluing
		$(\tilde{\mathcal{C}},\tilde{\phi})$
		instead of $(\mathcal{C},\phi)$,	
		we may further assume that once any component $F_i$ has 
		a glued cuff homotopic to a curve $\gamma\in\ocurves_{R,\epsilon}$,
		there are at least $r$ glued cuffs of $F_i$ homotopic to $\gamma$.
				
		Let 
			$$F\,=\,F_1\sqcup\cdots\sqcup F_r$$
		be the decomposition of $F$ into connected components.
		Then there is an induced decomposition
		of $\mathcal{F}$ into subunion of pairs of pants
		$\mathcal{F}_1,\cdots,\mathcal{F}_r$, such that components
		of each $\mathcal{F}_i$ is projected to be 
		$F_i$ under the gluing.
		It follows that $\mu$ equals $\mu_1+\cdots+\mu_r$,
		where the measure $\mu_i$ prescribes $\mathcal{F}_i$.
		Similarly, there is an induced decomposition
		of 
			$$\mathcal{C}\,=\,\mathcal{C}_1\sqcup\cdots\sqcup\mathcal{C}_r$$
		such that each $\mathcal{C}_i$ is the subunion of
		cuffs of $\mathcal{F}_i$, which is invariant under $\phi$.
		It follows that 
			$$\nu^{\sharp}_{\mathcal{C}}\,=\,
			\nu^{\sharp}_{\mathcal{C}_1}+\cdots+\nu^{\sharp}_{\mathcal{C}_r}.$$
		
		Consider a simplicial graph $X$ as follows
		of $r$ vertices $v_1,\cdots,v_r$.
		For any $1\leq i<j\leq r$,
		the vertices $v_i$ and $v_j$ are connected by
		an edge if and only if there is a pair of
		cuffs $c\subset \mathcal{C}_i$ and $c'\subset\mathcal{C}_j$
		representing the same curve class $\gamma\in\ocurves_{R,\epsilon}$,
		such that $\nu^{\sharp}_{\phi(c)}$ and $\nu^{\sharp}_{\phi(c')}$ are
		$(\frac{\epsilon}R)$-equivalent on $\vistorus_{\bar{\gamma}}$.
		We hence fix a choice of $c,c'$ as above, rewriting as
		$c_{ij},c'_{ij}$. 
		With the assumption that any glued cuff of $F_i$ or $F_j$ has
		at least $r$ copies on the same component, we may
		assume that $c_{ij}$ are mutually distinct
		components of $\mathcal{C}$, and similarly for
		$c'_{ij}$.	
		
		Observe that $X$ is connected. In fact, let $X_1,\cdots,X_s$
		be the components of $X$, and let $I_k\subset\{1,\cdots,r\}$
		be the subset of indices so that $i\in I_k$ if and only if
		$v_i\in X_k$. Suppose on the contrary that $s>1$. 
		We write the $(\frac{\epsilon}R)$-neighborhood
		of the support of $\nu^\sharp_{\mathcal{C}_i}$ as $U_i$,
		and write the union of $U_1,\cdots,U_r$ as $U$.
		Because $\mu$ is $(R,\epsilon)$-nearly evenly distributed and rich
		(Definition \ref{gluingConditions}),
		$U\cap\vistorus_\gamma$ is either
		the emptyset or	$\vistorus_\gamma$, for any
		curve $\gamma\in\ocurves_{R,\epsilon}$.
		If $U\cap\vistorus_\gamma$ equals $\vistorus_\gamma$,
		then for all the $U_i$ that meets $\vistorus_\gamma$
		in a nonempty set, the connectedness of $\vistorus_\gamma$
		implies that the corresponding vertices $v_i$ must lie on the same component
		of $X$. Therefore, writing $\mu_{I_k}$
		for the sum of $\mu_i$ for $i\in I_k$, it follows
		that the supports of
		$\partial\mu_{I_k}$ are mutually disjoint subsets of
		$\ocurves_{R,\epsilon}$.
		However, $\mu$ equals $\mu_{I_1}+\cdots+\mu_{I_s}$.
		This is contrary to the assumption that 
		$\mu$ is irreducible (Definition \ref{gluingConditions}).
		
		Guided by the simplicial graph $X$ together
		with the decorating data $c_{ij},c'_{ij}$, we choose a maximal tree
		$T$ of $X$ and perform the hybriding construction accordingly along the edges of $T$.
		In other words, we obtain a new gluing 
			$$(\mathcal{C},\phi')$$
		as follows.
		For any cuff $c\subset\mathcal{C}$, if $c$ 
		is some $c_{ij}$ corresponding to an edge of $T$, 
		we define $\phi'(c_{ij})$ to be $\phi(c'_{ij})$,
		and $\phi'(c'_{ij})$ to be $\phi(c_{ij})$; otherwise,
		we define $\phi'(c)$ to be $\phi(c)$ as before .
		Because $c_{ij}$ and $c'_{ij}$ are projected to nonseparating
		glued cuffs of $F_i$ and $F_j$ respectively, and because
		$T$ is a maximal tree of the connected graph $X$, the new
		gluing $(\mathcal{C},\phi')$ of $\mathcal{F}$
		results in a connected surface $F'$, by induction on the number
		of components $r$.

		Because $\nu^{\sharp}_{\phi(c_{ij})}$ and $\nu^{\sharp}_{\phi(c'_{ij})}$ are
		$(\frac{\epsilon}R)$-equivalent on $\vistorus_{\bar{\gamma}_{ij}}$,
		where $\gamma_{ij}\in\ocurves_{R,\epsilon}$ denotes the
		homotopy class represented by both $c_{ij}$ and $c'_{ij}$,
		and because $(\overline{A_1})_*(\nu^{\sharp}_{c_{ij}})$
		and $\nu^{\sharp}_{\phi(c_{ij})}$ as $(\mathcal{C},\phi)$ is 
		$(R,\epsilon)$-nearly unit shearing,
		the construction of $\phi'$ implies that
		$(\mathcal{C},\phi')$ is $(R,2\epsilon)$-nearly unit shearing.
		This completes the proof.
	\end{proof}
	
	\subsection{Proof of Proposition \ref{gluingProp}}\label{Subsec-proofOfPropositionGluingProp}
		We summarize the proof of Proposition \ref{gluingProp} as follows.
		As $\mu$ is irreducible, $(R,\epsilon)$-nearly evenly footed, and rich,	
		possibly after passing to a positive integral multiple of $\mu$,
		we may construct an $(R,2\epsilon)$-nearly unit-shearing
		gluing $(\mathcal{C},\phi)$ of 
		$\mathcal{F}$ prescribed by $\mu$, so that the resulting surface
		$F$ is connected (Lemmas \ref{gluedCuffs}, \ref{nonseparatingGluedCuffs}, \ref{hybriding}).
		If $\epsilon$ is sufficiently small so that $2\epsilon<\hat{\epsilon}$
		as in Lemma \ref{resultingQF},
		and if $R$ is sufficiently large depending only on $\epsilon$,
		then the induced immersion $j:\,F\looparrowright M$ 
		is $\pi_1$-injective and quasi-Fuchsian.
		This completes the proof of Proposition \ref{gluingProp}.

\section{Hyperbolic geometry of segments with framed endpoints}\label{Sec-basicConstructions}
	In this section, we study the techniques to construct
	$(R,\epsilon)$-panted surfaces in oriented closed hyperbolic $3$-manifolds
	via $\partial$-framed bigons and tripods,
	which generalizes the constructions of \cite{KM-Ehrenpreis}
	in the $2$-dimension case. In fact, our attempt is to
	develop a theory about the geometry of $\partial$-framed
	segments in a closed oriented hyperbolic $3$-manifold, which
	seems to be generalizable to any closed oriented hyperbolic
	manifold.
	
	Following the spirit of Euclid,
	objects to be studied in the geometry 
	are shapes that can be constructed via $\partial$-framed segments,
	so our discussion falls naturally into two parts,
	about \emph{shapes} and about \emph{constructions}.
	In the first part, for our purpose of application,
	we will provide an approximate formula that calculates
	the length and phase of
	sufficiently tame reduced concatenations of approximately
	consecutive chains and cycles, which should be compared
	to the Cosine Law in elementary Euclidean geometry.
	In the second part, we will define a list of basic constructions,
	and derive several more efficient constructions
	by composing the basic ones.
	These basic constructions should be regarded as axioms that can be implemented
	in an oriented closed hyperbolic $3$-manifold.
	The axiomatic approach to constructions brings at least 
	two benefits. First, it highlights
	the	Connection Principle (Lemma \ref{connectionPrinciple})
	as a featuring axiom (Definition \ref{axiomsOfConstructions} (5))
	in the theory. Secondly,
	it allows us to understand the limitation of constructions.
	For instance, as the Spine Principle (Lemma \ref{spinePrinciple})
	implies, any construction provides no additional information
	about the second homology of the $3$-manifold $M$.
	The second point will be of particular importance to us
	because it suggests that 
	certain \emph{a priori} knowledge about
	the fundamental group of $M$ is demanded
	so as to construct any	homologically interesting $(R,\epsilon)$-panted surface. 
	In the treatment of Sections \ref{Sec-pantedCobordismGroup} and \ref{Sec-pantifyingSecondHomologyClasses}, 
	this piece of information
	will be supplemented by a finite presentation of $\pi_1(M)$.
	
	The central problem of interest to us is the following: 
	Given a geodesic immersed graph $Z\looparrowright M$ and a simplical $1$-cycle $c$ of $Z$ realized by
	a union of immersed $(R,\epsilon)$-curves up to homotopy in $M$, if $c$ is null-homologous
	in $Z$, does $c$ bound an $(R,\epsilon)$-panted subsurface $F$ of $M$? 
	We remind the reader that $F$ is not required to be connected or $\pi_1$-injective
	(Definition \ref{pantedSubsurfaceDefinition}), so controlling feet for the gluing 
	is not part of the question. For the simplest example, take $Z$ to be a (pointed) $(R,\epsilon)$-curve
	$\gamma$ and $c$ to be the sum of $\gamma$ and its orientation reversal $\bar{\gamma}$.
	If $\Pi$ is an $(R,\epsilon)$-pair-of-pants with a cuff $\gamma$, which exists by
	a derived construction called \emph{splitting} (Construction \ref{splitting}), 
	then $F$ can be taken as the union of $\Pi$ and $\bar{\Pi}$
	glued along the two pairs of cuffs other than $\gamma$ and $\bar{\gamma}$.
	In Subsection \ref{Subsec-derivedConstructions},
	we will develop several more derived constructions, namely,
	\emph{swapping}, \emph{rotation}, and \emph{antirotation},
	where $Z$ is the carrier graph of a swap pair or a rotation pair
	under certain assumptions of nice geometry, 
	(Definitions \ref{swapPair} and \ref{rotationPair}).	
	These derived constructions are fundamental preparation for 
	Sections \ref{Sec-pantedCobordismGroup} 
	and \ref{Sec-pantifyingSecondHomologyClasses}.
	
	In Subsection \ref{Subsec-terminology}, we introduce some
	basic concepts in the geometry of
	$\partial$-framed segments. In Subsection \ref{Subsec-lengthAndPhaseFormula},
	we state and prove the Length and Phase Formula (Lemma \ref{lengthAndPhaseFormula}).
	In Subsection \ref{Subsec-principlesOfConstruction},
	we introduce the basic constructions in terms of
	the constructible classes, and 
	show the Connection Principle (Lemma \ref{connectionPrinciple})
	and the Spine Principle (Lemma \ref{spinePrinciple}).
	In Subsection \ref{Subsec-derivedConstructions},
	we discuss the derived constructions.		 

	\subsection{Terminology}\label{Subsec-terminology}
		Suppose $M$ is an oriented hyperbolic $3$-manifold.
		We introduce several basic concepts in the geometry of
		$\partial$-framed segments.
			
		\subsubsection{Segments with framed endpoints}
		
		\begin{definition}\label{partialFramedSegment}
			An \emph{oriented $\partial$-framed segment} 
			in $M$ is a triple 
				$$\mathfrak{s}=(s,\vec{n}_{\ini},\vec{n}_{\ter}),$$
			such that $s$ is an immersed oriented compact geodesic segment,
			and that $\vec{n}_{\ini}$ and $\vec{n}_{\ter}$ are two unit normal vectors at
			the initial endpoint and the terminal endpoint, respectively.
			\begin{itemize}
				\item The \emph{carrier segment} is the oriented segment $s$;
				\item The \emph{initial endpoint} $p_\ini(\mathfrak{s})$ and the
				\emph{terminal endpoint} $p_\ter(\mathfrak{s})$ are the initial
				endpoint and the terminal endpoint of $s$, respectively;
				\item The \emph{initial framing} $\vec{n}_\ini(\mathfrak{s})$ and 
				the \emph{terminal framing} $\vec{n}_\ter(\mathfrak{s})$ 
				are the unit normal vectors $\vec{n}_\ini$ and $\vec{n}_\ter$,
				\item The \emph{initial direction} $\vec{t}_\ini(\mathfrak{s})$ and 
				the \emph{terminal direction} $\vec{t}_\ter(\mathfrak{s})$ 
				are the unit tangent vectors in the direction of $s$ 
				at the initial point and the terminal point,
				respectively.
			\end{itemize}			 
			The \emph{orientation reversal} of $\mathfrak{s}$
			is defined to be 
				$$\bar{\mathfrak{s}}=(\bar{s},\vec{n}_{\ter},\vec{n}_{\ini}),$$
			where $\bar{s}$ is the orientation reversal of $s$.
			The \emph{framing rotation} of $\mathfrak{s}$
			by an angle $\phi\in\RR\,/\,2\pi\ZZ$ is defined to be 
			$$\mathfrak{s}(\phi)=
			\left(s,\,
			\vec{n}_\ini\cos\phi+(\vec{t}_\ini\times\vec{n}_\ini)\sin\phi,\,
			\vec{n}_\ter\cos\phi+(\vec{t}_\ter\times\vec{n}_\ter)\sin\phi\right),$$
			where $\times$ means the cross product in the tangent space.
			In particular,
			the \emph{framing flipping} of $\mathfrak{s}$ is defined to be framing rotation by $\pi$,
			denoted as
				$$\mathfrak{s}^*=(s,-\vec{n}_\ini,-\vec{n}_\ter).$$			
		\end{definition}

		It follows from the definition that
			$$\overline{\mathfrak{s}(\phi)}=\bar{\mathfrak{s}}(-\phi),$$
		and in particular, framing flipping commutes with orientation reversion.
		
		\begin{definition}\label{lengthAndPhase}
			For an oriented $\partial$-framed segment $\mathfrak{s}$ in $M$,
			the \emph{length} of $\mathfrak{s}$, denoted as 
				$$\ell(\mathfrak{s})\in(0,+\infty),$$
			is the length of the unframed segment $s$ carrying $\mathfrak{s}$, 
			and the \emph{phase} of $\mathfrak{s}$, denoted
			as 
				$$\varphi(\mathfrak{s})\in\RR\,/\,2\pi\ZZ,$$
			is the angle from the initial framing $\vec{n}_\ini$ to
			the transportation of $\vec{n}_\ter$ to the initial point
			of $s$ via $s$, signed with respect to the normal orientation induced from 
			$\vec{t}_\ini$ and the orientation of $M$.
			We may combine the length and phase
			into a complex value known 
			as the \emph{phasor} of $\mathfrak{s}$, defined as
				$$\phasor(\mathfrak{s})\,=\,e^{\ell(\mathfrak{s})+\imunit\varphi(\mathfrak{s})}.$$
			The value of a phasor always lies outside the unit circle of $\CC$.
			For an oriented closed geodesic curve $c$ in $M$,
			we will also speak of its \emph{length}, \emph{phase}, or \emph{phasor}, by taking 
			an arbitrary unit normal vector $\vec{n}$ at a point $p\in c$, and regarding
			$c$ as a $\partial$-framed segment obtained from cutting at $p$ and endowed
			with framing $\vec{n}$ at both endpoints. 
		\end{definition}
		
		It follows from the definition that length and phase are invariant
		under orientation reversal and under framing rotation.

		\subsubsection{Chains and cycles}
		

		\begin{definition}
			Let $0\leq\delta<\frac\pi3$, and $L>0$, and $0<\theta<\pi$ be constants. 
			Let $M$ be an oriented hyperbolic $3$-manifold.
			\begin{enumerate}
				\item Two oriented $\partial$-framed segments $\mathfrak{s}$ and $\mathfrak{s}'$
				are said to be \emph{$\delta$-consecutive} if the terminal endpoint of $\mathfrak{s}$ 
				is the initial endpoint of $\mathfrak{s}'$, and 
				if the terminal framing of $\mathfrak{s}$
				is $\delta$-close to the initial framing of $\mathfrak{s}'$.
				We simply say \emph{consecutive} if $\delta$ equals zero.
				The \emph{bending angle} between $\mathfrak{s}$ and $\mathfrak{s}'$
				is the angle between the 
				the terminal direction of $\mathfrak{s}$ and 
				the initial direction of $\mathfrak{s}'$,
				which is valued in $[0,\pi]$.
				\item A \emph{$\delta$-consecutive chain} of oriented $\partial$-framed
				segments is a finite sequence 
				$\mathfrak{s}_1,\cdots,\mathfrak{s}_m$ such that each $\mathfrak{s}_i$ is
				$\delta$-consecutive to $\mathfrak{s}_{i+1}$. It is a \emph{$\delta$-consecutive cycle}
				if furthermore $\mathfrak{s}_{m}$ is $\delta$-consecutive to $\mathfrak{s}_1$.
				A $\delta$-consecutive chain or cycle is said to be \emph{$(L,\theta)$-tame},
				if each $\mathfrak{s}_i$ has length greater than $2L$, and the bending angle at each joint point
				is less than $\theta$. 
			\item For an $(L,\theta)$-tame $\delta$-consecutive chain $\mathfrak{s}_1,\cdots,
				\mathfrak{s}_m$, the \emph{reduced concatenation}, denoted as
					$$\mathfrak{s}_1\cdots\mathfrak{s}_m,$$
				is the oriented $\partial$-framed segment as follows.
				The unframed oriented segment of $\mathfrak{s}_1\cdots\mathfrak{s}_m$ is
				the geodesic path which is homotopic to the piecewise geodesic path 
				obtained from concatenating the unframed oriented segments carrying
				$\mathfrak{s}_i$, relative to the initial point of $\mathfrak{s}_1$ and the terminal
				point of $\mathfrak{s}_m$; the initial framing of $\mathfrak{s}_1\cdots\mathfrak{s}_m$
				is the closest unit normal vector to the initial framing of $\mathfrak{s}_1$;
				the terminal framing of $\mathfrak{s}_1\cdots\mathfrak{s}_m$
				is the closest unit normal vector to	the terminal framing of $\mathfrak{s}_m$.
				\item In the case of $(L,\theta)$-tame $\delta$-consecutive cycles, the \emph{reduced cyclic concatenation},
				denoted as 
					$$[\mathfrak{s}_1\cdots\mathfrak{s}_m],$$
				is the unframed oriented closed geodesic curve
				free-homotopic to the concatenation of the unframed oriented segments
				defined similarly as above, 
				assuming the result not contractible to a point.
			\end{enumerate}
		\end{definition}


		\subsubsection{Bigons and tripods}
		We introduce $(L,\delta)$-tame bigons and tripods.
		These objects should be thought of as 
		nearly hyperbolic curves and geodesic
		$2$-simplices in the context of $\partial$-framed segment 
		geometry. We also introduce swap pairs and rotation pairs which
		will be used in Subsection \ref{Subsec-derivedConstructions}.

		\begin{definition}\label{bigon}
			An \emph{$(L,\delta)$-tame bigon} is an $(L,\delta)$-tame $\delta$-consecutive cycle of 
			two oriented $\partial$-framed segments
			$\mathfrak{a},\mathfrak{b}$	of phase $\delta$-close to $0$,
			with respect to the path-metric distance on $\RR/2\pi\ZZ$ valued in $[0,\pi]$.
			We usually say that the
			reduced cyclic concatenation $[\mathfrak{a}\mathfrak{b}]$
			is a \emph{$(L,\delta)$-tame bigon} with the cycle understood. 
			Furthermore, it is said to be \emph{$(l,\delta)$-nearly regular} if 
			the edges $\mathfrak{a}$ and $\mathfrak{b}$ 
			have length $\delta$-close to $l$. 
		\end{definition}

		Note that framing flipping does not change the reduced cyclic 
		concatenation, namely, $[\mathfrak{a}^*\mathfrak{b}^*]$
		is the same as $[\mathfrak{a}\mathfrak{b}]$. However, 
		the orientation of $[\bar{\mathfrak{a}}\bar{\mathfrak{b}}]$ is
		exactly opposite to that of $[\mathfrak{a}\mathfrak{b}]$.

		\begin{definition}\label{swapPair}
			A \emph{$(L,\delta)$-tame swap pair} of bigons
			is a pair of $(L,\delta)$-tame bigons
			$[\mathfrak{a}\mathfrak{b}]$ and 
			$[\mathfrak{a}'\mathfrak{b}']$,
			such that $[\mathfrak{a}\mathfrak{b}']$ 
			and $[\mathfrak{a}'\mathfrak{b}]$ also form $(L,\delta)$-tame bigons,
			and that $\mathfrak{a}$ and
			$\mathfrak{a}'$ have length and phase $\delta$-close to each other respectively,
			and that the same holds for $\mathfrak{b}$ and $\mathfrak{b}'$.
			In this case, we say that the new pair of $(L,\delta)$-tame bigons
			$[\mathfrak{a}\mathfrak{b}']$ and 
			$[\mathfrak{a}'\mathfrak{b}]$
			is the $\delta$-swap pair resulted from \emph{swapping}
			$[\mathfrak{a}\mathfrak{b}]$ and 
			$[\mathfrak{a}'\mathfrak{b}']$, and vice versa.			
		\end{definition}

		\begin{definition}\label{tripod}
			An \emph{$(L,\delta)$-tame tripod}, denoted as
				$$\mathfrak{a}_0\vee\mathfrak{a}_1\vee\mathfrak{a}_2,$$
			is a triple $(\mathfrak{a}_0,\mathfrak{a}_1,\mathfrak{a}_2)$
			of oriented $\partial$-framed segments of length at least $2L$
			and of phase $\delta$-close to $0$,
			such that $\bar{\mathfrak{a}}_i$ is $\delta$-consecutive to $\mathfrak{a}_{i+1}$
			with bending angle $\delta$-close to $\frac\pi3$,
			for $i\in\ZZ_3$.
			Furthermore, it is said to be \emph{$(l,\delta)$-nearly regular} if 
			the legs $\mathfrak{a}_i$ have length $\delta$-close to
			$l$, for $i\in\ZZ_3$. For each $i\in\ZZ_3$,
			$\mathfrak{a}_i$ will be referred to
			as a \emph{leg} of $\mathfrak{a}_0\vee\mathfrak{a}_1\vee\mathfrak{a}_2$,
			and 
				$$\mathfrak{a}_{i,i+1}\,=\,\bar{\mathfrak{a}}_i\mathfrak{a}_{i+1}$$
			will be referred to as a \emph{side}
			of $\mathfrak{a}_0\vee\mathfrak{a}_1\vee\mathfrak{a}_2$.
			Note that the initial framings of $\mathfrak{a}_i$
			are $\delta$-close to each other, so approximately
			the ordered initial directions rotates 
			either couterclockwise or clockwise
			around any of them. We say that $\mathfrak{a}_0\vee\mathfrak{a}_1\vee\mathfrak{a}_2,$
			is \emph{right-handed} if it is the former case, or \emph{left-handed}
			if the latter.
		\end{definition}
		
		Note that framing flipping switches the chirality and the side order of 
		a nearly regular tripod, namely, the chirality 
		of $\mathfrak{a}^*_0\vee\mathfrak{a}^*_1\vee\mathfrak{a}^*_2$
		is exactly opposite to 
		that of $\mathfrak{a}_0\vee\mathfrak{a}_1\vee\mathfrak{a}_2$.
		However, the chirality of $\mathfrak{a}^*_0\vee\mathfrak{a}^*_{-1}\vee\mathfrak{a}^*_{-2}$
		is the same as that of $\mathfrak{a}_0\vee\mathfrak{a}_1\vee\mathfrak{a}_2$, where the indices
		are considered to be in $\ZZ_3$.

		\begin{definition}\label{rotationPair}
			A \emph{$(L,\delta)$-tame rotation pair} of tripods
			is a pair of $(L,\delta)$-tame tripods
			$\mathfrak{a}_0\vee\mathfrak{a}_1\vee\mathfrak{a}_2$ and 
			$\mathfrak{b}_0\vee\mathfrak{b}_1\vee\mathfrak{b}_2$,
			where $\mathfrak{a}_0\vee\mathfrak{a}_1\vee\mathfrak{a}_2$
			is $(l_a,\delta)$-nearly regular, and 
			$\mathfrak{b}_0\vee\mathfrak{b}_1\vee\mathfrak{b}_2$
			is $(l_b,\delta)$-nearly regular. Moreover,
			the chains $\mathfrak{a}_i,\bar{\mathfrak{b}}_j$ 
			are $\delta$-consecutive and $(L,\delta)$-tame for all $i,j\in\ZZ_3$.
			In particular, the terminal endpoints of all the legs are the same. 
		\end{definition}
		
		The reader should not confuse an $(L,\delta)$-tame rotation pair with 
		a pair of well-connected tripods in the sense of \cite{KM-surfaceSubgroup}.
		Rotation pairs are introduced for the derived constructions
		rotation and antirotation (Subsection \ref{Subsec-derivedConstructions}).
		In a quite different form, these constructions will play a role similar
		to the rotation lemmas in \cite[Subsection 8.1]{KM-Ehrenpreis},
		(cf.~Remark \ref{remarkTriangularRelation}).		

	\subsection{The Length and Phase Formula}\label{Subsec-lengthAndPhaseFormula}
		For sufficiently tame concatenations of approximately
		consecutive chains and cycles, the change of
		length and phase under reduction of 
		the concatenation can be approximately calculated.
			
		Recall that for any bending angle $0\leq\vartheta<\pi$,
		the \emph{limit inefficiency} associated to $\vartheta$ is defined as
			$$I(\vartheta)\,=\,2\log(\sec(\vartheta/2)).$$
		The function $I(\vartheta)$ is unbounded, strictly convex, and increasing on $[0,\pi)$,
		and the geometric meaning is explained by Lemma \ref{tameTriangle} (2).
		In particular, $I(0)=0$, $I(\frac\pi3)=2\log2-\log3$, and $I(\frac\pi2)=\log2$.
				
		\begin{lemma}[Length and Phase Formula]\label{lengthAndPhaseFormula}
			Given any positive constants $\delta$, $\theta$, $L$ 
			where $0<\theta<\pi$ and $L\geq I(\theta)+10\log2$,
			the following statements hold in any oriented hyperbolic $3$-manifold.
			\begin{enumerate}
				\item If $\mathfrak{s}_1,\cdots,
				\mathfrak{s}_m$ is an $(L,\theta)$-tame $\delta$-consecutive chain of oriented
				$\partial$-framed segments, denoting the bending angle 
				between $\mathfrak{s}_i$ and $\mathfrak{s}_{i+1}$ as $\theta_i\in[0,\theta)$,
				then
					$$\left|\ell(\mathfrak{s}_1\cdots\mathfrak{s}_m)
					-\sum_{i=1}^m\ell(\mathfrak{s}_i)+\sum_{i=1}^{m-1}
					 I(\theta_i)\right|\,<\,\frac{(m-1)e^{(-L+10\log2)/2}\sin(\theta/2)}{L-\log2},$$
				and
					$$\left|\varphi(\mathfrak{s}_1\cdots\mathfrak{s}_m)-\sum_{i=1}^m\varphi(\mathfrak{s}_i)\right|
					\,<\,(m-1)(\delta+e^{(-L+10\log2)/2}\sin(\theta/2)),$$
				where $|.|$ on $\RR/2\pi\ZZ$ is understood as the distance from zero valued in $[0,\pi]$.
				\item If $\mathfrak{s}_1,\cdots,
				\mathfrak{s}_m$ is an $(L,\theta)$-tame $\delta$-consecutive cycle of oriented
				$\partial$-framed segments, denoting the bending angle
				between $\mathfrak{s}_i$ and $\mathfrak{s}_{i+1}$ as $\theta_i\in[0,\pi-\theta]$
				with $\mathfrak{s}_{m+1}$ equal to $\mathfrak{s}_0$ by convention,
				then
					$$\left|\ell([\mathfrak{s}_1\cdots\mathfrak{s}_m])
					-\sum_{i=1}^m\ell(\mathfrak{s}_i)+\sum_{i=1}^{m} I(\theta_i)\right|\,<\,
					\frac{me^{(-L+10\log2)/2}\sin(\theta/2)}{L-\log2},$$
				and
					$$\left|\varphi([\mathfrak{s}_1\cdots\mathfrak{s}_m])
					-\sum_{i=0}^m\varphi(\mathfrak{s}_i)\right|\,<\,
					m(\delta+e^{(-L+10\log2)/2}\sin(\theta/2)),$$
				where $|.|$ on $\RR/2\pi\ZZ$ is understood as the distance from zero valued in $[0,\pi]$.
			\end{enumerate}
		\end{lemma}
		
		\begin{remark}\label{remarkLengthAndPhaseFormula}
			In this paper, we will only apply the formula for $\theta$ equal to $\pi/2$, $\pi/3$, or
			constant multiples of $\delta$. Moreover, for
			any positive constant $\delta$, $K$ where $K\geq1$ and $K\delta\,<\,1/\sqrt{2}$,
			we will always ensure that 
			the participating chains or cycles are $(K\delta)$-consecutive,
			$(-2\log\delta+10\log2,\,\pi/2)$-tame.
			then the error bounds can be replaced by 
			$2(m-1)K\delta$ and $2mK\delta$
			for (1) and (2), respectively.
		\end{remark}

		The proof relies on a lemma in elementary hyperbolic geometry, which provides 
		some key estimation for tame concatenation of segments. 		
						
		\begin{lemma}\label{tameTriangle}
			Given any constants $0<\theta<\pi$,
			and	$L\geq I(\theta)+\log 2$,
			suppose that $\triangle ABC$ is a geodesic triangle in hyperbolic space,
			where $|CA|>L$, $|CB|> L$, and $\angle C=\pi-\theta$,
			then
			\begin{enumerate}
				\item $\angle A+\angle B\,<\,e^{(-L+3\log2)/2}\sin(\theta/2)$.
				\item $I(\theta)-e^{(-L+5\log2)/2}\sin(\theta/2)\,/\,(L-\log2)\,<\,|CA|+|CB|-|AB|\,<\,I(\theta)$.			
			\end{enumerate}
		\end{lemma}
		
		\begin{proof} 
			To prove the inequality (1), it suffices to assume 
			that $|CA|=|CB|=L$ since $\angle A+\angle B$, which equals
			$\pi-\angle C-\mathrm{Area}(\triangle ABC)$,
			achieves its unique maximum in this case. Let $M$ be
			the midpoint of $AB$. In the right triangle
			$\triangle ACM$, it follows from the Dual
			Law of Cosines that
				$$-\cos\frac{\angle C}2\cos\angle A+\sin\frac{\angle C}2\sin\angle A\,\cosh|AC|\,=\,0.$$
			Therefore,
				$$\angle A\,<\,\tan\angle A\,=\,\frac{\tan(\theta/2)}{\cosh L}\,=\,
				\frac{\sec(\theta/2)\sin(\theta/2)}{\cosh L}.$$
			Since the assumption $L>I(\theta)+\log2$ implies 
				$$\sec^2(\theta/2)\,<\,e^L/2\,<\,\cosh L,$$			
			we have
				$$\angle A\,<\,\frac{(e^{L/2}/\sqrt{2})\sin(\theta/2)}{e^L/2}\,=\,\sqrt{2}\,e^{-L/2}\sin(\theta/2).$$
			The same estimation holds for $\angle B$, so
				$$\angle A+\angle B\,<\,\sqrt{2}\,e^{-L/2}\sin(\theta/2)\,=\,e^{(-L+3\log2)/2}\sin(\theta/2),$$
			which is the inequality (1).
			
			To prove the inequality (2), consider the inscribed circle $\odot J$ of $\triangle ABC$,
			denoting the tangent point of $\odot J$ with $AB$, $BC$, and $CA$ as $T_c$, $T_b$, and $T_a$,
			respectively. Then
				$$|CA|+|CB|-|AB|\,=\,|CT_a|+|CT_b|,$$
			which approaches
			the supremum as $|CA|$ and $|CB|$ tend to $+\infty$, and achieves the unique minimum
			when $|CA|=|CB|=L$. A direct computation shows that the supremum is exactly $I(\theta)$,
			so the upper bound of inequality (2) holds.
			
			For the lower bound, it hence suffices to assume that $|CA|=|CB|=L$.
			We write $\triangle A^*B^*C$ for the triangle with ideal points $A^*$ and $B^*$,
			and let $\odot J^*$ be the inscribed circle which is tangent to $A^*B^*$, $B^*C$ and $CA^*$
			at $T^*_c$, $T^*_a$ and $T^*_b$, respectively. Note that now 
			$T_c$ is the midpoint of $AB$, and similarly for $T^*_c$.
			It is also clear that 
				$$|CT^*_b|-|CT_b|\,=\,|T^*_bT_b|\,<\,|J^*J|\,<\,|T_cT^*_c|.$$
			In the right triangle $\triangle CT_cA$, it follows from the Dual Law
			of Cosines that
				$$\sin\angle A\cosh |AT_c|\,=\,\cos\frac{\angle C}2.$$
			Therefore,
				$$\cosh|AT_c|\,=\,\frac{\cos((\pi-\theta)/2)}{\sin\angle A}
				\,>\,\frac{\sin(\theta/2)}{\sqrt{2}\,e^{-L/2}\sin(\theta/2)}
				\,=\,e^{L/2}/\sqrt{2}.$$
			Since $(e^{|AT_c|}+1)/2\,>\,\cosh|AT_c|$, we obtain
			\begin{eqnarray*}
				|AT_c|&>&\log(\sqrt{2}\,e^{L/2}-1)\\
					&=&\log(\sqrt{2}\,e^{L/2})+\log(1-e^{-L/2}/\sqrt{2})\\
					&>&\log(\sqrt{2}\,e^{L/2})+\log(1-e^{-(\log2)/2}/{\sqrt{2}})\\
					&=&(L-\log2)/2.
			\end{eqnarray*}
			On the other hand, the difference between
			the area of $\triangle A^*B^*C$ and $\triangle ABC$ is the
			area of the quadrilateral $AT_cT^*_cA^*$, which clearly equals the value
			of $\angle A$ in $\triangle ABC$. Because $AT_c$ and $A^*T^*_c$ are perpendicular
			to $T_cT^*_c$, there is the comparison of area:
				$$|T_cT^*_c|\cdot|AT_c|\,<\,\mathrm{Area}(AT_cT^*_cA^*)\,=\,\angle A.$$
			Therefore, the estimations of $\angle A$ and $|AT_c|$ above imply
				$$|T_cT^*_c|\,<\,\frac{\angle A}{|AT_c|}
				\,<\,2\sqrt{2}\,e^{-L/2}\sin(\theta/2)\,/\,(L-\log2).$$
			We obtain
			\begin{eqnarray*}
				|CA|+|CB|-|AB|&=&|CT_a|+|CT_b|\\
				&>&|CT^*_a|+|CT^*|-2|T_cT^*_c|\\
				&=&I(\theta)-4\sqrt{2}\,e^{-L/2}\sin(\theta/2)\,/\,(L-\log2)\\
				&=&I(\theta)-e^{(-L+5\log2)/2}\sin(\theta/2)\,/\,(L-\log2),
			\end{eqnarray*}
			which verifies the lower bound in the inequality (2).
		\end{proof}
		
		\begin{proof}[{Proof of Lemma \ref{lengthAndPhaseFormula}}]
			Since it follows from certain standard estimation argument provided
			Lemma \ref{tameTriangle}, we only sketch the proof.
					
			To see the statement (1), we write
				$$\Delta\ell(\mathfrak{s}_1\cdots\mathfrak{s}_m)=\ell(\mathfrak{s}_1\cdots\mathfrak{s}_m)
					-\sum_{i=1}^m\ell(\mathfrak{s}_i)+\sum_{i=1}^{m-1} I(\theta_i)$$
			and
				$$\Delta\varphi(\mathfrak{s}_1\cdots\mathfrak{s}_m)=
				\varphi(\mathfrak{s}_1\cdots\mathfrak{s}_m)-\sum_{i=1}^m\varphi(\mathfrak{s}_i)$$
			for the error terms that we will estimate. 
			
			For each $1\leq i<m$, let $\vec{n}_i$ be any unit vector 
			perpendicular to both $\mathfrak{s}_i$ and 
			$\mathfrak{s}_{i+1}$ at their joint point, and let $\vec{n}_0=\vec{n}_\ini(\mathfrak{s}_1)$,
			$\vec{n}_m=\vec{n}_\ter(\mathfrak{s}_m)$ by convention.
			Let $\tilde{\mathfrak{s}}_i$ be the $\partial$-framed segment
			with the same carrier segment as that of $\mathfrak{s}_i$,
			and with $\vec{n}_\ini(\tilde{\mathfrak{s}}_i)$ and $\vec{n}_\ter(\tilde{\mathfrak{s}}_i)$
			be $\vec{n}_{i-1}$ and $\vec{n}_i$ respectively.
			Then $\tilde{\mathfrak{s}}_1,\cdots,\tilde{\mathfrak{s}}_m$ is a consecutive
			$(L,\theta)$-chain. Note that $\varphi(\tilde{\mathfrak{s}}_i)-\varphi(\mathfrak{s}_i)$
			equals $\angle(\vec{n}_{i-1},\vec{n}_\ini(\mathfrak{s}_i))-\angle(\vec{n}_i,\vec{n}_\ter(\mathfrak{s}_i))$.
			Since $\mathfrak{s}_1,\cdots,\mathfrak{s}_m$ is $\delta$-consecutive,
			$\angle(\vec{n}_i,\vec{n}_\ini(\mathfrak{s}_{i+1}))$
			is $\delta$-close to $\angle(\vec{n}_i,\vec{n}_\ter(\mathfrak{s}_i))$,
			so the new error of phase differ from the old by
			$|\Delta\varphi(\tilde{\mathfrak{s}}_1\cdots\tilde{\mathfrak{s}}_m)-
				\Delta\varphi(\mathfrak{s}_1\cdots\mathfrak{s}_m)|\,<\,(m-1)\delta.$
			It is clear that 
			$|\Delta\ell(\tilde{\mathfrak{s}}_1\cdots\tilde{\mathfrak{s}}_m)-
				\Delta\ell(\mathfrak{s}_1\cdots\mathfrak{s}_m)|\,=\,0.$

			If $m$ equals $1$, we have already done, as
			$\Delta\ell(\tilde{\mathfrak{s}}_1)$ and
			$\Delta\varphi(\tilde{\mathfrak{s}}_1)$
			are both $0$ by definition. If $m$ is greater than $1$,
			we may consider a chain $\mathfrak{s}'_1,\cdots,\mathfrak{s}'_{m-1}$
			as follows. For each $1<i<m$, write $\tilde{\mathfrak{s}}_i$
			as the concatenation of two consecutive oriented $\partial$-framed segments
			$\tilde{\mathfrak{s}}_{i-}$ and $\tilde{\mathfrak{s}}_{i+}$
			of equal length and phase. For $1< i< m-1$,
			let $\mathfrak{s}'_i$ be $\tilde{\mathfrak{s}}_{i+}\tilde{\mathfrak{s}}_{(i+1)-}$.
			Let $\mathfrak{s}'_1$ be $\tilde{\mathfrak{s}}_{1}\tilde{\mathfrak{s}}_{2-}$,
			and $\mathfrak{s}'_m$ be $\tilde{\mathfrak{s}}_{(m-1)+}\tilde{\mathfrak{s}}_m$,
			or in the case that $m$ equals $2$, let $\mathfrak{s}'_1$ be
			$\tilde{\mathfrak{s}}_1\tilde{\mathfrak{s}}_2$.
			It follows immediately from Lemma \ref{tameTriangle}
			that $\mathfrak{s}'_1,\cdots,\mathfrak{s}'_{m-1}$ is
			$\beta'$-consecutive and
			$(L,\beta')$-tame,
			where 
				$$\beta'\,=\,e^{(-L+3\log2)/2}\sin(\theta/2).$$
			Moreover,
			$\Delta\ell(\mathfrak{s}'_1\cdots\mathfrak{s}'_{m-1})$
			is $(2(m-1)\beta'/(L-\log2))$-close to
			$\Delta\ell(\tilde{\mathfrak{s}}_1\cdots\tilde{\mathfrak{s}}_m)$.
			It is also clear that
			$\Delta\varphi(\mathfrak{s}'_1\cdots\mathfrak{s}'_{m-1})$
			is $((m-1)\beta')$-close to 
			$\Delta\varphi(\tilde{\mathfrak{s}}_1\cdots\tilde{\mathfrak{s}}_m)$.
			Thus, 
				$$|\Delta\ell(\mathfrak{s}'_1\cdots\mathfrak{s}'_{m-1})-
				\Delta\ell(\mathfrak{s}_1\cdots\mathfrak{s}_m)|\,<\,2(m-1)\beta'/(L-\log2),$$
			and			
				$$|\Delta\varphi(\mathfrak{s}'_1\cdots\mathfrak{s}'_{m-1})-
				\Delta\varphi(\mathfrak{s}_1\cdots\mathfrak{s}_m)|\,<\,(m-1)(\delta+\beta').$$
			
			If $m$ equals $2$, we have done since the chain
			$\mathfrak{s}'_1,\cdots,\mathfrak{s}'_{m-1}$ has only 
			one term. If $m$ is greater than $2$, we may further obtain 
			a chain $\mathfrak{s}''_1,\cdots,\mathfrak{s}''_{m-2}$
			from $\mathfrak{s}'_1,\cdots,\mathfrak{s}'_{m-1}$, in a similar
			way as we obtain the latter from 
			$\mathfrak{s}_1,\cdots,\mathfrak{s}_m$.
			Proceed iteratively to obtain new chains
			$\mathfrak{s}^r_1,\cdots,\mathfrak{s}^r_{m-r}$
			from $\mathfrak{s}^{r-1}_1,\cdots,\mathfrak{s}^{r-1}_{m-{r-1}}$
			until $m-1$ equals $1$. 
			The chain $\mathfrak{s}^r_1,\cdots,\mathfrak{s}^r_{m-r}$
			is $\beta^{(r)}$-consective and $(L,\beta^{(r)})$-tame
			where
				$$\beta^{(r+1)}\,=\,2^{3/2}e^{-L/2}\sin(\beta^{(r)}/2)$$ 
			for $0<r<m$.
			Thus for $0<r<m$, we have 
				$$|\Delta\ell(\mathfrak{s}^{r}_1\cdots\mathfrak{s}^{r}_{m-r})-
				\Delta\ell(\mathfrak{s}^{r-1}_1\cdots\mathfrak{s}^{r-1}_{m-r+1})|\,<\,2(m-r+1)\beta^{(r)}/(L-\log2),$$
			and			
				$$|\Delta\varphi(\mathfrak{s}^{r}_1\cdots\mathfrak{s}^{r}_{m-r})-
				\Delta\varphi(\mathfrak{s}^{r-1}_1\cdots\mathfrak{s}^{r-1}_{m-r+1})|\,<\,(m-r+1)(\beta^{(r-1)}+\beta^{(r)}),$$
			where 
				$$\beta^{(0)}=\delta$$
			by convention.
						
			Summing up the error of length in each step
			yields that
			\begin{eqnarray*}
				|\Delta\ell(\mathfrak{s}\cdots\mathfrak{s}_m)|&<&\sum_{r=1}^{m-1}2(m-r+1)\beta^{(r)}/(L-\log2)\\
				&<&4(m-1)\beta'/(L-\log2)\\
				&<&(m-1)e^{(-L+10\log2)/2}\sin(\theta/2)/(L-\log2),
			\end{eqnarray*}
			and that
			\begin{eqnarray*}
				|\Delta\varphi(\tilde{\mathfrak{s}}_1\cdots\tilde{\mathfrak{s}}_m)|&<&\sum_{r=1}^{m-1}(m-r+1)(\beta^{(r-1)}+\beta^{(r)})\\
				&<& (m-1)\delta+2(m-1)\beta'\\
				&<& (m-1)(\delta+e^{(-L+10\log2)/2}\sin(\theta/2)).
			\end{eqnarray*}
			Therefore, we have the estimations of the statement (1).
			
			The statment (2) can be proved similarly. We first obtain
			a consecutive cycle $\tilde{\mathfrak{s}}_1,\cdots,\tilde{\mathfrak{s}}_m$
			from $\mathfrak{s}_1,\cdots,\mathfrak{s}_m$,
			then construct iteratively the cycles
			$\mathfrak{s}^{r+1}_1,\cdots,\mathfrak{s}^{r+1}_m$
			from $\mathfrak{s}^{r}_1,\cdots,\mathfrak{s}^{r}_m$ by joining
			consequential midpoints,
			starting with $\mathfrak{s}^{0}_1,\cdots,\mathfrak{s}^{0}_m$
			which is $\tilde{\mathfrak{s}}_1,\cdots,\tilde{\mathfrak{s}}_m$.
			Similar estimations as before hold in this case, and as
			$r$ tends to infinity, the (non-reduced) cyclic
			concatenation $\mathfrak{s}^{r}_1,\cdots,\mathfrak{s}^{r}_m$
			converges to $[\mathfrak{s}_1\cdots\mathfrak{s}_m]$ geometrically.
			Note that the summation of errors in this case will be a
			series which converges absolutely, but the upper bound will
			stay unchanged except for the replacement of
			$m-1$ by $m$. Combining the estimations as before yields
			the estimations in the statement (2).			
		\end{proof}
		
		We close this subsection with a lemma of elementary hyperbolic geometry,
		which will be used later.
		Recall that the \emph{Fermat point} of a geodesic triangle $\triangle ABC$
		in hyperbolic space is the point $F$ minimizing $|FA|+|FB|+|FC|$,
		which lies on the $2$-simplex bounded by $\triangle ABC$.
		If the inner angles of $\triangle ABC$ are 
		all smaller than $120^\circ$,
		then $F$ lies inside $\triangle ABC$
		and $FA$, $FB$, $FC$ form 
		an angle of $120^\circ$ pairwise.
		
		\begin{lemma}\label{FermatPoint}
			For any positive constant $d$,
			if the radius $d+\mathrm{arccosh}(2/\sqrt3)$
			ball centered at each vertex of $\triangle ABC$ is separated
			by a hyperplane from the other vertices,
			then $F$ lies inside $\triangle ABC$,
			and $|FA|$, $|FB|$, and $|FC|$ are all greater than $d$.
		\end{lemma}
		
		\begin{proof}
			Under the assumption, for any point $P$ with $|PA|$ bounded by
			$d$, the angle $\angle BPC$ is at most $120^\circ$,
			so the Fermat point lies outside the $d$ neighborhood of $A$.
			Similar statements hold for $B$ and $C$.
			Thus $F$ lies inside $\triangle ABC$ of distance at least
			$d$ from the vertices.
		\end{proof}

	\subsection{Principles of construction}\label{Subsec-principlesOfConstruction}
		Before we discuss how to construct 
		various $(R,\epsilon)$-panted surfaces
		with $\partial$-framed segments (Subsection \ref{Subsec-derivedConstructions}),
		in this subsection, we wish to formally
		discuss what we mean by a \emph{construction}
		(Definition \ref{constructibleExtensions}).
		We will enumerate our basic constructions 
		as axioms (Definition \ref{axiomsOfConstructions}).
		These constructions are realizable 		
		in an oriented closed hyperbolic $3$-manifold $M$
		essentially because of the Connection Principle (Lemma \ref{connectionPrinciple}).
		Then we will state the Spine Principle (Lemma \ref{spinePrinciple}),
		which morally says that since we 
		are only drawing auxiliary $\partial$-framed segments
		in any such construction, we will not gain any 
		new knowledge about the second homology
		of $M$. This observation will be of fundamental importance
		when we pantify a second homology class
		(Section \ref{Sec-pantifyingSecondHomologyClasses}).
		In practice, it will be convenient
		to describe constructions more naturally in terms
		of $\partial$-framed segments, and at the end of this
		subsection, we will explain
		how to translate between the natural description
		and the formal description in terms of constructible extensions.
		However, the reader may safely skip the discussion of this subsection 
		until Section \ref{Sec-pantifyingSecondHomologyClasses}.
				
		\subsubsection{Constructible classes}
		We provide a formal definition of constructible objects
		in terms of partially-$\Delta$ spaces over an oriented closed
		hyperbolic $3$-manifold.
		
		\begin{definition}\label{partiallyDeltaSpace}
			Let $M$ be an oriented hyperbolic $3$-manifold.
			A \emph{partially-$\Delta$ space over $M$}
			is a triple $(X,X_{\Delta},f_X)$ as follows.
			The space $X_\Delta$ is
			a CW subspace of a CW space $X$,
			enriched with a $\Delta$-complex structure;
			the map $f_X:X\to M$ is geodesic restricted to simplices of
			$X_\Delta$, possibly degenerate.
			We often simply mention
			a partially-$\Delta$ space $X$ with $X_\Delta$ and $f_X$
			implicitly assumed.
			A \emph{partially combinatorial} map
			between two partially-$\Delta$ space $X$ and $Y$
			is a CW map $\phi:(X,X_\Delta)\to (Y,Y_\Delta)$,
			combinatorial with respect to the $\Delta$-complex
			structure of $X_\Delta$ and $Y_\Delta$,
			such that $f_X=f_Y\circ\phi$.
			For a partially-$\Delta$ space $Z$
			over $M$, an \emph{extension} of 
			$Z$ is a partially-$\Delta$ space $X$ 
			together with a partially combinatorial
			embedding $\phi:(Z,Z_\Delta)\to (X,X_\Delta)$.
		\end{definition}
		
		We list our basic constructions by the following axioms,
		and verify that they are all possible in the situation that
		we will be concerned with. Denote by 
		$[\varepsilon_0,\cdots,\varepsilon_n]$ the standard $n$-simplex
		in $\RR^{n+1}$ spanned by the standard basis vectors $\varepsilon_i$.
				
		\begin{definition}\label{axiomsOfConstructions}
			For an oriented hyperbolic $3$-manifold $M$
			and a pair of positive constants $(L,\delta)$,
			the \emph{axioms of constructions} 
			are the following statements:
			\begin{enumerate}
				\item \emph{Vertex Creation}. 
				Suppose that $p\in M$ is a point.
				If $X$ is a partially-$\Delta$ space over $M$, then
				there exists a partially-$\Delta$
				space $X'$ as follows. The space pair $(X',X'_\Delta)$ is 
				$(X\sqcup v,\,X_\Delta\sqcup v)$ where $v$ is a new vertex;
				the map $f_{X'}$ is an extension of $f_X$ such that $f_{X'}(v)=p$.
				\item \emph{Simplex Subdivision}.
				Suppose that $\vec{t}_p\in T_pM$ is a unit vector at a point $p\in M$,
				and that $\lambda$ is a positive real constant.
				If $X$ is a partially-$\Delta$ space 
				over $M$ with an $n$-simplex $\sigma:[\varepsilon_0,\cdots,\varepsilon_{n}]\to X_\Delta$,
				such that $f(\varepsilon_0)=p$,
				then there exists a partially-$\Delta$ space $X'$ as follows.
				The space pair $(X',X'_\Delta)$ is 
				$(X\cup_{\sigma}[\varepsilon_0,\cdots,\varepsilon_{n+1}],\, X_\Delta\cup_{\sigma}[\varepsilon_0,\cdots,\varepsilon_{n+1}])$,
				where $X_\Delta\cup_\sigma [\varepsilon_0,\cdots,\varepsilon_{n+1}]$ denotes the mapping cone
				of $\sigma$, and similarly for $X_\Delta\cup_\sigma [\varepsilon_0,\cdots,\varepsilon_{n+1}]$;
				the map $f_{X'}$ is the extension of $f_X$ so that $f_{X'}|_{[\varepsilon_0,\varepsilon_{n+1}]}$
				is a geodesic segment of length $\lambda$ with the initial point $p$ and
				the initial direction $\vec{t}_p$, and that $f_{X'}|_{[\varepsilon_0,\cdots,\varepsilon_{n+1}]}$ is
				geodesic.
				\item\emph{Simplex Filling}. 
				If $X$ is a partially-$\Delta$
				space over $M$, and if there is a combinatorial map $\phi:
				\cup_{i=1}^n [\varepsilon_0,\varepsilon_1,
				\cdots,\widehat{\varepsilon_i},\cdots\varepsilon_n]
				\to X_\Delta$ from the last $n$ faces
				of an $n$-simplex
				$[\varepsilon_0,\cdots,\varepsilon_n]$ to $X_\Delta$,
				then there exists a partially-$\Delta$
				space $X'$ over $M$ as follows.
				The space pair $(X',X'_\Delta)$ is 
				$(X\cup_{\phi}[\varepsilon_0,\cdots,\varepsilon_{n}],\, X_\Delta\cup_{\phi}[\varepsilon_0,\cdots,\varepsilon_{n}])$;
				the map $f_{X'}$ is the extension of $f_X$ so that $f_{X'}|_{[\varepsilon_0,\cdots,\varepsilon_n]}$
				is geodesic.				
				\item \emph{Free Loop Reduction}.
				If $X$ is a partially-$\Delta$
				space over $M$, and if there is an edge $e$
				of $X_\Delta$ with its boundary attached
				the same vertex	$v$ and with $f_X|_e$ nondegenerate,
				then there exists a partially-$\Delta$
				space $X'$ over $M$ as follows.
				The space pair $(X',X'_\Delta)$ is 
				$(X\cup v'\cup e' \cup A,\,X_\Delta\cup v'\cup e')$, where
				$v'$ is a new vertex, $e'$ is a new edge with its boundary
				attached to $v$,
				and $A$ is a new oriented annulus with the boundary 
				the disjoint union of $v\cup e$ and the orientation reversal
				of $v'\cup e'$; 
				the map $f_{X'}$ is an extension of $f_{X}$
				such that $f_{X'}|_{v'\cup e'}$ carries 
				the closed geodesic in $M$ freely homotopic
				to $f_{X}|_{v\cup e}$.
				\item \emph{Edge Connection}.
				Suppose that $\vec{t}_p,\,\vec{n}_p\in T_pM$ and $\vec{t}_q,\,\vec{n}_q\in T_{q}M$ are pairs
				of orthogonal unit vectors at points $p,q\in M$ respectively,
				and that $\lambda$ is a complex constant of modulus at least $L$.
				If $X$ is a partially-$\Delta$
				space over $M$ with (not necessarily distinct) vertices $v,w$
				of $X_\Delta$ such that
				$f_{X}(v)=p$ and $f_{X}(w)=q$, then there exists a partially-$\Delta$
				space $X'$ over $M$ as follows.
				The space pair $(X',X'_\Delta)$ is $(X\cup e,\,X_\Delta\cup e)$,
				where $e$ is a new edge from $v$ to $w$;
				the map $f_{X'}$ is an extension of $f_{X}$ such that $f_{X'}|_e$ carries
				an oriented $\partial$-framed segment $\mathfrak{s}$ 
				from $p$ to $q$ satisfying the following.
				\begin{itemize}
					\item The oriented $\partial$-framed segment
					$\mathfrak{s}$ has length and phase $\delta$-close to $\log |\lambda|$ and 
					$\arg(\lambda)$, respectively.
					The initial direction and framing of $\mathfrak{s}$
					are $\delta$-close to $\vec{t}_p$ and $\vec{n}_p$, respectively.
					The terminal direction and framing of $\mathfrak{s}$
					are $\delta$-close to $\vec{t}_q$ and $\vec{n}_q$, respectively.
				\end{itemize}
			\end{enumerate}
		\end{definition}
		
		\begin{definition}\label{constructibleExtensions}
			Suppose that $Z$ is a partially-$\Delta$ space over $M$.
			We define the \emph{class of constructible extensions of $Z$}
			with respect to $(L,\delta)$
			to be the smallest class $\mathscr{C}$ of extensions of $Z$,
			satisfying that $Z\in\mathscr{C}$, and that $X\in\mathscr{C}$
			implies	$X'\in\mathscr{C}$ for any $X'$
			obtained from $X$ by one of the axioms of constructions above.
			A partially-$\Delta$ space $U$ over $M$ is
			said to be \emph{constructible from $Z$} if there exists
			a constructible extension $X\in \mathscr{C}$ of $Z$ and
			a partially combinatorial map $\psi:U\to X$, such that $f_U=f_X\circ\psi$.				
		\end{definition}
		
		In practice, $Z$ will serve as the object that a construction starts with,
		where $Z_\Delta$ records the piece of information that are available for
		the construction; $X$ will serve as the recipe of the construction;
		and $U$ will serve as the resulting object of the construction, 
		where $U_\Delta$ 
		is meant to record certain shape properties that result should satisfy.

		For any oriented closed hyperbolic $3$-manifold $M$, the axioms of constructions
		listed in Definition \ref{axiomsOfConstructions} are all realizable
		for any positive constant $\delta$ universally small positive,
		and any positive constant
		$L$ is sufficiently	large depending only on $\delta$ and $M$.
		In fact, the first four axioms are true for any hyperbolic manifold,
		and the last axiom follows from the Connection Principle (Lemma \ref{connectionPrinciple}),
		which holds for oriented closed hyperbolic $3$-manifolds.		

		\subsubsection{The Connection Principle}
		We emphasize the following Connection Principle
		as it is the fundamental reason
		for all our constructions of nearly regular pants to work. For example,
		it implies that $\ocurves_{R,\epsilon}$ and $\opants_{R,\epsilon}$ are nonempty
		for an oriented closed hyperbolic $3$-manifold $M$, provided that
		$\epsilon$ is universally small positive and that $R$ is sufficiently large
		depending only on $M$ and $\epsilon$.		
		
		\begin{lemma}[Connection Principle]\label{connectionPrinciple}
			For any universally small positive $\delta$, and for any sufficiently
			large positive $L$ depending only on $\delta$ and $M$,
			the following statement holds.
			If 
			$\vec{t}_p,\,\vec{n}_p\in T_pM$ and $\vec{t}_q,\,\vec{n}_q\in T_{q}M$ are pairs
			of orthogonal unit vectors at points $p,q\in M$ respectively,
			and if $\lambda$ is a complex number of modulus at least $L$,
			then there exists an oriented $\partial$-framed 
			segment $\mathfrak{s}$ from $p$ to $q$ satisfying the following.
			\begin{itemize}
				\item The oriented $\partial$-framed segment
				$\mathfrak{s}$ has length and phase $\delta$-close to $\log |\lambda|$ and 
				$\arg(\lambda)$, respectively.
				The initial direction and framing of $\mathfrak{s}$
				are $\delta$-close to $\vec{t}_p$ and $\vec{n}_p$, respectively.
				The terminal direction and framing of $\mathfrak{s}$
				are $\delta$-close to $\vec{t}_q$ and $\vec{n}_q$, respectively.
			\end{itemize}
		\end{lemma}
		
		\begin{proof}
			This follows from the fact that
			the frame flow is mixing \cite[Theorem 4.2]{KM-surfaceSubgroup}.
			The argument is the same as that of \cite[Lemma 4.4]{KM-surfaceSubgroup}.
		\end{proof}

		\subsubsection{The Spine Principle}
		The Spine Principle says that any constructible extension
		of a partially-$\Delta$ space over an oriented hyperbolic $3$-manifold
		is \emph{relatively $1$-spined}, or precisely as follows. 
		
		\begin{lemma}[Spine Principle]\label{spinePrinciple}
			If $(X,X_\Delta,f_X)$ is a constructible extension of a partially-$\Delta$ space 
			$(Z,Z_\Delta,f_Z)$ over an oriented hyperbolic $3$-manifold $M$,
			then the defining inclusion $\phi:Z\to X$ can be 
			extended to be a
			homotopy equivalence $Z'\simeq X$ where $Z'$ 
			is obtained from $Z$ by attaching cells of dimension at most $1$.
		\end{lemma}
		
		\begin{proof}
			This follows immediately from inspecting the construction axioms
			listed in Definition \ref{constructibleExtensions}.
		\end{proof}

		\subsubsection{Describing a construction}\label{Subsubsec-describingAConstruction}
		In the rest of this paper, we will often describe a construction without
		explicitly writing down the associated partially-$\Delta$ spaces.
		Instead, the hypothesis of a construction will be stated in terms of 
		$\partial$-framed segments. 
		
		Using Simplex Subdivision of dimension $0$, $1$, or $2$, we can construct
		\begin{itemize}
			\item a segment emanating from a constructed point in a given direction;
			\item division of a constructed geodesic segment in a given ratio; or
			\item a tripod carried by a constructed $2$-simplex with a given center.
		\end{itemize}
		We can enrich these with $\partial$-framings to turn them into 
		$\partial$-framed objects.
		For example, by saying that 
		we construct a right-handed $\partial$-framed right-handed tripod 
		$\mathfrak{a}_0\vee\mathfrak{a}_1\vee\mathfrak{a}_2$ carried by
		an immersed oriented $2$-simplex in a hyperbolic $3$-manifold
		$M$, we precisely mean a procedure as follows.
		Start with a given constructible extension
		$X$ of a given partially-$\Delta$ space $Z$ over $M$, together with
		a given $2$-simplex $\sigma:[\varepsilon_0,\varepsilon_1,\varepsilon_2]\to X_\Delta$,
		such that the immersed $2$-simplex is given by $f_X\circ\sigma$;
		let $X'$ be an extension of $X$ attaching a new $3$-simplex 
		$\sigma':[\varepsilon_0,\varepsilon_1,\varepsilon_2,
		\varepsilon_3]\to X'$; we consider the leg $\mathfrak{a}_i$
		to be carried by the geodesic segment 
		$f_{X'}\circ\sigma'|_{[\varepsilon_3,\varepsilon_i]}$.
		By suitably choosing a tangent direction $\vec{t}_p$
		of $f_X\circ\sigma$ and the 
		length $\lambda$, we can control the place of the center
		of the tripod, and hence the shape of 
		$\mathfrak{a}_0\vee\mathfrak{a}_1\vee\mathfrak{a}_2$.
		In fact, there is a unique way to assign 
		$\partial$-framings of $\mathfrak{a}_i$ so that
		$\mathfrak{a}_{i,i+1}$ are all consecutive for $i\in\ZZ_3$.
		The other two constructions in dimension $0$ and $1$ can be 
		understood in a similar fashion.
		
		Using Simplex Filling of dimension $1$ and
		Free Loop Reduction, we can construct
		\begin{itemize}
			\item the reduced concatenation of a chain of $\partial$-framed segments; and
			\item the cyclically reduced concatenation of a cycle of $\partial$-framed segments.
		\end{itemize}
		We consider these to be carried by $1$-chains and $1$-cycles of
		a given constructible extension of $X$. By iterately applying Simplex Filling
		for a finite number of times,
		we can obtain the reduced concatenation of any chain
		of $\partial$-framed segments, or obtain
		a cycle of a single $\partial$-framed segment 
		from any cycle of $\partial$-framed segments.
		In the latter case,
		we can then obtain the cyclically reduced concatenation
		by further applying Free Loop Reduction once.
				
		In Subsection \ref{Subsec-derivedConstructions},
		we will describe more constructions of
		$(R,\epsilon)$-panted surfaces $j:F\to M$ 
		with the boundary prescribed in terms of 
		$\partial$-framed segments under various hypotheses.
		To translate any description there 
		into one in terms of partially-$\Delta$ spaces,
		the instruction is as follows.
		The hypothesis describes a partially-$\Delta$ space 
		$(Z,Z_\Delta,f_Z)$. Namely, $Z_\Delta$ is considered to be
		a $1$-complex such that each $\partial$-framed segment 
		in the hypothesis corresponds to
		an oriented edge, and any two edges have identified endpoints
		if and only if the corresponding $\partial$-framed segments are declared
		to be nearly consecutive; $Z$ is the same as $Z_\Delta$; 
		and $f_Z$ is considered to be
		the obvious map that send any $1$-simplex to the carrier segment
		of its defining $\partial$-framed segment.
		In the same fashion, the recipe of the construction applies the
		axioms of the construction (Definition \ref{axiomsOfConstructions})
		step by step, by indicating at each step
		that an auxiliary point, segment, or $\partial$-framed
		segment should be drawn, 
		so the procedure gives rise to an extension $X$ of $Z$.
		The result of the construction can then be translated in
		terms of a partially-$\Delta$
		space $(F,\partial F,j)$ constructible from $Z$, 
		where $\partial F$ has a preferred $1$-complex structure 
		since it is prescribed
		by the $\partial$-framed segments from the hypothesis.
		In fact, one can always write down the partially combinatorial
		map $\psi:F\to X$ explicitly,
		where $X$ is the extension of $Z$ from the recipe of the construction.

	\subsection{Derived constructions}\label{Subsec-derivedConstructions}
		In this subsection, we exhibit several constructions
		of $(R,\epsilon)$-panted surfaces using $\partial$-framed
		segments that will be applied in the rest of this paper.
		These constructions are all \emph{derived} 
		from the basic constructions listed
		in Definition \ref{axiomsOfConstructions}, and
		are all \emph{definite} in the sense that
		the number of $\partial$-framed segments involved
		are universally	bounded.

		Throughout this subsection, we assume $M$ to be an oriented closed
		hyperbolic $3$-manifold, and
		\begin{itemize}
			\item $(L,\delta)$ is any pair of positve constants
			where $\delta\leq1$ and $L\geq-2\log\delta+10\log2$,
			and $L$ satisfies the conclusion of the Connection Principle
			(Lemma \ref{connectionPrinciple}) with respect to $\delta$ and $M$.
			\item $(R,\epsilon)$ is any pair of positive constants where $\epsilon\leq 1/\sqrt{2}$.
		\end{itemize}
		The constants have been chosen according to the Length and Phase Formula 
		(Definition \ref{lengthAndPhaseFormula} and Remark \ref{remarkLengthAndPhaseFormula}).
		For each of the constructions of this subsection,
		we will need additional assumptions comparing $\epsilon$ and $\delta$.
		However, for later applications, it will be convenient to
		adopt a uniform requirement:
		\begin{itemize}
			\item $\epsilon\geq \delta\times 10^3$.
		\end{itemize}


		\subsubsection{Splitting}
		The splitting construction below
		gives rise to a nearly regular pair of pants by adding
		a bisecting segment to a nearly purely hyperbolic curve. 
		The reader should compare it with \cite[Lemma 3.2 and Remark]{KM-Ehrenpreis}.
		
		\begin{construction}[Splitting]\label{splitting}
			Suppose $\epsilon\geq10\delta$.
			Suppose that $[\mathfrak{s}\mathfrak{s}']$ 
			is an $(L,\delta)$-tame bigon with the length and phase
			of $\mathfrak{s}$ and $\mathfrak{s}'$
			$\delta$-close to each other, and that
			$[\mathfrak{s}\mathfrak{s}']\in\ocurves_{R,\epsilon}$.
			Then a pair of pants $\Pi\in\opants_{R,\epsilon}$ can be constructed, 
			with one cuff $[\mathfrak{s}\mathfrak{s}']\in\ocurves_{R,\epsilon}$, 
			and with the other two cuffs in
			$\ocurves_{R,10\delta}$.
		\end{construction}	
		
		\begin{proof}
			By the Connection Principle (Lemma \ref{connectionPrinciple}),
			draw an oriented $\partial$-framed segment 
			$\mathfrak{m}$ from $p_\ter(\mathfrak{s})$
			to $p_\ini(\mathfrak{s})$ as follows:
			\begin{itemize}
				\item The initial and terminal directions are
				$\delta$-close to $\vec{t}_\ter(\mathfrak{s})\times\vec{n}_\ter(\mathfrak{s})$ and 
				$-\vec{t}_\ini(\mathfrak{s})\times\vec{n}_\ini(\mathfrak{s})$ respectively,
				and the initial and terminal framings are $\delta$-close to 
				$\vec{n}_\ter(\mathfrak{s})$ and $\vec{n}_\ini(\mathfrak{s})$ respectively,
				and the length and phase are $\delta$-close to 
				$R-\ell(\mathfrak{s})+2I(\frac\pi2)$ and $-\varphi(\mathfrak{s})$ respectively.
			\end{itemize}
			
			\begin{figure}[htb]
			\centering
			\includegraphics{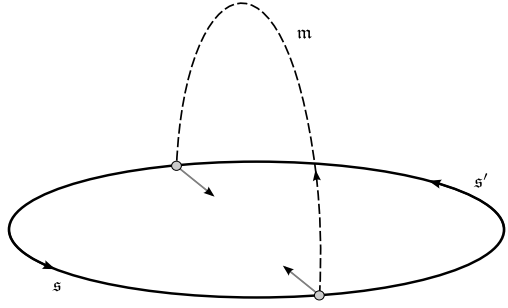}
			\caption{Splitting}\label{figSplitting}
		\end{figure}
			
			Then a pair of pants $\Pi\in\opants_{R,\epsilon}$ can be constructed,
			such that one cuff of $\Pi$ is $[\mathfrak{s}\mathfrak{s}']\in\ocurves_{R,\epsilon}$,
			and that the other two cuffs are
			$\overline{[\mathfrak{s}\mathfrak{m}]},\,
			\overline{[\bar{\mathfrak{m}}\mathfrak{s}']}\in\ocurves_{R,10\delta}$.
			The verification follows 
			from the Length and Phase Formula (Lemma \ref{lengthAndPhaseFormula}).		
		\end{proof}

		\subsubsection{Swapping}
		The swapping construction below allows us to exchange the arcs
		of two nearly purely hyperbolic curves if they fellow travel
		near a pair of common points, as long as the result are still
		purely hyperbolic curves. The reader should compare it
		with the Geometric Square Lemma \cite[Lemma 5.4]{KM-Ehrenpreis}.
		Our construction simplifies some of the arguments there.
		
		\begin{construction}[Swapping]\label{swapping}
			Suppose $\epsilon\geq\delta\times10^2$. Suppose that $[\mathfrak{a}\mathfrak{b}]$ and 
			$[\mathfrak{a}'\mathfrak{b}']$ is a $(10L,\delta)$-tame swap pair
			of bigons,
			and that $[\mathfrak{a}\mathfrak{b}],\,[\mathfrak{a}'\mathfrak{b}'],\,
			\overline{[\mathfrak{a}\mathfrak{b}']},\,
			\overline{[\mathfrak{a}'\mathfrak{b}]}\in\ocurves_{R,\epsilon}$.
			Then
			an oriented connected compact
			$(R,\epsilon)$-panted surface
			$F$ can be constructed, with exactly four boundary components
			$[\mathfrak{a}\mathfrak{b}],\,[\mathfrak{a}'\mathfrak{b}'],\,
			\overline{[\mathfrak{a}\mathfrak{b}']},\,
			\overline{[\mathfrak{a}'\mathfrak{b}]}$.
		\end{construction}
		
		\begin{proof}
			We first prove a preliminary case when the joint points are $\delta$-closely
			antipodal, and then derive the general case from the preliminary case.
		
			\medskip\noindent\textbf{Step 1}.
			Suppose $\epsilon\geq10\delta$.
			Suppose $[\mathfrak{a}\mathfrak{b}]$ and 
			$[\mathfrak{a}'\mathfrak{b}']$ are a $(L,\delta)$-tame swap pair
			of bigons,
			and that $[\mathfrak{a}\mathfrak{b}],\,[\mathfrak{a}'\mathfrak{b}'],\,
			\overline{[\mathfrak{a}\mathfrak{b}']},\,
			\overline{[\mathfrak{a}'\mathfrak{b}]}\in\ocurves_{R,10\delta}$.
			Moreover, suppose that the length and phase of $\mathfrak{a}$, $\mathfrak{b}$,
			$\mathfrak{a}'$, $\mathfrak{b}'$ are all $\delta$-close to each other.
			We construct an oriented connected compact
			$(R,\epsilon)$-panted surface
			$F$ with exactly four boundary components
			$[\mathfrak{a}\mathfrak{b}],\,[\mathfrak{a}'\mathfrak{b}'],\,
			\overline{[\mathfrak{a}\mathfrak{b}']},\,
			\overline{[\mathfrak{a}'\mathfrak{b}]}$ as follows.
			
			By the Connection Principle (Lemma \ref{connectionPrinciple}),
			draw an oriented $\partial$-framed segment 
			$\mathfrak{m}$ from $p_\ter(\mathfrak{a})$
			to $p_\ini(\mathfrak{a})$ as follows:
			\begin{itemize}
				\item The initial and terminal directions are
				$\delta$-close to $\vec{t}_\ter(\mathfrak{a})\times\vec{n}_\ter(\mathfrak{a})$ and 
				$-\vec{t}_\ini(\mathfrak{a})\times\vec{n}_\ini(\mathfrak{a})$ respectively,
				and the initial and terminal framings are $\delta$-close to 
				$\vec{n}_\ter(\mathfrak{a})$ and $\vec{n}_\ini(\mathfrak{a})$ respectively,
				and the length and phase are $\delta$-close to 
				$R-\ell(\mathfrak{a})+2I(\frac\pi2)$ and $-\varphi(\mathfrak{a})$ respectively.
			\end{itemize}
			With $\mathfrak{a}^\circ$ standing for $\mathfrak{a}$ or $\mathfrak{a}'$,
			and $\mathfrak{b}^\circ$ standing for
			$\mathfrak{b}$ or $\mathfrak{b}'$,
			there is a unique pair of pants $\Pi_{\mathfrak{a}^\circ,\mathfrak{b}^\circ}\in\opants_{R,\epsilon}$ 
			determined by its cuffs $[\mathfrak{a}^\circ\mathfrak{b}^\circ],\,
			\overline{[\mathfrak{a}^\circ\mathfrak{m}]},
			\overline{[\mathfrak{b}^\circ\bar{\mathfrak{m}}]}
			\in\ocurves_{R,\epsilon}$. Note that 
			the curve $\overline{[\mathfrak{a}^\circ\mathfrak{m}]}$			
			appears in exactly two pairs of pants as a cuff, and that the same holds
			for $\overline{[\mathfrak{b}^\circ\bar{\mathfrak{m}}]}$.
			Thus the four pairs of pants
			$\Pi_{\mathfrak{a},\mathfrak{b}}, \Pi_{\mathfrak{a}',\mathfrak{b}'}, 
			\overline{\Pi_{\mathfrak{a},\mathfrak{b}'}}, \overline{\Pi_{\mathfrak{a}',\mathfrak{b}}}\in\opants_{R,\epsilon}$
			can be glued along these cuffs, yielding the desired $(R,\epsilon)$-panted surface
			$F$, which is a torus with four holes.
			
			\begin{figure}[htb]
				\centering
				\includegraphics{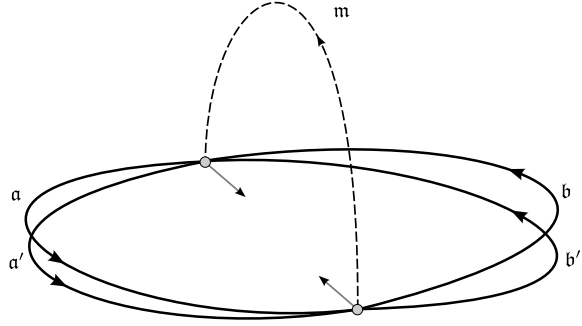}
				\caption{Swapping, Step 1}\label{figSwapping1}
			\end{figure}
			
			\medskip\noindent\textbf{Step 2}. 
			For $\epsilon\geq\delta\times10^2$, suppose that $[\mathfrak{a}\mathfrak{b}]$ and 
			$[\mathfrak{a}'\mathfrak{b}']$ is a $(10L,\delta)$-tame swap pair
			of bigons.
			Suppose $[\mathfrak{a}\mathfrak{b}],\,[\mathfrak{a}'\mathfrak{b}'],\,
			\overline{[\mathfrak{a}\mathfrak{b}']},\,
			\overline{[\mathfrak{a}'\mathfrak{b}]}\in\ocurves_{R,10\delta}$.
			We construct an oriented connected compact
			$(R,\epsilon)$-panted surface
			$F$ with exactly four boundary components
			$[\mathfrak{a}\mathfrak{b}]$, $[\mathfrak{a}'\mathfrak{b}']$,
			$\overline{[\mathfrak{a}\mathfrak{b}']}$,
			$\overline{[\mathfrak{a}'\mathfrak{b}]}$.

			By edge division, we may decompose $\mathfrak{a}$ as
			the concatenation $\mathfrak{a}_-\mathfrak{a}_+$ of
			two consecutive $\partial$-framed subsegments of equal length and
			phase, and similarly decompose $\mathfrak{b}$, $\mathfrak{a}'$,
			$\mathfrak{b}'$ as $\mathfrak{b}_-\mathfrak{b}_+$, $\mathfrak{a}'_-\mathfrak{a}'_+$,
			$\mathfrak{b}'_-\mathfrak{b}'_+$.
			Moreover, we may require the length and phase of $\mathfrak{a}_\pm$ to be $\delta$-close
			to that of $\mathfrak{a}'_\pm$, and the length and phase of $\mathfrak{b}_\pm$ to
			be $\delta$-close to $\frac{R}2-\ell(\mathfrak{a}_\mp)$ and $-\varphi(\mathfrak{a}_\mp)$
			respectively, and the length and phase of $\mathfrak{b}'_\pm$
			$\delta$-close to $\frac{R}2-\ell(\mathfrak{a}'_\mp)$ and $-\varphi(\mathfrak{a}'_\mp)$
			respectively. Choose a pair of auxiliary points $p,q\in M$. Draw five auxiliary $\partial$-framed
			segments $\mathfrak{r}$, $\mathfrak{s}_a$, $\mathfrak{s}'_a$, $\mathfrak{s}_b$, $\mathfrak{s}'_b$ as
			follows.
			\begin{itemize}
				\item The initial and terminal endpoints of $\mathfrak{r}$ 
				are $p$ and $q$ respectively, 
				and the length and phase of $\mathfrak{r}$ are $\delta$-close to $\frac{R}2-4L$
				and $0$ respectively.
				\item The initial and terminal endpoints of $\mathfrak{s}_a$ are $p_\ini(\mathfrak{a}_+)$ and
				$p$ respectively, and the length and phase of $\mathfrak{s}_a$ is $\delta$-close to
				$2L+I(\frac\pi2)$ and $-\varphi(\mathfrak{a}_+)$ respectively. The same holds for $\mathfrak{s}'_a$
				with $\mathfrak{a}_+$ replaced by $\mathfrak{a}'_+$.
				\item The initial and terminal endpoints of $\mathfrak{s}_b$ are $q$ and
				$p_\ini(\mathfrak{b}_+)$ respectively, and the length and phase of $\mathfrak{s}_b$ is $\delta$-close to
				$2L+I(\frac\pi2)$ and $-\varphi(\mathfrak{b}_+)$ respectively. The same holds for $\mathfrak{s}'_b$
				with $\mathfrak{b}_+$ replaced by $\mathfrak{b}'_+$.
			\end{itemize}
			
			\begin{figure}[htb]
				\centering
				\includegraphics{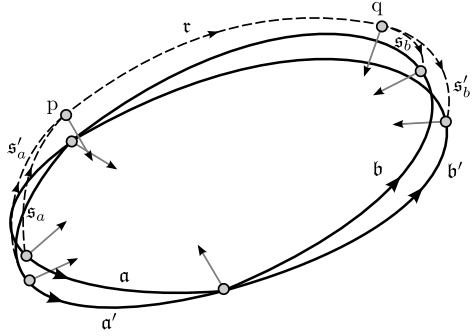}
				\caption{Swapping, Step 2}\label{figSwapping2}
			\end{figure}
			
			With $(\mathfrak{a}_-\mathfrak{s}_a)^\circ$ standing for 
			$\mathfrak{a}_-\mathfrak{s}$ or $\mathfrak{a}'_-\mathfrak{s}'_a$,
			and similarly for the notations
			$(\bar{\mathfrak{s}}_a\mathfrak{a}_+)^\circ$,
			$(\mathfrak{b}_-\mathfrak{s}_b)^\circ$, $(\bar{\mathfrak{s}}_b\mathfrak{b}_+)^\circ$,
			$\mathfrak{a}^\circ$, and $\mathfrak{b}^\circ$,
			there are four pair of pants 
			$\Pi_{\mathfrak{a}^\circ,\mathfrak{b}^\circ}\in\opants_{R,10\delta}$,
			each determined by its cuffs $[\mathfrak{a}^\circ\mathfrak{b}^\circ],\,
			\overline{[(\bar{\mathfrak{s}}_a\mathfrak{a}_+)^\circ(\mathfrak{b}_-\mathfrak{s}_b)^\circ\bar{\mathfrak{r}}]},
			\overline{[(\mathfrak{a}_-\mathfrak{s}_a)^\circ\mathfrak{r}(\bar{\mathfrak{s}}_b\mathfrak{b}_+)^\circ]}
			\in\ocurves_{R,\epsilon}$. Decompose $\mathfrak{r}$ as the concatenation
			$\mathfrak{r}_a\mathfrak{r}_b$ of two consecutive $\partial$-framed segments
			of phase $\delta$-close to $0$, such that $\ell(\mathfrak{r}_a)$ is $\delta$-close
			to $\frac{R}2-\ell(\mathfrak{a}_+)+2L$ and $\ell(\mathfrak{r}_b)$ is $\delta$-close
			to $\frac{R}2-\ell(\mathfrak{b}_+)+2L$. Note that the chains and cycles involved are
			all $(L,10\delta)$-tame and $(10\delta)$-consecutive since 
			$[\mathfrak{a}\mathfrak{b}]$ and 
			$[\mathfrak{a}'\mathfrak{b}']$ are $(10L,\delta)$-tame bigons.
			
			Since $\epsilon\geq\delta\times10^2$,  
			the preliminary case of Step 1 applies to the $(10L,10\delta)$-swap pair
			$[(\bar{\mathfrak{r}}_a\bar{\mathfrak{s}}_a\mathfrak{a}_+)(\mathfrak{b}_-\mathfrak{s}_b\bar{\mathfrak{r}}_b)]$
			and 
			$[(\bar{\mathfrak{r}}_a\bar{\mathfrak{s}}'_a\mathfrak{a}'_+)(\mathfrak{b}'_-\mathfrak{s}'_b\bar{\mathfrak{r}}_b)]$.
			Thus there is an $(R,\epsilon)$-panted four-hole torus $E_{a_+b_-}$ with boundary components
			$[(\bar{\mathfrak{r}}_a\bar{\mathfrak{s}}_a\mathfrak{a}_+)(\mathfrak{b}_-\mathfrak{s}_b\bar{\mathfrak{r}}_b)]$,
			$[(\bar{\mathfrak{r}}_a\bar{\mathfrak{s}}'_a\mathfrak{a}'_+)(\mathfrak{b}'_-\mathfrak{s}'_b\bar{\mathfrak{r}}_b)]$,
			$\overline{[(\bar{\mathfrak{r}}_a\bar{\mathfrak{s}}_a\mathfrak{a}_+)(\mathfrak{b}'_-\mathfrak{s}'_b\bar{\mathfrak{r}}_b)]}$,
			$\overline{[(\bar{\mathfrak{r}}_a\bar{\mathfrak{s}}'_a\mathfrak{a}'_+)(\mathfrak{b}_-\mathfrak{s}_b\bar{\mathfrak{r}}_b)]}$.
			Similarly, there is an $(R,\epsilon)$-panted four-hole torus $E_{b_+a_-}$
			with boundary components
			$[(\mathfrak{a}_-\mathfrak{s}_a\mathfrak{r}_a)(\mathfrak{r}_b\mathfrak{s}_b\mathfrak{b}_-)]$,
			$[(\mathfrak{a}'_-\mathfrak{s}'_a\mathfrak{r}_a)(\mathfrak{r}_b\mathfrak{s}'_b\mathfrak{b}'_-)]$,
			$\overline{[(\mathfrak{a}_-\mathfrak{s}_a\mathfrak{r}_a)(\mathfrak{r}_b\mathfrak{s}'_b\mathfrak{b}'_-)]}$,
			$\overline{[(\mathfrak{a}'_-\mathfrak{s}'_a\mathfrak{r}_a)(\mathfrak{r}_b\mathfrak{s}_b\mathfrak{b}_-)]}$.
			Glue up $\Pi_{\mathfrak{a},\mathfrak{b}}$, $\Pi_{\mathfrak{a}',\mathfrak{b}'}$,
			$\overline{\Pi_{\mathfrak{a},\mathfrak{b}'}}$, $\overline{\Pi_{\mathfrak{a}',\mathfrak{b}}}$, $E_{a_+b_-}$,
			and $E_{b_+a_-}$ along their oppositely oriented common boundary components.
			The result is an $(R,\epsilon)$-panted surface $F$ with exactly four boundary
			components $[\mathfrak{a}\mathfrak{b}]$, $[\mathfrak{a}'\mathfrak{b}']$,
			$\overline{[\mathfrak{a}\mathfrak{b}']}$,
			$\overline{[\mathfrak{a}'\mathfrak{b}]}$, as desired.
		\end{proof}

		\subsubsection{Rotation} 
		When two nearly regular tripods of opposite
		chiralities have legs almost opposite to the ones of each other,
		one can naturally build a nearly regular pair of pants spined
		on the union of the two tripods. This will be 
		the first statement of the rotation construction below.
		However, there is another case when the tripods have identical chirality.
		Thus we have two rotation constructions.
		Note that in the identical chirality case,
		we take two copies of the curves
		arising from the construction. This turns out to be necessary
		due to Theorem \ref{theoremPantedCobordism}.
				
		\begin{construction}[Rotation]\label{rotationConstruction}
			Suppose $\mathfrak{a}_0\vee\mathfrak{a}_1\vee\mathfrak{a}_2$ and 
			$\mathfrak{b}_0\vee\mathfrak{b}_1\vee\mathfrak{b}_2$ are a $(100L,\delta)$-tame rotation
			pair of  tripods. Suppose 
			$[\mathfrak{a}_{i,i+1}\bar{\mathfrak{b}}_{j,j\pm1}]\in\ocurves_{R,\epsilon}$
			for $i,j\in\ZZ_3$. Then an oriented connected compact
			$(R,\epsilon)$-panted surface $F$ can be constructed satisfying the following.
			\begin{enumerate}
				\item 
				If $\mathfrak{a}_0\vee\mathfrak{a}_1\vee\mathfrak{a}_2$ and 
				$\mathfrak{b}_0\vee\mathfrak{b}_1\vee\mathfrak{b}_2$ are of opposite charalities,
				then $F$ is a pair of pants $\Pi\in\opants_{R,\epsilon}$
				with cuffs $[\mathfrak{a}_{i,i+1}\bar{\mathfrak{b}}_{i,i+1}]$, for $i\in\ZZ_3$.
				\item Suppose $\epsilon\geq\delta\times10^3$.
				If $\mathfrak{a}_0\vee\mathfrak{a}_1\vee\mathfrak{a}_2$ and 
				$\mathfrak{b}_0\vee\mathfrak{b}_1\vee\mathfrak{b}_2$ are of identical charality,
				then $F$ has exactly six boundary components,
				namely, two copies of each $[\mathfrak{a}_{i,i+1}\bar{\mathfrak{b}}_{i,i+1}]$,
				for $i\in\ZZ_3$.
			\end{enumerate}				
		\end{construction}

		\begin{proof}
			To prove the statement (1),
			suppose that $\mathfrak{a}_0\vee\mathfrak{a}_1\vee\mathfrak{a}_2$ and 
			$\mathfrak{b}_0\vee\mathfrak{b}_1\vee\mathfrak{b}_2$ are of opposite chiralities. This is the simple case
			because $F$ can be naturally chosen as the pair of pants $\Pi\in\opants_{R,\epsilon}$ with cuffs 
			$[\mathfrak{a}_{i,i+1}\bar{\mathfrak{b}}_{i,i+1}]$ for $i\in\ZZ_3$. The
			spine of $\Pi$ is the figure-$\theta$ graph which is approximately the union of the carrier segments
			of $\mathfrak{a}_i\bar{\mathfrak{b}}_i$ for $i\in\ZZ_3$. 
			
			To prove the statement (2), suppose that $\mathfrak{a}_0\vee\mathfrak{a}_1\vee\mathfrak{a}_2$ and 
			$\mathfrak{b}_0\vee\mathfrak{b}_1\vee\mathfrak{b}_2$ are of identical chirality.
			Considering the framing flipping $\mathfrak{a}^*_0\vee\mathfrak{a}^*_1\vee\mathfrak{a}^*_2$ and 
			$\mathfrak{b}^*_0\vee\mathfrak{b}^*_1\vee\mathfrak{b}^*_2$ instead if necessary,
			we may assume that the tripods are both right-handed without loss of generality.			

			\medskip\noindent\textbf{Step 1}. 
			Suppose $\epsilon\geq \delta\times10^3$.
			Suppose $\mathfrak{a}_0\vee\mathfrak{a}_1\vee\mathfrak{a}_2$ and 
			$\mathfrak{b}_0\vee\mathfrak{b}_1\vee\mathfrak{b}_2$ are a $(10L,10\delta)$-tame rotation
			pair of right-handed tripods, and that
			$[\mathfrak{a}_{i,i+1}\bar{\mathfrak{b}}_{i,i+1}]\in\ocurves_{R,\epsilon}$
			for $i\in\ZZ_3$. Moreover, suppose that $\mathfrak{b}_0\vee\mathfrak{b}_1\vee\mathfrak{b}_2$
			can be written as $\mathfrak{c}_0\mathfrak{r}\vee\mathfrak{c}_1\mathfrak{r}\vee\mathfrak{c}_2\mathfrak{r}$
			where $\mathfrak{c}_i,\mathfrak{r}$ is a $\delta$-consecutive $(10L,\delta)$-tame
			chain for each $i\in\ZZ_3$ and where $\mathfrak{c}_0\vee\mathfrak{c}_1\vee\mathfrak{c}_2$
			is $(l_b-\ell(\mathfrak{r}),\delta)$-nearly regular.
			We construct an $(R,\epsilon)$-panted surface $F$
			with exactly six boundary components, which are
			two copies of each $[\mathfrak{a}_{i,i+1}\bar{\mathfrak{b}}_{i,i+1}]$
			for $i\in\ZZ_3$. 
			 
			Draw an auxiliary oriented $\partial$-oriented segment $\mathfrak{s}$ satisfying the following.
			\begin{itemize}
				\item The length and phase of $\mathfrak{s}$
				is $\delta$-close to $\ell(\mathfrak{r})$ and $0$
				respectively. The initial and terminal endpoints of $\mathfrak{s}$ coincide
				with $p_\ini(\mathfrak{r})$ and $p_\ter(\mathfrak{r})$ 
				respectively. The initial and terminal
				directions of $\mathfrak{s}$ are $\delta$-close to 
				$\vec{t}_\ini(\mathfrak{r})$ and $\vec{t}_\ter(\mathfrak{r})$
				respectively. The initial framing of $\mathfrak{s}$ is $\delta$-close
				to $-\vec{n}_\ini(\mathfrak{r})$, and the terminal framing of $\mathfrak{s}$
				is $\delta$-close to $\vec{n}_\ter(\mathfrak{r})$.
			\end{itemize}
			
			\begin{figure}[htb]
				\centering
				\includegraphics{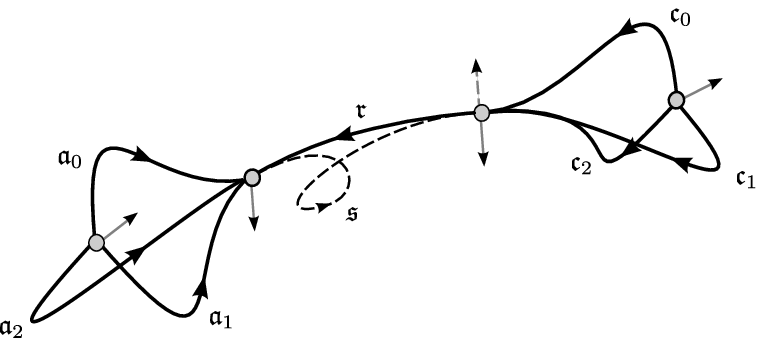}
				\caption{Rotation with identical chirality, Step 1}\label{figRotation1}
			\end{figure}
			
			It follows that $\mathfrak{c}^*_0\mathfrak{s}\vee\mathfrak{c}^*_1\mathfrak{s}\vee\mathfrak{c}^*_2\mathfrak{s}$
			is an $(l_b-\ell(\mathfrak{r}),10\delta)$-nearly regular left-handed tripod,
			where $\mathfrak{c}^*_i,\mathfrak{s}$ is a $\delta$-consecutive $(10L,\delta)$-tame
			chain for each $i\in\ZZ_3$. 
				
			Observe the following $(R,\epsilon)$-panted surfaces. 
			Since $\epsilon\geq\delta\times10^3$, by swapping (Construction \ref{swapping})
			of $(10L,10\delta)$-tame swap pairs,
			for each $i\in\ZZ_3$,
			there is an $(R,\epsilon)$-panted surface $E_i$ with boundary components
			the curves $[\mathfrak{a}_{i,i+1}\bar{\mathfrak{r}}\,\bar{\mathfrak{c}}_{i,i+1}\mathfrak{r}]$,
			$[\bar{\mathfrak{a}}_{i,i+1}\,\bar{\mathfrak{s}}\mathfrak{c}^*_{i,i+1}\mathfrak{s}]$,
			$\overline{[\bar{\mathfrak{a}}_{i,i+1}\bar{\mathfrak{r}}\,\bar{\mathfrak{c}}_{i,i+1}\mathfrak{r}]}$,
			$\overline{[\mathfrak{a}_{i,i+1}\bar{\mathfrak{s}}\mathfrak{c}^*_{i,i+1}\mathfrak{s}]}$
			in $\ocurves_{R,\epsilon}$; 
			also by swapping (Construction \ref{swapping}), for each $i\in\ZZ_3$,
			there is an $(R,\epsilon)$-panted surface $E'_i$ with boundary components
			the curves $[\mathfrak{a}_{i,i+1}\bar{\mathfrak{r}}\,\bar{\mathfrak{c}}_{i,i+1}\mathfrak{r}]$,
			$[\bar{\mathfrak{a}}^*_{i,i+1}\,\bar{\mathfrak{s}}^*\mathfrak{c}_{i,i+1}\mathfrak{s}^*]$,
			$\overline{[\mathfrak{a}_{i,i+1}\bar{\mathfrak{r}}\mathfrak{c}_{i,i+1}\mathfrak{r}]}$,
			$\overline{[\bar{\mathfrak{a}}^*_{i,i+1}\,\bar{\mathfrak{s}}^*\bar{\mathfrak{c}}_{i,i+1}\mathfrak{s}^*]}$
			in $\ocurves_{R,\epsilon}$; 
			applying rotation in the opposite chirality case (Statement (1))
			to the $(10L,10\delta)$-tame rotation pair 
			$\mathfrak{c}^*_0\mathfrak{s}\vee\mathfrak{c}^*_1\mathfrak{s}\vee\mathfrak{c}^*_2\mathfrak{s}$
			and $\mathfrak{a}_0\vee\mathfrak{a}_1\vee\mathfrak{a}_2$,
			there is a pair of pants $\Pi\in\opants_{R,\epsilon}$ with boundary components 
			the curves 
			$\overline{[\bar{\mathfrak{a}}^*_{i,i+1}\bar{\mathfrak{s}}^*\mathfrak{c}_{i,i+1}\mathfrak{s}^*]}$
			in $\ocurves_{R,\epsilon}$, where $i$ runs over $\ZZ_3$;
			for another copy $\Pi'$ of $\Pi$, we may rewrite the boundary components
			of $\Pi'$ as 
			$\overline{[\bar{\mathfrak{a}}_{i,i+1}\bar{\mathfrak{s}}\mathfrak{c}^*_{i,i+1}\mathfrak{s}]}$
			in $\ocurves_{R,\epsilon}$, where $i$ runs over $\ZZ_3$.
			Furthermore, observe the following common boundary components of opposite orientations.
			The third curve of $\partial E_i$ is the orientation reversal
			of the third curve of $\partial E'_i$; the fourth curve of $\partial E$ can be rewritten as 
			$\overline{[\mathfrak{a}^*_{i,i+1}\bar{\mathfrak{s}}^*\mathfrak{c}_{i,i+1}\mathfrak{s}^*]}$,
			which is clearly the orientation reversal of the fourth curve of $\partial E'$; 
			the orientation reversal of each curve
			of $\partial\Pi$ appears exactly once as the second curve
			in $\partial E_i$, and similarly for $\partial\Pi'$ and $\partial E'_i$.
			Gluing the $(R,\epsilon)$-panted surfaces $E_i$, $E'_i$, $\Pi$, and $\Pi'$
			along these oppositely oriented common boundary components,
			the result is an oriented connected compact surface $F$
			with exactly six boundary components, namely,
			two copies for each $[\mathfrak{a}_{i,i+1}\bar{\mathfrak{r}}\bar{\mathfrak{c}}_{i,i+1}\mathfrak{r}]$,
			where $i$ runs over $\ZZ_3$. Since
			$[\mathfrak{a}_{i,i+1}\bar{\mathfrak{r}}\,\bar{\mathfrak{c}}_{i,i+1}\mathfrak{r}]$
			equals $[\mathfrak{a}\bar{\mathfrak{b}}_{i,i+1}]$ under the additional assumption
			of the current step, the $(R,\epsilon)$-panted surface $F$ is as desired. 

			\medskip\noindent\textbf{Step 2}. 
			Suppose $\epsilon\geq\delta\times10^3$.
			Suppose $\mathfrak{a}_0\vee\mathfrak{a}_1\vee\mathfrak{a}_2$ and 
			$\mathfrak{b}_0\vee\mathfrak{b}_1\vee\mathfrak{b}_2$ are a $(10L,\delta)$-tame rotation
			pair of right-handed tripods, and that
			$[\mathfrak{a}_{i,i+1}\bar{\mathfrak{b}}_{i,i+1}]\in\ocurves_{R,\epsilon}$
			for $i\in\ZZ_3$.
			We construct an $(R,\epsilon)$-panted surface $F$
			with exactly six boundary components, which are
			two copies of each $[\mathfrak{a}_{i,i+1}\bar{\mathfrak{b}}_{i,i+1}]$
			for $i\in\ZZ_3$. 
			
			Since $\mathfrak{b}_0\vee\mathfrak{b}_1\vee\mathfrak{b}_2$ is $(l_b,\delta)$-nearly
			regular where $l_b\geq100L$,
			an auxiliary $(l_b,10\delta)$-nearly regular right-handed tripod
			$\mathfrak{b}'_0\vee\mathfrak{b}'_1\vee\mathfrak{b}'_2$ can be drawn so that
			$\mathfrak{a}_0\vee\mathfrak{a}_1\vee\mathfrak{a}_2$
			and $\mathfrak{b}'_0\vee\mathfrak{b}'_1\vee\mathfrak{b}'_2$ form a $(100L,10\delta)$-tame rotation pair,
			and that $\mathfrak{b}'_0\vee\mathfrak{b}'_1\vee\mathfrak{b}'_2$ satisfies the additional assumption
			of Step 1.

			\begin{figure}[htb]
				\centering
				\includegraphics{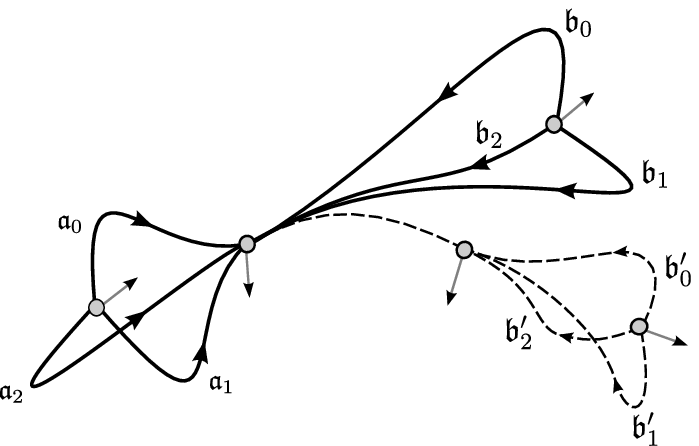}
				\caption{Rotation with identical chirality, Step 2}\label{figRotation2}
			\end{figure}
			
			Observe the following $(R,\epsilon)$-panted surfaces. Applying Step 1
			to $\mathfrak{a}_0\vee\mathfrak{a}_1\vee\mathfrak{a}_2$ and
			$\mathfrak{b}'_0\vee\mathfrak{b}'_1\vee\mathfrak{b}'_2$,
			there is an $(R,\epsilon)$-panted surface $F'$ with six boundary components,
			namely, two copies of each curve $[\mathfrak{a}_{i,i+1}\bar{\mathfrak{b}}'_{i,i+1}]$
			in $\ocurves_{R,\epsilon}$ for $i\in\ZZ_3$;
			by rotation in the opposite chirality case (Statement (1)), 
			applied to the $(10L,10\delta)$-tame rotation pair $\mathfrak{b}_0\vee\mathfrak{b}_1\vee\mathfrak{b}_2$
			and $\mathfrak{a}_0\vee\mathfrak{a}_{-1}\vee\mathfrak{a}_{-2}$,
			there is a pair of pants $\Pi\in\opants_{R,\epsilon}$ with boundary components
			$[\mathfrak{a}_{-i,-i-1}\bar{\mathfrak{b}}_{i,i+1}]$
			in $\ocurves_{R,\epsilon}$ for $i\in\ZZ_3$;
			also by rotation in the opposite chirality case (Statement (1)), 
			applied to the $(100L,10\delta)$-tame rotation pair $\mathfrak{b}'_0\vee\mathfrak{b}'_1\vee\mathfrak{b}'_2$
			and $\mathfrak{a}_0\vee\mathfrak{a}_{-1}\vee\mathfrak{a}_{-2}$,
			there is a pair of pants $\Pi'\in\opants_{R,\epsilon}$ 
			with boundary components
			$\overline{[\mathfrak{a}_{-i,-i-1}\bar{\mathfrak{b}}'_{i,i+1}]}$
			in $\ocurves_{R,\epsilon}$ for $i\in\ZZ_3$;
			since $\epsilon\geq\delta\times10^3$,
			by swapping (Construction \ref{swapping}) of $(100L,10\delta)$-tame
			swap pairs, for each $i\in\ZZ_3$, there
			is an $(R,\epsilon)$-panted surface $E_i$ with boundary components
			the curves 
			$[\mathfrak{a}_{i,i+1}\bar{\mathfrak{b}}_{i,i+1}]$
			$[\mathfrak{a}_{-i,-i-1}\bar{\mathfrak{b}}'_{i,i+1}]$
			$\overline{[\mathfrak{a}_{i,i+1}\bar{\mathfrak{b}}'_{i,i+1}]}$
			$\overline{[\mathfrak{a}_{-i,-i-1}\bar{\mathfrak{b}}_{i,i+1}]}$
			in $\ocurves_{R,\epsilon}$.
			Gluing the $(R,\epsilon)$-panted surfaces $F'$, and
			two copies of the $(R,\epsilon)$-panted surfaces
			$\Pi$, $\Pi'$, $E_0$, $E_1$, and $E_2$,
			along their oppositely oriented common boundary components,
			the result is an oriented connected compact surface $F$
			with exactly six boundary components, namely,
			two copies for each $[\mathfrak{a}_{i,i+1}\bar{\mathfrak{b}}_{i,i+1}]$
			for $i\in\ZZ_3$ as desired. This completes the proof 
			of the statement (2).	 
		\end{proof}

		\subsubsection{Antirotation} 
		The antirotation construction is a variation of rotation
		when we join the legs in an `unnatural' way. As before, 
		we have two cases depending on the chiralities
		of the pair of tripods in consideration.
		
		\begin{construction}[Antirotation]\label{antirotationConstruction}
			Suppose that $\mathfrak{a}_0\vee\mathfrak{a}_1\vee\mathfrak{a}_2$ and 
			$\mathfrak{b}_0\vee\mathfrak{b}_1\vee\mathfrak{b}_2$ are a $(100L,\delta)$-tame rotation
			pair of tripods, and that
			$[\mathfrak{a}_{i,i+1}\bar{\mathfrak{b}}_{j,j\pm1}]\in\ocurves_{R,\epsilon}$
			for $i,j\in\ZZ_3$.
			Then an oriented connected compact
			$(R,\epsilon)$-panted surface $F$ can be constructed satisfying the following.
			\begin{enumerate}
				\item
				Suppose $\epsilon\geq\delta\times10^3$. 
				If $\mathfrak{a}_0\vee\mathfrak{a}_1\vee\mathfrak{a}_2$ and 
				$\mathfrak{b}_0\vee\mathfrak{b}_1\vee\mathfrak{b}_2$ are of opposite charalities,
				then $F$ has exactly six boundary components,
				namely, two copies of each $[\mathfrak{a}_{i,i+1}\bar{\mathfrak{b}}_{i+1,i}]$
				in $\ocurves_{R,\epsilon}$, for $i\in\ZZ_3$.
				\item
				Suppose $\epsilon\geq\delta\times10^2$.
				If $\mathfrak{a}_0\vee\mathfrak{a}_1\vee\mathfrak{a}_2$ and 
				$\mathfrak{b}_0\vee\mathfrak{b}_1\vee\mathfrak{b}_2$ are of identical charality,
				then $F$ has exactly three boundary components
				namely, $[\mathfrak{a}_{i,i+1}\bar{\mathfrak{b}}_{i+1,i}]$
				in $\ocurves_{R,\epsilon}$, for $i\in\ZZ_3$.
			\end{enumerate}				
		\end{construction}
		
		\begin{proof}
			Because $[\mathfrak{a}_{01}\bar{\mathfrak{b}}_{10}]$ and $[\mathfrak{a}_{20}\bar{\mathfrak{b}}_{02}]$
			form a $(100L,\delta)$-tame swap pair, by swapping (Construction \ref{swapping}),
			there is an oriented connected compact $(R,\epsilon)$-panted surface $E$ 
			with exactly four boundary components 
			$[\mathfrak{a}_{01}\bar{\mathfrak{b}}_{10}]$,
			$[\mathfrak{a}_{20}\bar{\mathfrak{b}}_{02}]$,
			$\overline{[\mathfrak{a}_{01}\bar{\mathfrak{b}}_{02}]}$,
			$\overline{[\mathfrak{a}_{20}\bar{\mathfrak{b}}_{10}]}$.
			Thus, it suffices an oriented connected compact $(R,\epsilon)$-panted surface
			$F'$ satisfying the following. 
			\begin{enumerate}
				\item If $\mathfrak{a}_0\vee\mathfrak{a}_1\vee\mathfrak{a}_2$ and 
				$\mathfrak{b}_0\vee\mathfrak{b}_1\vee\mathfrak{b}_2$ are of opposite charalities,
				then $F'$ has exactly six boundary components,
				namely, two copies of $[\mathfrak{a}_{01}\bar{\mathfrak{b}}_{02}]$,
				$[\mathfrak{a}_{12}\bar{\mathfrak{b}}_{21}]$, and $[\mathfrak{a}_{20}\bar{\mathfrak{b}}_{10}]$.
				\item If $\mathfrak{a}_0\vee\mathfrak{a}_1\vee\mathfrak{a}_2$ and 
				$\mathfrak{b}_0\vee\mathfrak{b}_1\vee\mathfrak{b}_2$ are of identical charality,
				then $F'$ has exactly three boundary components
				namely, $[\mathfrak{a}_{01}\bar{\mathfrak{b}}_{02}]$,
				$[\mathfrak{a}_{12}\bar{\mathfrak{b}}_{21}]$, and $[\mathfrak{a}_{20}\bar{\mathfrak{b}}_{10}]$.
			\end{enumerate}
			In fact, this follows immediately from rotation (Construction \ref{rotationConstruction})
			with respect to the $(100L,\delta)$-tame rotation pair $\mathfrak{a}_0\vee\mathfrak{a}_1\vee\mathfrak{a}_2$ and 
			$\mathfrak{b}_0\vee\mathfrak{b}_2\vee\mathfrak{b}_1$.
			Note that the chirality of
			$\mathfrak{b}_0\vee\mathfrak{b}_2\vee\mathfrak{b}_1$
			is exactly opposite to that of $\mathfrak{b}_0\vee\mathfrak{b}_1\vee\mathfrak{b}_2$.
			This completes the proof.		
		\end{proof}

\section{Panted cobordism group}\label{Sec-pantedCobordismGroup}
	In this section, we introduce the $(R,\epsilon)$-panted cobordism
	group of oriented $(R,\epsilon)$-multicurves
	in an oriented closed hyperbolic $3$-manifold.
	This will serve as a correction theory which will reduce
	the relative case of Theorem \ref{homologyViaPants} (1) to the absolute case.
	
	Let $M$ be an oriented closed hyperbolic $3$-manifold.
	By an \emph{$(R,\epsilon)$-nearly hyperbolic multicurve}, or simply an
	\emph{$(R,\epsilon)$-multicurve},
	we mean a (possibly disconnected)
	nonempty immersed oriented closed $1$-submanifold of
	$M$ of which all the components are $(R,\epsilon)$-nearly hyperbolic curves.

	\begin{definition}\label{pantedCobordismGroup}
		An \emph{$(R,\epsilon)$-nearly regularly panted cobordism},
		or simply an \emph{$(R,\epsilon)$-panted cobordism},
		between two $(R,\epsilon)$-multicurves $L,L'$
		in $M$ is an $(R,\epsilon)$-panted subsurface $F$ such that 
		$\partial F$ is the disjoint union
		$L\sqcup\bar{L}'$, where $\bar{L}'$ denotes the orientation-reversal 
		of $L'$. 
		We say that $L,L'$ are \emph{$(R,\epsilon)$-panted cobordant}
		if there exists an $(R,\epsilon)$-panted cobordism between them.
		Assuming that being $(R,\epsilon)$-panted cobordant is an equivalence relation
		on the set of $(R,\epsilon)$-multicurves,
		we denote by $\ocobordism_{R,\epsilon}(M)$, or simply $\ocobordism_{R,\epsilon}$,
		the set of all $(R,\epsilon)$-panted cobordism classes 
		of $(R,\epsilon)$-multicurves,
		and by $[L]_{R,\epsilon}$		
		the cobordism class of any $(R,\epsilon)$-multicurve $L$.
	\end{definition}
	
	An $(R,\epsilon)$-panted cobordism is not necessarily connected or $\pi_1$-injective,
	(cf.~Definition \ref{pantedSubsurfaceDefinition}).
	For any universally small positive $\epsilon$ and 
	any sufficiently large positive $R$ depending only on $M$ and $\epsilon$,
	it will be verified that being $(R,\epsilon)$-panted cobordant is an equivalence relation for $(R,\epsilon)$-multicurves.
	Furthermore, $\ocobordism_{R,\epsilon}$ is an abelian
	group with the addition induced by the disjoint union operation between 
	$(R,\epsilon)$-multicurves and the inverse induced by the orientation reversion
	of $(R,\epsilon)$-multicurves, (Subsection \ref{Subsec-Phi}).
	
	\begin{theorem}\label{theoremPantedCobordism}
		Let $M$ be an oriented closed hyperbolic $3$-manifold.
		For any universally small positive $\epsilon$, 
		and any sufficiently large
		positive $R$ depending only on $M$ and $\epsilon$, 
		there is a canonical isomorphism 
			$$\Phi:\,\ocobordism_{R,\epsilon}(M)\,{\longrightarrow}\,H_1(\SO(M);\ZZ),$$
		where $\SO(M)$ denotes the bundle over $M$ 
		of special orthonormal frames with respect to 
		the orientation of $M$. Moreover,
		for all $[L]_{R,\epsilon}\in\ocobordism_{R,\epsilon}(M)$, the image of
		$\Phi([L]_{R,\epsilon})$ under the bundle projection
		is the homology class $[L]\in H_1(M;\ZZ)$.
	\end{theorem}
	
	\begin{remark}
		The last part of the statement implies that the integral module
		$\ocobordism_{R,\epsilon}(M)$ is equivalent
		to $H_1(\SO(M);\,\ZZ)$ as a split extension of
		$H_1(M;\,\ZZ)$ by $\ZZ_2$ (Subsection \ref{Subsec-Phi}).
		Therefore, the isomorphisms between
		$\ocobordism_{R,\epsilon}(M)$ and
		$H_1(\SO(M);\,\ZZ)$	that are extension equivalences
		are in natural bijection to $H^1(M;\,\ZZ_2)$,
		where the canonical isomorphism $\Phi$ corresponds to $0$.	
	\end{remark}
	
	In fact, one may require $\epsilon$ to be at most $10^{-2}$ and $R$
	to be at least $R(\epsilon,M)$ as provided by Lemma \ref{setupData},
	then the conclusion of Theorem \ref{theoremPantedCobordism} holds.
	
	The idea of Theorem \ref{theoremPantedCobordism}
	is developed from the (non-random) correction theory part of \cite{KM-Ehrenpreis}.
	In that paper, 
	the notion \emph{Good Pants Homology} of an oriented closed hyperbolic surface
	$S$ is informally introduced, and with our notations above,
	the Good Pants Homology of $S$ there means precisely
	the rational 
	$(R,\epsilon)$-panted cobordism group $\ocobordism_{R,\epsilon}(S)\otimes \QQ$,
	cf.~\cite[Definition 3.2]{KM-Ehrenpreis}.
	The proof of the Good Correction Theorem \cite[Theorem 3.2]{KM-Ehrenpreis}
	essentially implies that there is an isomorphism
	$\phi:\,\ocobordism_{R,\epsilon}(S)\otimes \QQ\to H_1(S;\QQ)$,
	such that $\phi([\gamma]_{R,\epsilon})$ equals $[\gamma]$.
	In fact, most part of the proof of 
	\cite[Theorem 3.2]{KM-Ehrenpreis}
	can be extended directly to the $3$-dimensional case,
	yielding an isomorphism 
	$\phi:\,\ocobordism_{R,\epsilon}(M)\otimes \QQ\to H_1(M;\QQ)$
	as above. Motivated by pushing the result to
	the integral coefficient case,
	the main innovation of Theorem \ref{theoremPantedCobordism}
	lies in the observation that in many senses, it should
	be more natural
	to replace $H_1(M;\ZZ)$	with $H_1(\SO(M);\ZZ)$.
	One reason, for instance, is that
	$(R,\epsilon)$-multicurves and $(R,\epsilon)$-pants admit 
	certain canonical lifts into $\SO(M)$
	because of their geometry; another reason
	is that technically, passing to $\SO(M)$ resolves
	certain ambiguity in the definition of the inverse of $\Phi$;
	the reader may also observe a vague analogy between the
	statement of Theorem \ref{theoremPantedCobordism}
	and the Thom--Pontrjagin correspondence in cobordism
	theory, thinking of $H_1(\SO(M);\ZZ)$
	as $\pi_1(\SO(M))$ modulo the action of conjugations.	
	
	The proof of Theorem \ref{theoremPantedCobordism} is organized slightly
	differently from the treatment of \cite{KM-Ehrenpreis}.
	Fix a basepoint $\pt$ of $M$ and a special orthonormal
	frame $\mathbf{e}\in\SO(M)|_\pt$ as a basepoint of $\SO(M)$.
	We will construct the homorphism $\Phi$ and a homomorphism $\Psi:\pi_1(\SO(M),\mathbf{e})\to \ocobordism_{R,\epsilon}$,
	which descends to be a homomorphism $\Psi^{\mathtt{ab}}:H_1(\SO(M);\ZZ)\to \ocobordism_{R,\epsilon}$
	by abelianization. We will verify that $\Phi\circ\Psi^{\mathtt{ab}}=\id$ and that $\Psi$ 
	surjects $\ocobordism_{R,\epsilon}$.
	This will imply that $\Phi$ is an isomorphism with the inverse $\Psi^{\mathtt{ab}}$.
	To briefly describe $\Phi$, note that any curve $\gamma\in\ocurves_{R,\epsilon}$ has
	a framing over its geodesic representative given by (nearly) parallel transporting a frame
	around $\gamma$. The \emph{canonical lift} $\hat{\gamma}:S^1\to \SO(M)$ of $\gamma$,
	well defined up to homotopy, is
	then a framing that differs from the the parallel-transportation framing by a loop 
	of matrices $S^1\to\SO(3)$ that represents the nontrivial element of $\pi_1(\SO(3))\cong\ZZ_2$
	(Definition \ref{canonicalLift}).
	The homomorphism $\Phi:\ocobordism_{R,\epsilon}\to H_1(\SO(M);\ZZ)$ will hence
	be uniquely defined so that $\Phi([\gamma]_{R,\epsilon})$ equals $[\hat{\gamma}]$.
	The definition of $\Psi$ will depend on a choice of a finite triangular generating set 
	$\hat{g}_1,\cdots,\hat{g}_s$ of $\pi_1(\SO(M),\mathbf{e})$, together with some other setup data.
	Here \emph{triangular} means that all the relators of length at most $3$ in
	the generating set gives rise to a finite presentation of $\pi_1(\SO(M),\mathbf{e})$.
	We will define $\Psi(\hat{g}_i)$ (or decorated with some
	setup data, $\Psi_h(\hat{g}_i)$) to be represented by some $(R,\epsilon)$-multicurve 
	constructed from $\hat{g}_i$, so that $[\Psi(\hat{g}_i)]$ is obviously equal
	to $[\hat{g}_i]$ in $H_1(\SO(M);\ZZ)$. With the derived constructions
	of Subsection \ref{Subsec-derivedConstructions}, we will verify that
	$\Psi(\hat{g}_i)+\Psi(\hat{g}_j)+\Psi(\hat{g}_k)=0$ in $\ocobordism_{R,\epsilon}$
	whenever there is a triangular relation $\hat{g}_i\hat{g}_j\hat{g}_k=\id$.
	It will follow that $\Psi:\pi_1(\SO(M),\mathbf{e})\to\ocobordism_{R,\epsilon}$
	is a well defined homomorphism. 
	We point out that the triangular generating set of $\pi_1(\SO(M),\mathbf{e})$,
	or more essentially, a finite triangular presentation of $\pi_1(M,\pt)$,
	is the source of topological information which makes the construction of $\Psi$ possible.
	This idea will be further investigated in Section \ref{Sec-pantifyingSecondHomologyClasses}
	when we pantify second homology classes.	
	
	The rest of this section is devoted to the proof of Theorem \ref{theoremPantedCobordism}.
	In Subsection \ref{Subsec-Phi}, we define the homomorphism $\Phi$. 
	In Subsection \ref{Subsec-theInverseOfPhi},
	we define the homomorphism $\Psi$ and
	verify that $\Psi^{\mathtt{ab}}$ is the inverse of $\Phi$.
	In Subsection \ref{Subsec-proofOfTheoremPantedCobordism}, we summarize
	the proof of Theorem \ref{theoremPantedCobordism}.
	
	Throughout this section, after fixing a basepoint $\pt$ of $M$,
	we will no longer distinguish a nontrivial element of $\pi_1(M,\pt)$
	from its pointed geodesic loop representative, 
	so it makes sense to speak of the length, or the initial
	or terminal direction of a nontrivial element in $\pi_1(M,\pt)$.

	\subsection{The homomorphism $\Phi$}\label{Subsec-Phi}
		In this subsection, we define the homomorphism $\Phi$. We also
		need to mention some basic facts
		about the $(R,\epsilon)$-panted cobordism gorup $\ocobordism_{R,\epsilon}$,
		and about the special orthonomal framing bundle	$\SO(M)$.
		
		The following assumptions of constants will be adopted throughout this subsection:
		\begin{itemize}
			\item $(L,\delta)$ is any pair of positve constants
			where $\delta\leq10^{-2}$ and $L\geq-2\log\delta+10\log2$,
			and $L$ satisfies the conclusion of the Connection Principle
			(Lemma \ref{connectionPrinciple}) with respect to $\delta$ and $M$.
			\item $(R,\epsilon)$ is any pair of positive constants where $\epsilon$ is at most $10^{-2}$.
			\item $R\geq 10L$ and $\epsilon\geq 10\delta$.
		\end{itemize}
		We will need some stronger assumptions in the next subsection, which will be
		specified by Lemma \ref{setupData}.
		
		\subsubsection{The panted cobordism group $\ocobordism_{R,\epsilon}$}\label{Subsubsec-pantedCobordism}
		%
		\begin{lemma}\label{pantedCobordismEquivalence}
			Being $(R,\epsilon)$-panted cobordant is an  equivalence relation
			on the set of $(R,\epsilon)$-multicurves.
		\end{lemma}
		
		\begin{proof} 
			The relation of being $(R,\epsilon)$-cobordant is clearly symmetric and transitive,
			so it remains to check the reflexivity.
			Any $(R,\epsilon)$-curve can be written as the cyclic concatenation
			of two consecutive $\partial$-framed segments of the same length and phase,
			which is $(L,\delta)$-tame by the assumptions of constants for this subsection.
			Then by splitting (Construction \ref{splitting}), any $(R,\epsilon)$-curve $\gamma$ is 
			the cuff of an $(R,\epsilon)$-pair-of-pants $\Pi$. Taking two oppositely oriented copies
			of $\Pi$ and gluing along the two pair of cuffs other than the pair of cuff $\gamma$ and 
			its orientation reversal yields an $(R,\epsilon)$-panted cobordism between $\gamma$ and
			itself. This implies that any $(R,\epsilon)$-multicurve is 
			$(R,\epsilon)$-panted cobordant to itself.
		\end{proof}
				
		\begin{lemma}\label{pantedCobordismGroupStructure}
			The set of $(R,\epsilon)$-panted cobordism classes of $(R,\epsilon)$-multicurves form
			a finitely generated abelian group $\ocobordism_{R,\epsilon}$.
			The addition is induced by the disjoint union operation,
			and the inverse is induced by orientation reversion.
		\end{lemma}
				
		\begin{proof}
			The Connection Principle (Lemma \ref{connectionPrinciple})
			implies that $\ocurves_{R,\epsilon}$ and $\opants_{R,\epsilon}$
			are nonempty finite sets. Hence $\ocobordism_{R,\epsilon}$ is nonempty.
			Define the addition on $\ocobordism_{R,\epsilon}$
			by $[L]_{R,\epsilon}+[L']_{R,\epsilon}=[L\sqcup L']_{R,\epsilon}$.
			It is straightforward to check that $\ocobordism_{R,\epsilon}$ forms
			an abelian group under the addition finitely generated over $\ocurves_{R,\epsilon}$.
			In fact, we have a natural presentation of $\ocobordism_{R,\epsilon}$
			given by the exact sequence
				$$\ZZ\opants_{R,\epsilon}\stackrel{\partial}{\longrightarrow}\ZZ\ocurves_{R,\epsilon}
				\longrightarrow\ocobordism_{R,\epsilon}\longrightarrow0.$$
			The zero of $\ocobordism_{R,\epsilon}$
			is represented by the sum of the three cuffs for any $\Pi\in\opants_{R,\epsilon}$.
			The same argument as the proof of Lemma \ref{pantedCobordismEquivalence}
			implies that $[L]_{R,\epsilon}=-[\bar{L}]_{R,\epsilon}$.
		\end{proof}
		
		\begin{proof} 
			It suffices to assume that $L$ has only one component.
			Then $[c]_{R,\epsilon}=-[\bar{c}]_{R,\epsilon}$
			if there is a pair of pants $\Pi_c\in\opants_{R,\epsilon}$
			with a cuff $c$, for each component $c$ of $L$.
			The condition certainly holds for universally small $\epsilon$
			if $R$ is sufficiently large (Construction \ref{splitting}).
			In fact, one may take two 
			oppositely oriented copies of $\Pi^\pm_c$ of $\Pi_c$, and glue them along their two
			cuffs other than $c$ or $\bar{c}$. The resulting $(R,\epsilon)$-panted surface
			has exactly two boundary component $c$ and $\bar{c}$.
		\end{proof}
		
		\begin{lemma}\label{finerCurves}
			$\ocobordism_{R,\epsilon}$ is generated by the $(R,\epsilon)$-panted
			cobordism classes of $(R,10\delta)$-multicurves.
		\end{lemma}
		
		\begin{proof} 
			This follows immediately from splitting (Construction \ref{splitting}).
		\end{proof}

		\subsubsection{The special orthonormal frame bundle $\SO(M)$}\label{Subsubsec-theSOM}
		Fix an orthonormal frame 
			$$\mathbf{e}\,=\,(\vec{t},\vec{n},\vec{t}\times\vec{n})$$
		at a fixed basepoint $\pt$ of $M$, and regard $\mathbf{e}$ as a basepoint of 
		the total space $\SO(M)$ of the bundle
		over $M$ of special orthonormal frames 
		with respect to the orientation of $M$. 
		One may naturally identify $\SO(M)$ as 
		$\Isom_0(\Hyp)\,/\,\pi_1(M)$.
		Since the tangent bundle of a closed 
		orientable $3$-manifold is always trivializable,
		$\SO(M)$ is a trivial $\SO(3)$-principal bundle
		over $M$.
		There are canonical short exact sequences
			$$1\longrightarrow\pi_1(\SO(3),I)\longrightarrow \pi_1(\SO(M),\mathbf{e})\longrightarrow \pi_1(M,\pt)\longrightarrow1,$$
		and
			$$0\longrightarrow\ZZ_2\longrightarrow H_1(\SO(M);\ZZ)\longrightarrow H_1(M;\ZZ)\longrightarrow0,$$
		both of which are splitting but not naturally.
		Note that $\pi_1(\SO(3),I)\cong\ZZ_2$ is the center
		of $\pi_1(\SO(M),\mathbf{e})$.
		We will usually write
			$$\hat{c}\,\in\,\pi_1(\SO(M),\mathbf{e})$$
		for the nontrivial central element.
		
		Certain noncentral
		elements of $\pi_1(\SO(M),\mathbf{e})$, namely, the $\delta$-sharp
		elements in the sense of Definition \ref{sharpElement} below,
		can be naturally represented by
		their associated oriented $\partial$-framed segments.
		This provides a convenient way to understand such elements,
		which will be especially useful 
		when we construct the inverse of $\Phi$.
		
		\begin{definition}\label{sharpElement}
			For any positive constant $\delta$ which is at most $10^{-2}$,
			a noncentral element $\hat{g}\in\pi_1(\SO(M),\mathbf{e})$
			is said to be \emph{$\delta$-sharp} if
			its image $g$ in $\pi_1(M,\pt)$ has 
			the initial and terminal directions
			$\delta$-close to $\vec{t}$ and $-\vec{t}$, respectively.
			For a $\delta$-sharp $\hat{g}$, we will
			say that an oriented $\partial$-framed segment
			$\mathfrak{g}$ is \emph{associated to} $\hat{g}$,
			and vice versa,
			if $\mathfrak{g}$ satisfies the following.
			\begin{itemize}
				\item The carrier segment of $\mathfrak{g}$ is $g$.
				The phase of $\mathfrak{g}$ is $\delta$-close to $0$.
				The initial and terminal framings of $\mathfrak{g}$ are $\delta$-close
				to each other.
				\item The element $\hat{g}$ is the homotopy class
				represented by a closed path of frames based
				at $\mathbf{e}\in\SO(M)$ as follows. The path first flows $\mathbf{e}$
				along $g$ by parallel transportation, and then
				rotates $180^\circ$ counterclockwise about $\vec{n}_\ter(\mathfrak{g})$, and then
				returns to $\mathbf{e}$ along a $\delta$-short path within $\SO(M)|_\pt$.
			\end{itemize}
		\end{definition}
		
		We will justify the existence of the associated $\partial$-segment in Lemma \ref{sharpElementLemma}.
		To understand the issue which the definition addresses, 
		consider for the moment a general nontrivial element $g\in\pi_1(M,*)$
		represented by a closed geodesic path based at $*$. The parallel transportation of $\mathbf{e}$ along
		$g$ gives rise to a path of frames $[0,1/2]\to\SO(M)$ which departs from $\mathbf{e}$ and arrives at
		another frame $\mathbf{e}'$ at $*$. Concatenate this path of frames with any path of frames
		$[1/2,1]\to\SO(M)|_\pt$, then we obtain a closed path based at $\mathbf{e}$ which represents
		a lift $\hat{g}\in\pi_1(\SO(M),\mathbf{e})$ of the element $g\in\pi_1(M,*)$. 
		Note that there are two possible lifts of $g$, namely, $\hat{g}$ and $\hat{c}\hat{g}$.
		In general, there is no natural way to distinguish these two possibilities by choosing the second path
		without ambiguity. However, when the angle of $g$ at $*$ is sufficiently sharp, 
		there are essentially (up to $\delta$-small difference)
		two ways to enrich $g$ with a $\partial$-framing	subject to the first condition of Definition \ref{sharpElement},
		and each of them specifies a lift of $g$ in $\pi_1(\SO(M),\mathbf{e})$
		by the description in second condition of Definition \ref{sharpElement}.
		We point out that our biased choice of the $180^\circ$ \emph{counterclockwise}
		rotation in Definition \ref{sharpElement} determines our choices of 
		chirality for tripods in the rest of this section.
		The essential difference between distinct chiralities
		is revealed by Lemma \ref{sharpTripodLemma} and 
		Remark \ref{sharpTripodLemmaRemark}.
				
		\begin{lemma}\label{sharpElementLemma}
			For any $\delta$-sharp element $\hat{g}\in\pi_1(\SO(M),\mathbf{e})$,
			there is an oriented $\partial$-framed segment
			$\mathfrak{g}$ associated to $\hat{g}$, unique up to $\delta$-small change of framings
			at endpoints.
			Moreover, $\mathfrak{g}^*$ is associated to $\hat{c}\hat{g}$, and $\bar{\mathfrak{g}}^*$ 
			is associated to $\hat{g}^{-1}$.
		\end{lemma}
		
		\begin{proof} 
			Let $\vec{m}$ be a unit vector orthogonal to $\vec{t}$
			such that the parallel transportation of $\vec{m}$ along $g$
			to the other end is $\delta$-close to $\vec{m}$.
			Up to $\delta$-small change, there are only two possible such vectors, namely,
			$\pm\vec{m}$. Enriching $g$ with initial and terminal framings 
			both $\delta$-close to $\pm\vec{m}$ yield
			oriented $\partial$-framed segments $\mathfrak{g}_\pm$ satisfying the first
			part of the listed properties in Definition \ref{sharpElement}.
			It is clear that exactly one of $\mathfrak{g}_\pm$
			fulfills the second part of the listed properties, so we pick it as $\mathfrak{g}$.
			The `moreover' part is straightforward from the construction above as well.
		\end{proof}
		
		\begin{lemma}\label{sharpTripodLemma}
			For any positive constant $\delta$ which is at most $10^{-2}$,
			suppose that $\mathfrak{t}_0\vee\mathfrak{t}_1\vee\mathfrak{t}_2$
			is a $(10,\delta)$-tame left-handed tripod with coincident terminal points
			of legs $*$ and with the terminal directions
			of legs $\delta$-close to $-\vec{t}$. Then there exist
			angles $\phi_0,\phi_1,\phi_2\in\RR/2\pi\ZZ$ satisfying 
			$$\phi_0+\phi_1+\phi_2=0,$$
			such that $2\phi_{i+2}\in\RR/2\pi\ZZ$
			is $\delta$-close to
			the directed angle from $\vec{n}_\ter(\mathfrak{t}_i)$
			to $\vec{n}_\ter(\mathfrak{t}_{i+1})$
			with respect to the common orthogonal vector at $\pt$ 
			which is $\delta$-close to $-\vec{t}$,
			and that each $\mathfrak{t}_{i,i+1}(\phi_{i+2})$
			(Definition \ref{partialFramedSegment})
			is associated to a $\delta$-sharp element $\hat{g}_{i+2}\in\pi_1(\SO(M),\mathbf{e})$
			for $i\in\ZZ_3$. For any such $\phi_i$ as above, the triangular relation
			$$\hat{g}_0\hat{g}_1\hat{g}_2=\id$$
			is satisfied. Moreover, adding two of the three $\phi_i$ by $\pi$ yields
			another triple of angles satisfying the conditions above,
			with two corresponding $\hat{g}_i$ changed
			into $\hat{c}\hat{g}_i$. 
		\end{lemma}
		
		\begin{remark}\label{sharpTripodLemmaRemark}
			If $\mathfrak{t}_0\vee\mathfrak{t}_1\vee\mathfrak{t}_2$
			is right-handed, we must either replace
			the triangular relation in the conclusion with
			the twisted triangular relation
			$\hat{g}_0\hat{g}_1\hat{g}_2=\hat{c}$, or alternatively,
			replace the equation for the anlges 
			with $\phi_0+\phi_1+\phi_2=\pi$.
		\end{remark}
		
		\begin{proof}
			For each $i\in\ZZ_3$, pick a unit vector $\vec{n}_i$ at $\pt$ orthogonal to $\vec{t}$,
			such that $\vec{n}_i$ is $\delta$-close to $\vec{n}_\ini(\mathfrak{t}_i)$.
			Let $\psi_{i,i+1}\in\RR/2\pi\ZZ$ be the angle 
			from $\vec{n}_i$ to $\vec{n}_{i+1}$ with respect to $\vec{t}$,
			and let $\phi'_{i+2}$ be half of $\psi_{i,i+1}$, valued in $\RR/\pi\ZZ$.
			Choose a lift $\phi_{i+2}\in\RR/2\pi\ZZ$ for each $\phi'_{i+2}$, so that 
			$\phi_0+\phi_1+\phi_2=0$. Note that any other lift can be obtained
			from changing two $\phi_i$ by adding $\pi$.
			It is clear that $\mathfrak{t}_{i,i+1}(\phi_{i+2})$
			is associated to a $\delta$-sharp $\hat{g}_{i+2}\in\pi_1(\SO(M),\mathbf{e})$,
			and $\hat{g}_0\hat{g}_1\hat{g}_2$ equals either $\id$ or $\hat{c}$.
			We claim that it is the former case.
			
			It suffices to verify that $\hat{g}_0\hat{g}_1\hat{g}_2$
			is trivial in $H_1(\SO(M),\mathbf{e})$. The argument is routine
			and easy so we only include an outline below.
			Let $\hat{\beta}_{i+2}$ be a path of frames
			from $(\vec{t},\vec{n}_i,\vec{t}\times\vec{n}_i)|_\pt$ to 
			$(\vec{t},\vec{n}_{i+1},\vec{t}\times\vec{n}_{i+1})|_\pt$
			that first flows along $g_{i+2}$ by parallel transportation,
			and then rotates $180^\circ$ counterclockwise about $\vec{n}_{i+1}$,
			and then rotates to $(\vec{t},\vec{n}_{i+1},\vec{t}\times\vec{n}_{i+1})|_\pt$
			via a $\delta$-short path in $\SO(M)|_\pt$. Here $g_{i+2}\in\pi_1(M,\pt)$
			is the image of $\hat{g}_{i+2}$, also regarded as a pointed geodesic loop.
			Since $\vec{n}_i$ are all $\delta$-close to the normal vector of
			the $2$-simplex $\sigma$ 
			spanned by the concatenation of $g_0,g_1,g_2$, the loop of frames 
			$\hat{\beta}_0\hat{\beta}_1\hat{\beta}_2$ obtained by concatenation
			is homotopic to the constant loop $(\vec{t},\vec{n}_1,\vec{t}\times\vec{n}_1)|_\pt$
			in $\SO(M)|_*$ as $\mathfrak{t}_0\vee\mathfrak{t}_1\vee\mathfrak{t}_2$
			is left-handed. (Compare with the right-handed case where the
			resulting loop would be the $360^\circ$ counterclockwise rotation
			of the frame $(\vec{t},\vec{n}_1,\vec{t}\times\vec{n}_1)|_\pt$ about 
			$\vec{n}_1$.) In particular, 
				$$[\hat{\beta}_0\hat{\beta}_1\hat{\beta}_2]=0$$
			in $H_1(\SO(M),\mathbf{e})$. 
			Let $\vec{m}_{i+2}$ be a unit vector at $\pt$ orthogonal to $\vec{t}$,
			such that $\vec{m}_{i+2}$ is $\delta$-close to both the initial and terminal
			framings of $\mathfrak{t}_{i,i+1}(\phi_{i+2})$. Let $\hat{\xi}_{i+2}$ be
			the path of frames in $\SO(M)|_\pt$ 
			from $(\vec{t},\vec{m}_{i+2},\vec{t}\times\vec{m}_{i+2})$ to
			$(\vec{t},\vec{n}_{i},\vec{t}\times\vec{n}_{i})$
			by a rotation of angle $\delta$-close to $\phi_{i+2}$, and let
			$\hat{\eta}_{i+2}$ be the path of frames
			in $\SO(M)|_\pt$ 
			from $(\vec{t},\vec{n}_{i+1},\vec{t}\times\vec{n}_{i+1})$
			to $(\vec{t},\vec{m}_{i+2},\vec{t}\times\vec{m}_{i+2})$ 			
			by a rotation of angle $\delta$-close to $\phi_{i+2}$.
			Then the loop of frames
			$\hat{\xi}_{i+2}\hat{\beta}_{i+2}\hat{\eta}_{i+2}$ based at
			$(\vec{t},\vec{m}_{i+2},\vec{t}\times\vec{m}_{i+2})$ can be
			conjugated to the loop based at $\mathbf{e}$ representing
			$\hat{g}_{i+2}$ described in Definition \ref{sharpElement}.
			Thus 
				$$[\hat{g}_{i+2}]=[\hat{\xi}_{i+2}\hat{\beta}_{i+2}\hat{\eta}_{i+2}]$$
			in $H_1(\SO(M),\mathbf{e})$.
			Note that $\hat{\xi}_0\hat{\eta}_0\hat{\xi}_1\hat{\eta}_1\hat{\xi}_2\hat{\eta}_2$
			is a loop of frames in $\SO(M)|_\pt$ 
			that rotates counterclockwise about $\vec{t}$ of angle
			$2(\phi_0+\phi_1+\phi_2)$. As $\phi_0+\phi_1+\phi_2=0$
			modulo $2\pi$, the winding number of this
			loop around $\vec{t}$ is even. Hence  
				$$[\hat{\xi}_0\hat{\eta}_0\hat{\xi}_1\hat{\eta}_1\hat{\xi}_2\hat{\eta}_2]=0$$
			in $H_1(\SO(M);\ZZ)$. Therefore, in $H_1(\SO(M);\ZZ)$,
			\begin{eqnarray*}
				[\hat{g}_0\hat{g_1}\hat{g}_2]&=&
				[\hat{\xi}_{0}\hat{\beta}_{0}\hat{\eta}_{0}]+[\hat{\xi}_{1}\hat{\beta}_{1}\hat{\eta}_{1}]+[\hat{\xi}_{2}\hat{\beta}_{2}\hat{\eta}_{2}]\\
				&=&
				[\hat{\xi}_{0}\hat{\beta}_{0}\hat{\eta}_{0}\hat{\xi}_{1}\hat{\beta}_{1}\hat{\eta}_{1}\hat{\xi}_{2}\hat{\beta}_{2}\hat{\eta}_{2}]\\
				&=&
				[\hat{\beta}_{0}\hat{\beta}_{1}\hat{\beta}_{2}]+[\hat{\xi}_0\hat{\eta}_0\hat{\xi}_1\hat{\eta}_1\hat{\xi}_2\hat{\eta}_2]\\
				&=&0.
			\end{eqnarray*}
			This implies that $\hat{g}_0\hat{g}_1\hat{g}_2=\id$ in $\pi_1(\SO(M),\mathbf{e})$.			
		\end{proof}

		There is a canonical way to 
		lift $(R,\epsilon)$-curves and $(R,\epsilon)$-pants into $\SO(M)$, up to homotopy.
		
		\begin{definition}\label{canonicalLift}
			For any curve $\gamma\in\ocurves_{R,\epsilon}$,
			a \emph{canonical lift} of $\gamma$ is a loop of frames
				$$\hat{\gamma}:S^1\to \SO(M)$$
			as follows.
			Choose a point $p$ on the geodesic representative
			of $\gamma$, and a normal vector $\vec{n}_p$
			of $\gamma$ at $p$. Let $\vec{t}_p$ be the direction vector of
			$\gamma$ at $p$. The frame 
			$\mathbf{e}_{\gamma,\vec{n}_p}=(\vec{t}_p,\vec{n}_p,\vec{t}_p\times\vec{n}_p)$ is an element
			of $\SO(M)|_p$.	With these notations, 
			a base-point free loop of frames $\hat{\gamma}$ 
			starts from $\mathbf{e}_{\gamma,\vec{n}_p}$,
			and then flows once around $\gamma$ by parallel transportation,
			and then rotates $360^\circ$ counterclockwise
			about $\vec{n}_p$,
			and then rotates back to $\mathbf{e}_{\gamma,\vec{n}_p}$
			along an $\epsilon$-short path within $\SO(M)|_p$.
			For any pair of pants $\Pi\in\opants_{R,\epsilon}$,
			a \emph{canonical lift} of $\Pi$ is a lift 
				$$\hat{\Pi}:\,\Sigma_{0,3}\to \SO(M)$$
			of $\Pi$, such that the three cuffs are canonically lifted.
		\end{definition}
		
		The definition is justified by the following Lemma \ref{canonicalLiftLemma}. 
		Similar to the explanation after Definition \ref{sharpElement},
		the nearly geodesic assumption helps to distinguish the two possible lifts
		of a $(R,\epsilon)$-curve $\gamma$ in $\SO(M)$ up to homotopy,
		namely, the canonical lift above and the lift by almost parallel transportation
		along $\gamma$.
		The fundamental difference between the two possible lifts is that  
		if we lift all the cuffs of $\Pi$ in the latter way,
		it would be impossible to extend the lift to $\Pi$ up to homotopy.
		
		\begin{lemma}\label{canonicalLiftLemma}
			For any positive constant $\epsilon$ which is at most $10^{-2}$, suppose $\gamma\in\ocurves_{R,\epsilon}$ and
			$\Pi\in\opants_{R,10\delta}$. 
			The canonical lifts $\hat{\gamma}$ and $\hat{\Pi}$ as described in
			Definition \ref{canonicalLift} exist and are unique up to homotopy.
		\end{lemma}
		
		\begin{remark}\label{canonicalLiftRemark}
			However, if the canonical lifts of $\hat{\Pi}$ on the cuffs have been chosen,
			$\hat{\Pi}$ is unique only up to homotopy relative to cuffs 
			together with $\ZZ_2$-Dehn twists
			in the fiber $\SO(3)$ near the boundary. In other words,
			the relative homotopy class of $\hat{\Pi}$ is determined by any
			class of $H_2(\SO(\Pi),\SO(\partial\Pi);\ZZ)$ that projects to the
			fundamental class $[\Pi]$ in $H_2(\Pi,\partial\Pi;\ZZ)$.
		\end{remark}

		\begin{proof}
			The existence of $\hat{\gamma}$ is by definition.
			The uniqueness follows from the fact 
			that the set of homotopy classes of framings of 
			$TM|_\gamma$ is bijective to 
			$[S^1,\SO(3)]\cong \ZZ_2$.
			To see the existence of $\hat{\Pi}$, note that 
			the pull-back tangent bundle $TM|_{\Pi}$, namely, $\Pi^*(TM)$,
			is isomorphic to $T\Sigma_{0,3}\oplus \epsilon^1$. Consider a trivialization
			of $T\Sigma_{0,3}$, for example, by embedding $\Sigma_{0,3}$ into the plane and
			endowing with the standard framing of $\RR^2$. By direct summing with the trivialization
			induced by the orientation of $\Pi$, the trivialization of $T\Sigma_{0,3}$
			naturally induces a framing of $TM|_{\Pi}$ up to homotopy. The 
			restriction of this framing on any cuff $\gamma$ of $\Pi$ is the canonical lift of $\gamma$.
			Thus this framing of $TM|_{\Pi}$ is a canonical lift $\hat{\Pi}:\Sigma_{0,3}\to\SO(M)$
			of $\Pi$ by definition. To see the uniqueness 
			of $\hat{\Pi}$, note that the
			set of homotopy classes of framings of $TM|_{\Sigma_{0,3}}$ is bijective to
			$[\Sigma_{0,3},\SO(3)]\cong \ZZ_2\oplus\ZZ_2$. Thus the homotopy class of a framing
			of $TM|_{\Sigma_{0,3}}$ is uniquely determined by its restriction to the cuffs.
			This completes the proof.			
		\end{proof}

		\subsubsection{Construction of $\Phi$}
		Under the assumptions of constants introduced at the beginning of this subsection,
		we construct 
			$$\Phi:\ocobordism_{R,\epsilon}\to H_1(\SO(M);\,\ZZ)$$
		as follows.
		Suppose $\gamma$ is a geodesic representative of
		a curve in $\ocurves_{R,\epsilon}$.
		We define $\Phi([\gamma]_{R,\epsilon})$ in $H_1(\SO(M);\ZZ)$
		to be represented by the canonical lift of $\gamma$ (Definition \ref{canonicalLift}).
		For an $(R,\epsilon)$-multicurve $L$, we define
			$$\Phi([L]_{R,\epsilon})\,\in\,H_1(\SO(M);\,\ZZ)$$
		to be the sum of $\Phi$ defined for each of its components.
		We verify that $\Phi$ is well defined.
						
		\begin{lemma}\label{wellDefinedPhi}
			The homology class
			$\Phi([L]_{R,\epsilon})$
			in $H_1(\SO(M);\,\ZZ)$ constructed above depends
			only on the $(R,\epsilon)$-panted cobordism
			class $[L]_{R,\epsilon}\in\ocobordism_{R,\epsilon}$ of $L$.
			Moreover, the induced map $\Phi$ from $\ocobordism_{R,\epsilon}$
			to $H_1(\SO(M);\,\ZZ)$ is a homomorphism.
		\end{lemma}
		
		\begin{proof}
			If $[L]_{R,\epsilon}$ vanishes in $\ocobordism_{R,\epsilon}$, there exists
			an $(R,\epsilon)$-panted surface $F\looparrowright M$ bounded by $L$. The canonical 
			lifts of pairs of pants of $F$ (Definition \ref{canonicalLift})
			yield a canonical lift $F\to \SO(M)$, whose restriction to the boundary
			are the canonical lifts
			of components of $L$. This implies that $(R,\epsilon)$-panted cobordant
			multicurves yield homologous canonical lifts, or in other words,
			that $\Phi([L]_{R,\epsilon})$
			in $H_1(\SO(M);\,\ZZ)$ depends only $[L]_{R,\epsilon}$. The `moreover' part is
			straightforward from the definition.
		\end{proof}
		
		\begin{lemma}\label{canonicallyDefinedPhi}
			For any $(R,\epsilon)$-panted cobordism class
			$[L]_{R,\epsilon}$, the image of
			$\Phi([L]_{R,\epsilon})$ under the natural
			projection from $H_1(\SO(M);\,\ZZ)$ to $H_1(M;\ZZ)$
			is the homology class $[L]$.
		\end{lemma}
		
		\begin{proof}
			This follows immediately from the construction of $\Phi$.
		\end{proof}
	
	\subsection{The inverse of $\Phi$}\label{Subsec-theInverseOfPhi}
		Fix an orthonormal frame $\mathbf{e}\,=\,(\vec{t},\vec{n},\vec{t}\times\vec{n})$
		at a fixed basepoint $\pt$ of $M$ as before.
		In this subsection, we construct a homomorphism
			$$\Psi:\,\pi_1(\SO(M),\mathbf{e})\to\ocobordism_{R,\epsilon}$$
		which, descending to the abelianization,
			$$\Psi^{\mathtt{ab}}:\,H_1(\SO(M);\ZZ)\to\ocobordism_{R,\epsilon},$$
		yields the inverse of $\Phi$.
		We need to choose some setup data including a triangular finite generating
		set of $\pi_1(M,\pt)$, and define $\Psi$ on a subset of $\pi_1(M,\pt)$
		that contains the triangular generating set. We verify that
		$\Psi$ extends as a homomorphism by showing that it 
		vanishes on words corresponding to the triangular relations.
		We also verify that $\Psi^{\mathtt{ab}}$ is the inverse of $\Phi$ by showing that
		it is surjective and is the pre-inverse of $\Phi$.

		\subsubsection{Setup}\label{Subsubsec-setup}
		Given any small positive
		number $\epsilon$ at most $10^{-2}$, we need to fix some setup 
		data 
			$$(\mathcal{B}^\circ_K,\mathcal{B}^\circ_D,\tau_h)$$
		in terms of	$\pi_1(M,\pt)$, restrained by some 
		environmental data $(\delta,L)$, which is a pair of 
		constant fulfilling the conditions of the Connection Principle (Lemma \ref{connectionPrinciple}).
		Such a collection of data is provided by Lemma \ref{setupData}
		below, together with a positive constant
			$$R(\epsilon,M)$$
		such that the homomorphism $\Psi$ from $\pi_1(\SO(M),\mathbf{e})$
		to $\ocobordism_{R,\epsilon}$ can be constructed
		for any constant $R$ greater than $R(\epsilon,M)$.	
	
		We naturally identify the universal cover of $M$ 
		as the hyperbolic $3$-space $\Hyp$ with a given basepoint $O$ 
		together with an orthonormal frame $(\vec{t}_O,\vec{n}_O,\vec{t}_O\times\vec{n}_O)$,
		so $\pi_1(M,\pt)$ naturally acts on $\Hyp$ by isometries.
		For any positive constant $r$, we denote by $B_r$
		the open ball of radius $r$ centered at $O$, and by $U_r$
		the open half-space containing $O$ bounded by the
		hyperplane tangent to $\partial B_r$
		at the point in the direction $\vec{t}_O$. Define 
			$$\mathcal{B}^\circ_r\,=\,\{g\in\pi_1(M,\pt)\,|\,g.O\in B_r\textrm{, and }g\neq\id\}$$
		and
			$$\mathcal{C}_r\,=\,\{g\in\pi_1(M,\pt)\,|\,g.U_r\cap U_r=\emptyset\}.$$
		To gain some intuition about these subsets, one may verify that
		elements of $\mathcal{B}^\circ_r$ are represented by geodesic closed paths in $M$ based at $\pt$
		which have length at most $r$, and elements of $\mathcal{C}_r$
		are represented by geodesic closed paths which form an angle at most $2\pi e^{-r}$ at $\pt$.

		In the following,
		a \emph{triangular generating set} of $\pi_1(M,\pt)$ means a generating
		set such that there is a finite set of defining relations in these generators (or their inverses)
		consisting of words of length at most $3$,
		(cf.~Subsection \ref{Subsec-triangularPresentationComplexes} for discussion
		in more details);
		the conjugation
		$\tau_h$ induced by an element $h\in\pi_1(M,\pt)$ acts
		on $\pi_1(M,\pt)$ by $\tau_h(g)=h^{-1}gh$.
			
		\begin{lemma}\label{setupData}
			Let $M$ be an oriented closed hyperbolic $3$-manifold. Given any positive
			constant $\epsilon$ at most $10^{-2}$, 
			there is a collection of data 
			depending only on $M$ and $\epsilon$ as follows.
			\begin{enumerate}
				\item There exist positive constants $\delta$ and $L$.
				The constant $\delta$ is at most $\epsilon\times10^{-6}$;
				the constant $L$ is at least $-2\log\delta+10\log2$,
				and satisfies the conclusion of
				the Connection Principle (Lemma \ref{connectionPrinciple}) 
				with respect to $\delta$ and $M$.
				\item There exist positive constants $D$ and $K$.
				The constant $D$ is at least $30L$, and
				$\mathcal{B}^\circ_{D}$ contains a generating set
				of $\pi_1(M,\pt)$; the constant $K$ is at least $D$,
				and $\mathcal{B}^\circ_{K}$ 
				contains a triangular generating set
				of $\pi_1(M,\pt)$ which further contains $\mathcal{B}^\circ_{D}$.
				\item There exists a conjugation $\tau_h$ of $\pi_1(M,\pt)$
				such that $\tau_h(\mathcal{B}^\circ_K)$ is contained in $\mathcal{C}_{300L}$.
				Moreover, $h$ can be chosen in $\mathcal{C}_{300L}$. 
				\item There exists a positive constant 
					$$R(\epsilon,M)$$
				such that $\mathcal{B}^\circ_{R(\epsilon,M)-300L}$
				contains $\tau_h(\mathcal{B}^\circ_K)$.
			\end{enumerate}
		\end{lemma}
		
		In particular, the derived constructions of Subsection \ref{Subsec-derivedConstructions}
		can be applied with respect to the constants $(L,\delta)$ of Lemma \ref{setupData} (1).
		With Lemma \ref{setupData} (2) and (3), we have a triangular generating set $\tau_h(\mathcal{B}^\circ_K)$
		of $\pi_1(M,*)$, which are represented by (unframed) $\delta$-sharp elements, (cf.~Lemma \ref{sharpCK}).
		This will enable us to define the homomorphism $\Psi$.
		The assumption about $\mathcal{B}^\circ_D$ ensures that $\mathcal{B}^\circ_K$
		contains enough generators so that with Lemma \ref{setupData} (4)
		we will be able show the surjectivity of $\Psi$.
		With some needless awkward intuition, the geodesic representative of $h$ based at $\pt$
		may look like a long, sharp, closed segment,
		of which the direction vectors at the endpoints 
		are not extremely close to any direction vector of any element of $\mathcal{B}^\circ_K$.
		We would rather suggest the reader 
		to understand $h$ simply to be an element that guarantees Lemma \ref{setupData} (3) and (4).		
		
		\begin{proof}
			The statement (1) is implied by the Connection Principle (Lemma \ref{connectionPrinciple}).
						
			The statement (2) follows from the fact
			that any presentation $(\mathcal{S},\mathcal{R})$
			of a group $G$ induces a triangular generating set $\tilde{\mathcal{S}}$
			consisting of all the elements that are represented by subwords
			of all the relators from $\mathcal{R}$. Note that $\tilde{\mathcal{S}}$ is finite
			if $(\mathcal{S},\mathcal{R})$ is a finite presentation. 
			Since $\pi_1(M,\pt)$ is finitely presented, we may choose $D$ at least $30L$
			to be sufficiently
			large so that $\mathcal{B}^\circ_{D}$ generates $\pi_1(M,\pt)$. Then for any
			finite presentation of $\pi_1(M,\pt)$ over $\mathcal{B}^\circ_{D}$,
			the induced finite triangular generating
			set $\tilde{\mathcal{B}}^\circ_{D}$ is finite, and hence contained in $U_K$
			for some sufficiently large $K$.
			
			To prove the statement (3), we must find an element $h\in\pi_1(M,\pt)$
			such that $\tau_h(g).U_K$ is disjoint from $U_K$
			for all $g\in\mathcal{B}^\circ_K$.
			Note that for any nontrivial $g\in\pi_1(M,\pt)$, 
			if an open half-space $W$ is sufficiently far
			from the axis of $g$, depending only on the translation distance of $g$,
			then $g.W\cap W=\emptyset$.	Because $\mathcal{B}^\circ_K$ is a finite set, 
			there exists an open half-space $W$ so that $g.W\cap W=\emptyset$ for all
			$g\in\mathcal{B}^\circ_K$. Let $h\in\pi_1(M,\pt)$ be an element whose axis 
			is contained in $W$. We may further assume that $W\cap U_K=\emptyset$.
			Possibly after passing to a sufficiently great power
			of $h$, we may assume $h.W^c\cap W^c=\emptyset$ where $W^c$ denotes the complement
			of $W$. See Figure \ref{figSetupData}. For such $h$, the half-space
			$h.U_K$ is contained in $W$, so we have $g.(h.U_K)\cap h.U_K=\emptyset$,
			for all $g\in\mathcal{B}^\circ_K$.
			In other words, $\tau_h(g).U_K$ is disjoint from $U_K$ for all $g\in\mathcal{B}^\circ_K$
			as desired. Moreover, if we took $W$ so that $W\cap U_{300L}=\emptyset$,
			then $h$ would be contained in $\mathcal{C}_{300L}$. In fact,
			$h.W^c\cap W^c=\emptyset$ implies that $h.U_{300L}\cap U_{300L}=\emptyset$.
			
			\begin{figure}[htb]
				\centering
				\includegraphics{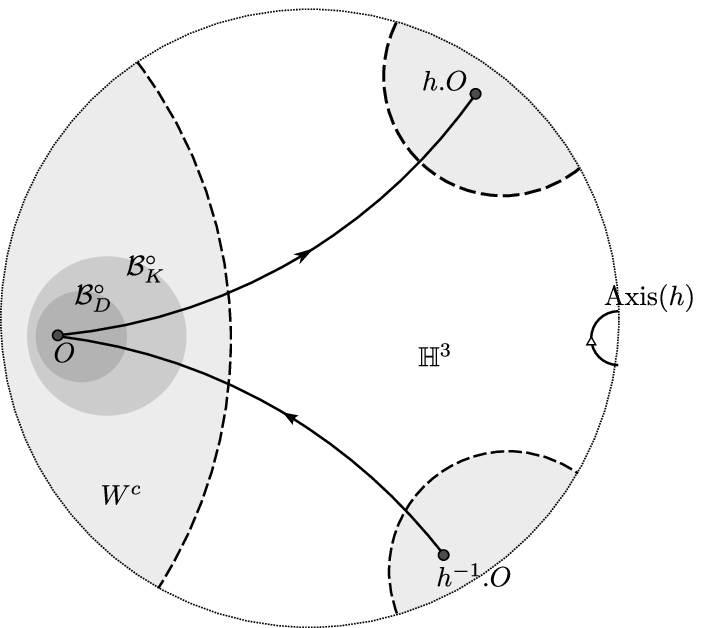}
				\caption{}\label{figSetupData}
			\end{figure}

			The statement (4) is now obvious since we have chosen $L$, $K$, and $\tau_h$ 
			in a way depending only on $M$ and $\epsilon$.			
		\end{proof}
		
		\begin{lemma}\label{sharpCK}
			With the constants from Lemma \ref{setupData}, for any $g\in\mathcal{C}_{300L}$, 
			the initial direction and the terminal direction of $g$
			at $\pt$ are $\delta$-close to $\vec{t}$ and $-\vec{t}$, respectively.
		\end{lemma}
		
		\begin{proof}
			The complement of the half-space $U_r$ is contained in the cone
			at $O$ with axis along $\vec{t}_O$ of cone angle $\alpha$ where
			$\sin(\alpha/2)\cosh r=1$,
			so $\alpha$ is clearly much less than $\delta$ when $r$ equals $300L$.
			Since $g\in\mathcal{C}_r$, $g.U_r$ and $g^{-1}.U_r$ are both
			contained in the complement of $U_r$. This implies that $g$ is $\delta$-sharp
			with the initial direction and the terminal direction 
			$\delta$-close to $\vec{t}$ and $-\vec{t}$, respectively.
		\end{proof}

		\subsubsection{Construction of $\Psi$}\label{Subsubsec-constructionPsi}
		Given any positive
		constant $\epsilon$, we fix a collection of 
		setup data $(\mathcal{B}^\circ_K,\mathcal{B}^\circ_D,\tau_h)$ subject to
		$(\delta,L)$ as provided by Lemma \ref{setupData},
		and obtain a positive constant $R(\epsilon,M)$ accordingly.
		For any constant $R$ at least $R(\epsilon,M)$, 
		we will construct the homomorphism 
			$$\Psi:\,\pi_1(\SO(M),\mathbf{e})\to\ocobordism_{R,\epsilon}$$
		in the following. More precisely, let $\hat{\mathcal{B}}^\circ_r$ and $\hat{\mathcal{C}}_r$ 
		denote the preimages of $\mathcal{B}^\circ_r$ and $\mathcal{C}_r$
		in $\pi_1(\SO(M),\mathbf{e})$, respectively.
		Since $\pi_1(\SO(M),\mathbf{e})$ is a central extension
		of $\pi_1(M,\pt)$ by $\ZZ_2$, the conjugation $\tau_h$ of $\pi_1(M,\pt)$
		naturally induces a conjugation of $\pi_1(\SO(M),\mathbf{e})$,
		which will be denoted as $\hat{\tau}_h$.
		We will construct a set-theoretic map 
			$$\Psi_1:\,\hat{\mathcal{C}}_{300L}\,\cap\,\hat{\mathcal{B}}^\circ_{R-300L}\longrightarrow\ocobordism_{R,\epsilon}.$$
		The restriction of $\Psi_1\circ\hat{\tau}_h$ to $\hat{\mathcal{B}}^\circ_{K}$, denoted as
			$$\Psi_h:\,\hat{\mathcal{B}}^\circ_K\longrightarrow\ocobordism_{R,\epsilon},$$
		will be a partial homomorphism, so the
		further restriction to $\hat{\mathcal{B}}^\circ_D$ will
		extend uniquely over $\pi_1(\SO(M),\mathbf{e})$
		to be a homomorphism, which will still be denoted as
		$\Psi_h$. However, it will be 
		verified that $\Psi_h$
		descends to the abelianization, yielding a homomorphism
			$$\Psi_h^{\mathtt{ab}}:\,H_1(\SO(M);\ZZ)\longrightarrow\ocobordism_{R,\epsilon},$$
		which is exactly the inverse of $\Phi$
		(Subsubsection \ref{Subsubsec-verifications}), so eventually we may 
		drop the subscript simply writing $\Psi_h$ as $\Psi$.

		For any element $\hat{g}\in\hat{\mathcal{C}}_{300L}\,\cap\,\hat{\mathcal{B}}^\circ_{R-300L}$, 
		by Lemmas \ref{sharpElementLemma}, \ref{sharpCK}, we may
		choose an oriented $\partial$-framed segment 
			$$\mathfrak{s}_{\hat{g}}$$
		associated to the $\delta$-sharp element $\hat{g}$.
		Note that $\mathfrak{s}_{\hat{g}}$ is unique up to $\delta$-small
		change of the initial and terminal framings, and 
		$\mathfrak{s}_{\hat{c}\hat{g}}$ can be chosen as
		the framing flipping $\mathfrak{s}^*_{\hat{g}}$.
		
		We define the claimed set-theoretic map
		$\Psi_1$ as follows.				
		For each element
		$\hat{g}\in\hat{\mathcal{C}}_{300L}\,\cap\,\hat{\mathcal{B}}^\circ_{R-300L}$,
		choose an oriented $\partial$-framed
		segment $\mathfrak{s}_{\hat{g}}$
		which is 
		associated to $\hat{g}$ as above.
		For convenience, choose a unit vector
			$$\vec{n}_{\hat{g}}\in T_{\pt}M$$
		orthogonal to $\vec{t}$, such that
		$\vec{n}_\ini(\mathfrak{s}_{\hat{g}})$ and $\vec{n}_\ter(\mathfrak{s}_{\hat{g}})$
		are both $\delta$-close to $\vec{n}_{\hat{g}}$.
		By the Connection Principle (Lemma \ref{connectionPrinciple}), choose a right-handed nearly regular 
		tripod $\mathfrak{a}_0\vee\mathfrak{a}_1\vee\mathfrak{a}_2$
		and an oriented $\partial$-framed segment $\mathfrak{b}$ satisfying the following:
		\begin{itemize}
			\item The right-handed tripod $\mathfrak{a}_0\vee\mathfrak{a}_1\vee\mathfrak{a}_2$
			is $(\frac{R}2-\frac{\ell(\mathfrak{s}_{\hat{g}})}2+\frac12 I(\frac\pi3),10\delta)$-nearly regular. 
			For each $i\in\ZZ_3$, the terminal endpoint $p_\ter(\mathfrak{a}_i)$
			equals $\pt$, and the terminal direction $\vec{t}_\ter(\mathfrak{a}_i)$ is $(10\delta)$-close
			to $\vec{t}$, and the terminal framing $\vec{n}_\ter(\mathfrak{a}_i)$ is $(10\delta)$-close
			to $\vec{n}_{\hat{g}}$.
			\item The oriented $\partial$-framed segment $\mathfrak{b}$ has length 
			$(10\delta)$-close to $\frac{R}2-\ell(\mathfrak{s}_{\hat{g}})$ and 
			phase $(10\delta)$-close to $0$. The initial and terminal
			directions of $\mathfrak{b}$ are $\delta$-close to $-\vec{t}$ and $\vec{t}$ respectively,
			and the initial and 
			terminal framings of $\mathfrak{b}$ are both $\delta$-close to $\vec{n}_{\hat{g}}$.
		\end{itemize}
		
		\begin{figure}[htb]
				\centering
				\includegraphics{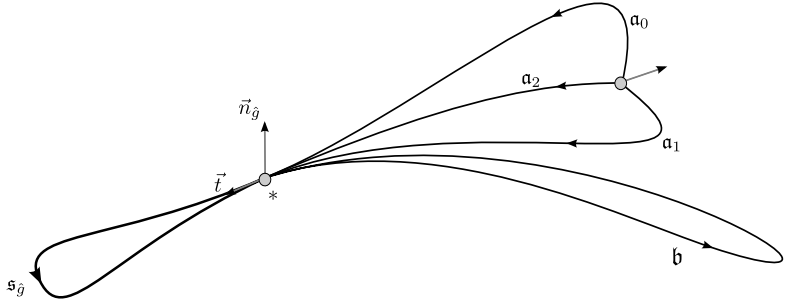}
				\caption{Definition of $\Psi_1(\hat{g})$}\label{figPsi1}
			\end{figure}
		
		In $\ocobordism_{R,\epsilon}$, for any
		element $\hat{g}\in\hat{\mathcal{C}}_{300L}\,\cap\,\hat{\mathcal{B}}^\circ_{R-300L}$, define
			$$\Psi_1(\hat{g})\,=\,
			[\mathfrak{s}_{\hat{g}}\mathfrak{a}_{01}]_{R,\epsilon}
			+[\mathfrak{s}_{\hat{g}}\mathfrak{a}_{12}]_{R,\epsilon}
			+[\mathfrak{s}_{\hat{g}}\mathfrak{a}_{20}]_{R,\epsilon}
			-\,[\mathfrak{s}_{\hat{g}}\mathfrak{b}]_{R,\epsilon}
			-\,[\mathfrak{s}_{\hat{g}}\bar{\mathfrak{b}}]_{R,\epsilon}.$$		
				
		For any element $\hat{g}\in\hat{\mathcal{B}}^\circ_K$, we define
			$$\Psi_h(\hat{g})\,=\,\Psi_1(\hat{\tau}_h(\hat{g})),$$
		(Lemma \ref{setupData} (3)(4)).

		\begin{lemma}\label{wellDefinition}
			The set-theoretic map $\Psi_1$ is well defined. In other words,
			\begin{enumerate}
				\item All the reduced cyclic concatenations involved are curves in $\ocurves_{R,100\delta}$;
				\item For any $\hat{g}\in \hat{\mathcal{C}}_{300L}\cap\hat{\mathcal{B}}^\circ_{R-300L}$,
				$\Psi_1(\hat{g})$ depends only on $\hat{g}$;
			\end{enumerate}
		\end{lemma}
		
		\begin{proof}
			The statement (1) follows from straightforward
			verification using 
			the Length and Phase Formula (Lemma \ref{lengthAndPhaseFormula})
			under our fixed choice of setup data (Lemma \ref{setupData}).
						
			To prove the statement (2), observe that $\Psi_1(\hat{g})$ is
			clearly independent of the choice of $\mathfrak{s}_{\hat{g}}$
			and $\vec{n}_{\hat{g}}$,
			since the carrier segment of
			$\mathfrak{s}_{\hat{g}}$ is unique
			and $\vec{n}_{\hat{g}}$ is unique up to $\delta$-closeness.
			Suppose that $\mathfrak{a}'_0\vee\mathfrak{a}'_1\vee\mathfrak{a}'_2$ 
			is another oriented $\partial$-framed segment
			satisfying the same conditions as of $\mathfrak{a}_0\vee\mathfrak{a}_1\vee\mathfrak{a}_2$, 
			and similarly for $\mathfrak{b}'$. We write the new $\Psi_1({\hat{g}})$
			as $\Psi'_1({\hat{g}})$ to distinguish. 
			We must show that $\Psi_1({\hat{g}})$ equals $\Psi'_1({\hat{g}})$ 
			in $\ocobordism_{R,\epsilon}$.			

			Choose an auxiliary left-handed nearly regular 
			tripod $\mathfrak{c}_0\vee\mathfrak{c}_1\vee\mathfrak{c}_2$
			satisfying the following. 
			\begin{itemize}
				\item The left-handed tripod 
				$\mathfrak{c}_0\vee\mathfrak{c}_1\vee\mathfrak{c}_2$ is 
				$(\frac{\ell(\mathfrak{s}_{\hat{g}})}2+\frac12I(\frac{\pi}3),\delta)$-nearly regular.
				For each $i\in\ZZ_3$, the terminal endpoint $p_\ter(\mathfrak{c}_i)$
				equals $*$, and the terminal direction $\vec{t}_\ter(\mathfrak{c}_i)$ is $\delta$-close
				to $-\vec{t}$, and the terminal framing $\vec{n}_\ter(\mathfrak{c}_i)$ is $\delta$-close
				to $\vec{n}_{\hat{g}}$.
			\end{itemize}
			Since both
			$\mathfrak{a}_0\vee\mathfrak{a}_1\vee\mathfrak{a}_2$ 
			and $\mathfrak{a}'_0\vee\mathfrak{a}'_1\vee\mathfrak{a}'_2$ have the
			opposite chirality to that of $\mathfrak{c}_0\vee\mathfrak{c}_1\vee\mathfrak{c}_2$,
			it follows from rotation (Construction \ref{rotationConstruction} (1)) that
				$$\sum_{i\in\ZZ_3}\,[\mathfrak{a}_{i,i+1}\bar{\mathfrak{c}}_{i,i+1}]_{R,\epsilon}\,=\,0$$
			and 
				$$\sum_{i\in\ZZ_3}\,[\mathfrak{a}'_{i,i+1}\bar{\mathfrak{c}}_{i,i+1}]_{R,\epsilon}\,=\,0,$$
			where $\mathfrak{a}_{i,i+1}$ means $\bar{\mathfrak{a}}_i\mathfrak{a}_{i+1}$ for $i\in\ZZ_3$,
			and similarly for the notations $\mathfrak{a}'_{i,i+1}$ and $\mathfrak{c}_{i,i+1}$.
			On the other hand, by swapping (Construction \ref{swapping}),
				$$[\mathfrak{s}_{\hat{g}}\mathfrak{a}_{i,i+1}]_{R,\epsilon}-[\mathfrak{s}_{\hat{g}}\mathfrak{a}'_{i,i+1}]_{R,\epsilon}=
				[\bar{\mathfrak{c}}_{i,i+1}\mathfrak{a}_{i,i+1}]_{R,\epsilon}-[\bar{\mathfrak{c}}_{i,i+1}\mathfrak{a}'_{i,i+1}]_{R,\epsilon},$$
			and
				$$[\mathfrak{s}_{\hat{g}}\mathfrak{b}]_{R,\epsilon}-[\mathfrak{s}_{\hat{g}}\mathfrak{b}']_{R,\epsilon}=
				[\bar{\mathfrak{s}}_{\hat{g}}\mathfrak{b}]_{R,\epsilon}-[\bar{\mathfrak{s}}_{\hat{g}}\mathfrak{b}']_{R,\epsilon}.$$
			For convenience, we write $\mathfrak{b}_{01}$ and $\mathfrak{b}_{10}$ for $\mathfrak{b}$
			and $\bar{\mathfrak{b}}$, respectively, and similarly for $\mathfrak{b}'_{01}$ and $\mathfrak{b}'_{10}$.
			Then $\Psi_1({\hat{g}})-\Psi'_1({\hat{g}})$ equals
			\begin{eqnarray*}
				&&\sum_{i\in\ZZ_3}\,([\mathfrak{s}_{\hat{g}}\mathfrak{a}_{i,i+1}]_{R,\epsilon}
				-[\mathfrak{s}_{\hat{g}}\mathfrak{a}'_{i,i+1}]_{R,\epsilon})
				-\sum_{j\in\ZZ_2}\,([\mathfrak{s}_{\hat{g}}\mathfrak{b}_{j,j+1}]_{R,\epsilon}
				-[\mathfrak{s}_{\hat{g}}\mathfrak{b}'_{j,j+1}]_{R,\epsilon})\\
				&=&
					\sum_{i\in\ZZ_3}\,([\bar{\mathfrak{c}}_{i,i+1}\mathfrak{a}_{i,i+1}]_{R,\epsilon}-
					[\bar{\mathfrak{c}}_{i,i+1}\mathfrak{a}'_{i,i+1}]_{R,\epsilon})
					-0\\
				&=&0,					
			\end{eqnarray*}
			or in other words, $\Psi_1({\hat{g}})$ equals $\Psi'_1({\hat{g}})$ in $\ocobordism_{R,\epsilon}$. This proves
			the statement (2).
		\end{proof}
		
		\begin{lemma}\label{triangularRelation}
			The set-theoretic map $\Psi_h$ on $\hat{\mathcal{B}}^\circ_{K}$
			is a partial homomorphism. In other words,
			\begin{enumerate}
				\item For any $\hat{g}\in\hat{\mathcal{B}}^\circ_{K}$,
					$$\Psi_h(\hat{g})+\Psi_h(\hat{g}^{-1})\,=\,0.$$
				\item For any triple
				$\hat{g}_0,\hat{g}_1,\hat{g}_2\in\hat{\mathcal{B}}^\circ_{K}$
				satisfying the triagular relation $\hat{g}_0\hat{g}_1\hat{g}_2=\id$,
					$$\Psi_h(\hat{g}_0)+\Psi_h(\hat{g}_1)+\Psi_h(\hat{g}_2)\,=\,0.$$
			\end{enumerate}
		\end{lemma}

		\begin{remark}\label{remarkTriangularRelation}
			We point out that in the proof of Lemma \ref{triangularRelation} (2),
			the last equality in Step 1 uses antirotation of tripod pairs with
			opposite charalities (Construction \ref{antirotationConstruction} (1)).
			The presence of the coefficient $2$ there is not only
			indispensible but also crucial for 
			Theorem \ref{theoremPantedCobordism} to work in the integral 
			coefficient case.			
			The corresponding fact is
			that in the paper \cite{KM-Ehrenpreis}, 
			the conclusion of the Second Rotation Lemma (Lemma 8.2) should be
				$$2\,\sum_{i=0}^2(R_i\bar{R}_{i+1})_T=0$$
			if one attempts to state with integral coefficients.			
		\end{remark}
		
		\begin{proof}
			To simplify notations, we prove the statements with respect to $\Psi_1$
			assuming $\hat{g}_i$ in $\tau_h(\hat{\mathcal{B}}^\circ_K)$.
%
%
			
			To prove the statement (1), suppose that
			$\hat{g}$ is an element of
			$\hat{\tau}_h(\hat{\mathcal{B}}_K)$.
			Observe that $\mathfrak{s}_{\hat{g}^{-1}}$ can be
			chosen as the orientation reversed framing
			flipping $\bar{\mathfrak{s}}_{\hat{g}}^*$,
			and that the defining right-handed tripod and $\partial$-framed segment can be chosen as 
			$\mathfrak{a}^*_0\vee\mathfrak{a}^*_{-1}\vee\mathfrak{a}^*_{-2}$
			and $\bar{\mathfrak{b}}^*$ respectively, provided
			that $\mathfrak{s}_{\hat{g}}$, $\mathfrak{a}_0\vee\mathfrak{a}_1\vee\mathfrak{a}_2$,
			and $\mathfrak{b}$ have been chosen to define $\Psi_1(\hat{g})$.
			This should be compared to the wrong
			choice $\bar{\mathfrak{s}}_{\hat{g}}$,
			$\mathfrak{a}_0\vee\mathfrak{a}_1\vee\mathfrak{a}_2$,
			and $\mathfrak{b}$, which is actually
			right for defining $\Psi_1(\hat{c}\hat{g}^{-1})$
			by Lemma \ref{setupData} (4).
			For convenience, we write $\mathfrak{b}_{01}$ and $\mathfrak{b}_{10}$ for $\mathfrak{b}$
			and $\bar{\mathfrak{b}}$, respectively.
			Then $\Psi_1(\hat{g})+\Psi_1(\hat{g}^{-1})$ equals
			\begin{eqnarray*}
				&&	\sum_{i\in\ZZ_3}\,
					([\mathfrak{s}_{\hat{g}}\mathfrak{a}_{i,i+1}]_{R,\epsilon}+[\bar{\mathfrak{s}}^*_{\hat{g}}\mathfrak{a}^*_{-i,-i-1}]_{R,\epsilon})
					-\sum_{j\in\ZZ_2}\,
					([\mathfrak{s}_{\hat{g}}\mathfrak{b}_{j,j+1}]_{R,\epsilon}
					+[\bar{\mathfrak{s}}^*_{\hat{g}}\bar{\mathfrak{b}}^*_{j,j+1}]_{R,\epsilon})\\
				&=&	\sum_{i\in\ZZ_3}\,
					([\mathfrak{s}_{\hat{g}}\mathfrak{a}_{i,i+1}]_{R,\epsilon}+[\bar{\mathfrak{s}}_{\hat{g}}\mathfrak{a}_{-i,-i-1}]_{R,\epsilon})
					-\sum_{j\in\ZZ_2}\,
					([\mathfrak{s}_{\hat{g}}\mathfrak{b}_{j,j+1}]_{R,\epsilon}+
					[\bar{\mathfrak{s}}_{\hat{g}}\bar{\mathfrak{b}}_{j,j+1}]_{R,\epsilon})\\
				&=&	\sum_{i\in\ZZ_3}\,
					([\mathfrak{s}_{\hat{g}}\mathfrak{a}_{i,i+1}]_{R,\epsilon}+[\bar{\mathfrak{s}}_{\hat{g}}\mathfrak{a}_{i+1,i}]_{R,\epsilon})
					-0\\
				&=&	0.	
			\end{eqnarray*}
			This proves the statement (1).

			To prove the statement (2), suppose that
			$\hat{g}_0,\hat{g}_1,\hat{g}_2$ is a triple
			of elements in $\tau_h(\hat{\mathcal{B}}^\circ_{K})$
			satisfying the triagular relation $\hat{g}_0\hat{g}_1\hat{g}_2=\id$.
			From the relation $\hat{g}_0\hat{g}_1\hat{g}_2=\id$, 
			the carrier segments of $\mathfrak{s}_{\hat{g}_i}$ form
			the boundary cycle of an oriented $2$-simplex $\sigma$ in $M$, 
			so by Simplex Subdivision (Definition \ref{axiomsOfConstructions} (2),
			cf.~Subsubsection \ref{Subsubsec-describingAConstruction})
			and by Lemma \ref{FermatPoint},
			there is a left-handed $(100L,\delta)$-tame
			tripod $\mathfrak{t}_0\vee\mathfrak{t}_1\vee\mathfrak{t}_2$,
			such that $\mathfrak{t}_{i,i+1}$ and $\mathfrak{s}_{\hat{g}_{i+2}}$
			are carried by the same segment for all $i$. 
			By Lemma \ref{sharpTripodLemma},
			there are $\phi_0,\phi_1,\phi_2\in\RR/2\pi\ZZ$ with $\phi_0+\phi_1+\phi_2=0$
			and $2\phi_{i+2}$ $\delta$-close to the angle from $\vec{n}_\ter(\mathfrak{t}_i)$
			to $\vec{n}_\ter(\mathfrak{t}_{i+1})$ with respect to $\vec{t}$,
			such that $\mathfrak{t}_{i,i+1}(\phi_{i+2})$ is the same
			as $\mathfrak{s}_{\hat{g}_{i+2}}$ up to $\delta$-small change of the
			initial and terminal framings.			
			In other words,
			we may choose $\mathfrak{s}_{\hat{g}_{i+2}}$
			to be $\mathfrak{t}_{i,i+1}(\phi_{i+2})$ instead.
				
			The strategy is to prove $\Psi_1(\hat{g}_0)+\Psi_1(\hat{g}_1)+\Psi_1(\hat{g}_2)=0$
			using the defining expression by manipulation and cancellation.
			The fundamental case is when $\mathfrak{t}_0\vee\mathfrak{t}_1\vee\mathfrak{t}_2$
			is nearly regular with the terminal framing of legs close to each other,
			or equivalently, when $\ell(\mathfrak{t}_i)$ are close to each other
			and $\phi_i$ are all close to $0$.
			As we can choose the defining right-handed tripod and the defining $\partial$-framed segment to be the same for all
			$\Psi_1(\hat{g}_i)$, we expect the most cancellation in this case and
			expect the triangular relation $\hat{g}_0\hat{g}_1\hat{g}_2=\id$ to be translated
			into a cancellation by an antirotation with distinct chiralities.
			In general, the difficulty to cancellation lies in that
			$\ell(\mathfrak{t}_i)$ may not be close to each other and 
			$\phi_i$ may not be close to $0$. Then we write $\mathfrak{t}_i$ as
			a nearly consecutive tame concatenation $\mathfrak{c}_i\mathfrak{r}_i$,
			in which $\ell(\mathfrak{c}_i)$ are close to each other.
			Apply the Connection Principle to construct a sequence of nearly consecutive tame concatenations 
			$\mathfrak{c}_i\mathfrak{r}^{(0)}_i,\cdots,\mathfrak{c}_i\mathfrak{r}^{(N)}_i$
			so that $\mathfrak{r}^{(k)}_i$ have the same initial and terminal endpoints respectively,
			and $\mathfrak{r}^{(N)}_i$ equals $\mathfrak{r}_i$. 
			Assume that $\mathfrak{r}^{(k)}_i$ and $\mathfrak{r}^{(k+1)}_i$ 
			have length and framings close to each other 
			for $0\leq k<N$, and that the initial and terminal directions of $\mathfrak{r}^{(k)}_i$ are
			close to those of $\mathfrak{r}_i$ for all $k$. 
			In a way similar to $\mathfrak{t}_0\vee\mathfrak{t}_1\vee\mathfrak{t}_2$ and $\hat{g}_i$,
			we obtain a sequence of left-handed tripods 
			$\mathfrak{t}^{(k)}_0\vee\mathfrak{t}^{(k)}_1\vee\mathfrak{t}^{(k)}_2$,
			and sharp elements $\hat{g}^{(k)}_i$.
			We expect to use swapping to obtain the equations
			$\Psi_1(\hat{g}^{(k)}_0)+\Psi_1(\hat{g}^{(k)}_1)+\Psi_1(\hat{g}^{(k)}_2)=
			\Psi_1(\hat{g}^{(k+1)}_0)+\Psi_1(\hat{g}^{(k+1)}_1)+\Psi_1(\hat{g}^{(k+1)}_2)$
			for $0\leq k<N$.
			Once this is done, we can immediately derive the general case 
			by requiring that $\mathfrak{t}^{(0)}_0\vee\mathfrak{t}^{(0)}_1\vee\mathfrak{t}^{(0)}_2$
			falls into the fundamental case.
			Note that we will use the triangular relation $\hat{g}_0\hat{g}_1\hat{g}_2=\id$
			essentially once, and that the number $N$ will 
			essentially depend on the geometry of $\hat{g}_i$.			
			
			\medskip\noindent\textbf{Step 1}. Suppose
			in addition that $\mathfrak{t}_0\vee\mathfrak{t}_1\vee\mathfrak{t}_2$
			is $(\frac{l}2+\frac{I(\frac{\pi}3)}2,10\delta)$-nearly regular
			for some constant $l$ with $\vec{n}_\ter(\mathfrak{t}_i)$ 
			are all $\delta$-close to the unit vector $\vec{n}$, and
			that $\phi_i$ are all $0$. We prove $\Psi_1(\hat{g}_0)+\Psi_1(\hat{g}_1)+\Psi_1(\hat{g}_2)=0$
			in this basic case.
			
			Observe that in this case, for each $\hat{g}_{r+2}$
			where $r\in\ZZ_3$,
			$\mathfrak{s}_{\hat{g}_{r+2}}$
			may be chosen as $\mathfrak{t}_{r,r+1}$,
			and the defining tripod can be chosen
			as $\mathfrak{a}_0\vee\mathfrak{a}_1\vee\mathfrak{a}_2$,
			and the defining $\partial$-framed segment $\mathfrak{b}$
			may be chosen as $\mathfrak{a}_{r,r+1}$.
			Then $\Psi_1(\hat{g}_0)+\Psi_1(\hat{g}_1)+\Psi_1(\hat{g}_2)$ equals
			\begin{eqnarray*}
				&&\sum_{r\in\ZZ_3}
				\left(\left(\sum_{i\in\ZZ_3}[\mathfrak{t}_{r,r+1}\mathfrak{a}_{i,i+1}]_{R,\epsilon}\right)
				-[\mathfrak{t}_{r,r+1}\mathfrak{a}_{r,r+1}]_{R,\epsilon}-[\mathfrak{t}_{r,r+1}\bar{\mathfrak{a}}_{r,r+1}]_{R,\epsilon}
				\right)\\
				&=&\sum_{r\in\ZZ_3}[\mathfrak{t}_{r,r+1}\mathfrak{a}_{r+1,r+2}]_{R,\epsilon}
				+\sum_{r\in\ZZ_3}[\mathfrak{t}_{r,r+1}\mathfrak{a}_{r+2,r}]_{R,\epsilon}
				-\sum_{r\in\ZZ_3}[\mathfrak{t}_{r,r+1}\bar{\mathfrak{a}}_{r,r+1}]_{R,\epsilon}\\
				&=&\sum_{r\in\ZZ_3}[\mathfrak{t}_{r,r+1}\mathfrak{a}_{r+1,r+2}]_{R,\epsilon}
				+\sum_{r\in\ZZ_3}[\mathfrak{t}_{r,r+1}\mathfrak{a}_{r+2,r}]_{R,\epsilon}
				-\sum_{r\in\ZZ_3}[\mathfrak{t}_{r,r+1}\bar{\mathfrak{a}}_{r,r+1}]_{R,\epsilon}\\
				&=&2\sum_{r\in\ZZ_3}[\mathfrak{t}_{r,r+1}\mathfrak{a}_{r,r+1}]_{R,\epsilon}-0\\
				&=&2\sum_{r\in\ZZ_3}[\mathfrak{t}_{r,r+1}\bar{\mathfrak{a}}_{r+1,r}]_{R,\epsilon}\\
				&=&0.
			\end{eqnarray*}
			In the third equality, the last summation equals zero by rotation (Construction \ref{rotationConstruction} (1));
			the first two summations are both equal to the summation of 
			$[\mathfrak{t}_{r,r+1}\mathfrak{a}_{r,r+1}]_{R,\epsilon}$ over $r\in\ZZ_3$, 
			since by swapping (Construction \ref{swapping}),
			\begin{eqnarray*}
				&&[\mathfrak{t}_{01}\mathfrak{a}_{12}]_{R,\epsilon}+
				[\mathfrak{t}_{12}\mathfrak{a}_{20}]_{R,\epsilon}+
				[\mathfrak{t}_{20}\mathfrak{a}_{01}]_{R,\epsilon}\\
				&=&[\mathfrak{t}_{01}\mathfrak{a}_{12}]_{R,\epsilon}+
				[\mathfrak{t}_{12}\mathfrak{a}_{01}]_{R,\epsilon}+
				[\mathfrak{t}_{20}\mathfrak{a}_{20}]_{R,\epsilon}\\
				&=&[\mathfrak{t}_{01}\mathfrak{a}_{01}]_{R,\epsilon}+
				[\mathfrak{t}_{12}\mathfrak{a}_{12}]_{R,\epsilon}+
				[\mathfrak{t}_{20}\mathfrak{a}_{20}]_{R,\epsilon},
			\end{eqnarray*}
			and similarly for $[\mathfrak{t}_{01}\mathfrak{a}_{20}]_{R,\epsilon}+
			[\mathfrak{t}_{12}\mathfrak{a}_{01}]_{R,\epsilon}+
			[\mathfrak{t}_{20}\mathfrak{a}_{12}]_{R,\epsilon}$.
			The last equality follows from antirotation (Construction \ref{antirotationConstruction} (1)).
			This proves $\Psi_1(\hat{g}_0)+\Psi_1(\hat{g}_1)+\Psi_1(\hat{g}_2)=0$
			in the basic case.
			
			\medskip\noindent\textbf{Step 2}. We prove a connecting step
			which is the following claim.
			Suppose that $\mathfrak{c}_0\mathfrak{r}_0\vee\mathfrak{c}_1\mathfrak{r}_1\vee\mathfrak{c}_2\mathfrak{r}_2$ and
			$\mathfrak{c}_0\mathfrak{r}'_0\vee\mathfrak{c}_1\mathfrak{r}'_1\vee\mathfrak{c}_2\mathfrak{r}'_2$
			are $(100L,10\delta)$-tame left-handed tripods satisfying the following:
			\begin{itemize}
				\item The left-handed tripod 
				$\mathfrak{c}_0\vee\mathfrak{c}_1\vee\mathfrak{c}_2$ is $(L+\frac12I(\frac\pi3),\delta)$-nearly
				regular.
				\item For each $i\in\ZZ_3$, the chain 
				$\mathfrak{c}_i,\mathfrak{r}_i$ is $(10\delta)$-consecutive and 
				$(L,10\delta)$-tame.
				The terminal direction of $\mathfrak{c}_i\mathfrak{r}_i$ is
				$\delta$-close to $\vec{t}$. 
				The same holds
				for $\mathfrak{c}_i\mathfrak{r}'_i$.
				\item For each $i\in\ZZ_3$, 
				$\ell(\mathfrak{r}_i)$ is $(10\delta)$-close to $\ell(\mathfrak{r}'_i)$,
				and $\vec{n}_\ter(\mathfrak{r}_i)$ is 
				$(10\delta)$-close to $\vec{n}_\ter(\mathfrak{r}'_i)$.
			\end{itemize}
			Let $\phi_0,\phi_1,\phi_2\in\RR/2\pi\ZZ$ be the angles 
			guaranteed by Lemma \ref{sharpTripodLemma} with respect to
			$\mathfrak{c}_0\mathfrak{r}_0\vee\mathfrak{c}_1\mathfrak{r}_1\vee\mathfrak{c}_2\mathfrak{r}_2$,
			which, hence, 
			works for 
			$\mathfrak{c}_0\mathfrak{r}'_0\vee\mathfrak{c}_1\mathfrak{r}'_1\vee\mathfrak{c}_2\mathfrak{r}'_2$
			up to error of $30\delta$.
			Let $\hat{g}_{i+2},\hat{g}'_{i+2}\in\pi_1(\SO(M),\mathbf{e})$ 
			be the $\delta$-sharp element associated to 
			$(\bar{\mathfrak{r}}_i\mathfrak{c}_{i,i+1}\mathfrak{r}_i)(\phi_{i+2})$,
			$(\bar{\mathfrak{r}}'_i\mathfrak{c}_{i,i+1}\mathfrak{r}_i')(\phi_{i+2})$,
			respectively. Assuming that $\hat{g}_i,\hat{g}'_i$
			lie in $\hat{\mathcal{C}}_{300L}\cap\hat{\mathcal{B}}^\circ_{R-300L}$ for all $i\in\ZZ_3$,
			we claim
				$$\Psi_1(\hat{g}_0)+\Psi_1(\hat{g}_1)+\Psi_1(\hat{g}_2)=
				\Psi_1(\hat{g}'_0)+\Psi_1(\hat{g}'_1)+\Psi_1(\hat{g}'_2).$$
				
			To prove the claim, observe that it suffices to prove a simple case
			that $\mathfrak{r}_i$ equals $\mathfrak{r}'_i$ except for one 
			$i\in\ZZ_3$. Then the claim follows by applying the simple case
			successively to each neighboring pair in the sequence
			of tripods
			$\mathfrak{c}_0\mathfrak{r}_0\vee\mathfrak{c}_1\mathfrak{r}_1\vee\mathfrak{c}_2\mathfrak{r}_2$,
			$\mathfrak{c}_0\mathfrak{r}_0\vee\mathfrak{c}_1\mathfrak{r}_1\vee\mathfrak{c}_2\mathfrak{r}'_2$,
			$\mathfrak{c}_0\mathfrak{r}_0\vee\mathfrak{c}_1\mathfrak{r}_1'\vee\mathfrak{c}_2\mathfrak{r}'_2$,
			$\mathfrak{c}_0\mathfrak{r}'_0\vee\mathfrak{c}_1\mathfrak{r}_1'\vee\mathfrak{c}_2\mathfrak{r}'_2$.
			Without loss of generality, we may assume that $\mathfrak{r}_i=\mathfrak{r}'_i$
			except for $i$ being $0$. Then $\Psi_1(\hat{g}_0)=\Psi_1(\hat{g}'_0)$, and we
			must show $\Psi_1(\hat{g}_1)+\Psi_1(\hat{g}_2)=\Psi_1(\hat{g}'_1)+\Psi_1(\hat{g}'_2)$.
			Observe further that for both $\hat{g}_1$ and $\hat{g}'_1$,
			we may choose the same defining right-handed tripod
			$\mathfrak{a}^{(1)}_0\vee\mathfrak{a}^{(1)}_1\vee\mathfrak{a}^{(1)}_2$
			and the same defining $\partial$-framed segment $\mathfrak{b}^{(1)}$.
			Similarly, for $\hat{g}_2$ and $\hat{g}'_2$
			we choose the same $\mathfrak{a}^{(2)}_0\vee\mathfrak{a}^{(2)}_1\vee\mathfrak{a}^{(2)}_2$
			and	$\mathfrak{b}^{(2)}$ for both $\hat{g}_2$ and $\hat{g}'_2$.
			By swapping the bigon pair 
			$[\mathfrak{r}_0\,(\mathfrak{a}^{(2)}_{01}(-\phi_1)\bar{\mathfrak{r}}_2\mathfrak{c}_{20})]$ and
			$[\mathfrak{r}'_0\,(\bar{\mathfrak{a}}^{(1)}_{01}(\phi_2)\bar{\mathfrak{r}}_1\bar{\mathfrak{c}}_{01})]$
			into
			$[\mathfrak{r}'_0\,(\mathfrak{a}^{(2)}_{01}(-\phi_1)\bar{\mathfrak{r}}_2\mathfrak{c}_{20})]$ and
			$[\mathfrak{r}_0\,(\bar{\mathfrak{a}}^{(1)}_{01}(\phi_2)\bar{\mathfrak{r}}_1\bar{\mathfrak{c}}_{01})]$
			(Construction \ref{swapping}), after rearrangement, we have
			$$[\mathfrak{t}_{01}(\phi_2)\mathfrak{a}^{(1)}_{01}]_{R,\epsilon}+[\mathfrak{t}_{20}(\phi_1)\mathfrak{a}^{(2)}_{01}]_{R,\epsilon}
			\,=\,[\mathfrak{t}'_{01}(\phi_2)\mathfrak{a}^{(1)}_{01}]_{R,\epsilon}+[\mathfrak{t}'_{20}(\phi_1)\mathfrak{a}^{(2)}_{01}]_{R,\epsilon},$$
			where $\mathfrak{t}_{i,i+1}=\bar{\mathfrak{r}}_i\mathfrak{c}_{i,i+1}\mathfrak{r}_{i+1}$
			and $\mathfrak{t}'_{i,i+1}=\bar{\mathfrak{r}}'_i\mathfrak{c}_{i,i+1}\mathfrak{r}'_{i+1}$;
			similarly,
			$$[\mathfrak{t}_{01}(\phi_2)\mathfrak{a}^{(1)}_{12}]_{R,\epsilon}+[\mathfrak{t}_{20}(\phi_1)\mathfrak{a}^{(2)}_{12}]_{R,\epsilon}
			\,=\,[\mathfrak{t}'_{01}(\phi_2)\mathfrak{a}^{(1)}_{12}]_{R,\epsilon}
			+[\mathfrak{t}'_{20}(\phi_1)\mathfrak{a}^{(2)}_{12}]_{R,\epsilon};$$
			$$[\mathfrak{t}_{01}(\phi_2)\mathfrak{a}^{(1)}_{20}]_{R,\epsilon}+[\mathfrak{t}_{20}(\phi_1)\mathfrak{a}^{(2)}_{20}]_{R,\epsilon}
			\,=\,[\mathfrak{t}'_{01}(\phi_2)\mathfrak{a}^{(1)}_{20}]_{R,\epsilon}
			+[\mathfrak{t}'_{20}(\phi_1)\mathfrak{a}^{(2)}_{20}]_{R,\epsilon};$$
			and
			$$-[\mathfrak{t}_{01}(\phi_2)\mathfrak{b}^{(1)}]_{R,\epsilon}-[\mathfrak{t}_{20}(\phi_1)\mathfrak{b}^{(2)}]_{R,\epsilon}
			\,=\,-[\mathfrak{t}'_{01}(\phi_2)\mathfrak{b}^{(1)}]_{R,\epsilon}-[\mathfrak{t}'_{20}(\phi_1)\mathfrak{b}^{(2)}]_{R,\epsilon};$$
			$$-[\mathfrak{t}_{01}(\phi_2)\bar{\mathfrak{b}}^{(1)}]_{R,\epsilon}-[\mathfrak{t}_{20}(\phi_1)\bar{\mathfrak{b}}^{(2)}]_{R,\epsilon}
			\,=\,-[\mathfrak{t}'_{01}(\phi_2)\bar{\mathfrak{b}}^{(1)}]_{R,\epsilon}-[\mathfrak{t}'_{20}(\phi_1)\bar{\mathfrak{b}}^{(2)}]_{R,\epsilon}.$$
			Summing up the five equations above yields 
			$\Psi_1(\hat{g}_1)+\Psi_1(\hat{g}_2)=\Psi_1(\hat{g}'_1)+\Psi_1(\hat{g}'_2)$.
			This finishes the proof the claim of the connecting step.
			
			\medskip\noindent\textbf{Step 3}. We finish the 
			proof in the general case.
			Let $\mathfrak{t}_0\vee\mathfrak{t}_1\vee\mathfrak{t}_2$
			and $\phi_0,\phi_1,\phi_2$ be as before so that $\hat{g}_{i+2}$
			is associated to $\mathfrak{t}_{i,i+1}(\phi_{i+2})$.
			We may write $\mathfrak{t}_i$ as concatenation of
			consecutive $\partial$-framed segments $\mathfrak{c}_i\mathfrak{r}_i$,
			so that $\mathfrak{c}_0\vee\mathfrak{c}_1\vee\mathfrak{c}_2$
			is $(L+\frac12I(\frac{\pi}3),\delta)$-nearly regular. 
			By the Connection Principle (Lemma \ref{connectionPrinciple}),
			we may interpolate a sequence of tripods 
			$\mathfrak{t}^{(k)}_0\vee\mathfrak{t}^{(k)}_1\vee\mathfrak{t}^{(k)}_2$
			where $k$ runs over $0,\cdots,N$,
			such that $\mathfrak{t}^{(k)}_i$ is the $\delta$-concatenation
			$\mathfrak{c}_i\mathfrak{r}^{(k)}_i$, and $\mathfrak{r}^{(k)}_i$ satisfies the
			following.
			\begin{itemize}
				\item For all $i\in\ZZ_3$ and $0\leq k\leq N$,
				$\mathfrak{r}^{(k)}_i$ have length at least $303L$
				and phase $(10\delta)$-close to $0$.
				\item For $i\in\ZZ_3$ and $0\leq k<N$, $\mathfrak{r}^{(k)}_i$
				and $\mathfrak{r}^{(k+1)}_i$
				have length $(10\delta)$-close to each other,
				and terminal framings $(10\delta)$-close to each other.				
				\item For $i\in\ZZ_3$, $\mathfrak{r}^{(0)}_i$ equals
				$\mathfrak{r}_i$.
				\item For $i\in\ZZ_3$, $\mathfrak{r}^{(N)}_i$ have length
				$(10\delta)$-close to each other, and terminal framings $(10\delta)$-close
				to each other.
			\end{itemize}
			For each $\mathfrak{t}^{(k)}_0\vee\mathfrak{t}^{(k)}_1\vee\mathfrak{t}^{(k)}_2$,
			let $\phi^{(k)}_0,\phi^{(k)}_1,\phi^{(k)}_2\in\RR/2\pi\ZZ$ be a triple of 
			angles guaranteed by Lemma \ref{sharpTripodLemma}, and let $\hat{g}^{(k)}_{i+2}\in
			\pi_1(\SO(M),\mathbf{e})$ be the $\delta$-sharp 
			element associated to $\mathfrak{t}^{(k)}_{i,i+1}(\phi_{i+2})$.
			It follows that $\hat{g}^{(k)}_0\hat{g}^{(k)}_1\hat{g}^{(k)}_2=\id$ for $0\leq k\leq N$.
			Moreover, we may assume without loss of generality that
			$\phi^{(0)}_0,\phi^{(0)}_1,\phi^{(0)}_2$ are all $\phi_0,\phi_1,\phi_2$ respectively,
			and that $\phi^{(N)}_0,\phi^{(N)}_1,\phi^{(N)}_2$ are all $0$.
			Note that we may also require $\hat{g}^{(k)}_i$ to lie in 
			$\hat{\mathcal{C}}_{300L}\cap\hat{\mathcal{B}}^\circ_{R-300L}$, for example,
			by constructing $\mathfrak{r}^{(k)}_i$ as a
			$\delta$-consecutive $(L,\delta)$-tame 
			concatenation $\mathfrak{n}^{(k)}_i\tilde{\mathfrak{r}}^{(k)}_i$,
			for $0<k\leq N$,
			where $\mathfrak{n}^{(k)}_i$ is carried by a segment of length $300L$
			with the initial direction $\vec{t}$.
			The $\partial$-framed segment $\mathfrak{n}^{(k)}_i$
			can be constructed using Simplex Subdivision in dimension $0$
			(Definition \ref{axiomsOfConstructions} (2), cf.~Subsubsection \ref{Subsubsec-describingAConstruction}),
			and $\tilde{\mathfrak{r}}^{(k)}_i$ can be constructed by the Connection Principle
			(Lemma \ref{connectionPrinciple}).
			
			Therefore, Step 1 implies that
				$$\Psi_1(\hat{g}^{(N)}_0)+\Psi_1(\hat{g}^{(N)}_1)+\Psi_1(\hat{g}^{(N)}_2)=0,$$
			and Step 2 implies that
				$$\Psi_1(\hat{g}^{(k)}_0)+\Psi_1(\hat{g}^{(k)}_1)+\Psi_1(\hat{g}^{(k)}_2)
				\,=\,\Psi_1(\hat{g}^{(k+1)}_0)+\Psi_1(\hat{g}^{(k+1)}_1)+\Psi_1(\hat{g}^{(k+1)}_2),$$
			for $0\leq k<N$. It follows that when $k$ equals $0$, we have
				$$\Psi_1(\hat{g}_0)+\Psi_1(\hat{g}_1)+\Psi_1(\hat{g}_2)=0.$$
			This completes the proof.			
		\end{proof}
		
		\begin{lemma}\label{homomorphismExtension}
			The restriction of $\Psi_h$ to $\hat{\mathcal{B}}^\circ_D$ extends uniquely to be
			a homomorphism $\Psi_h$ from $\pi_1(\SO(M),\mathbf{e})$ to
			$\ocobordism_{R,\epsilon}$.
		\end{lemma}
		
		\begin{proof}
%
%
%
			This follows from the facts that $\mathcal{B}^\circ_{D}$
			contains a generating set of $\pi_1(M,\pt)$, and
			that $\mathcal{C}_{300L}$ contains $\tau_h(\mathcal{B}^\circ_K)$ which
			further contains a triangular generating set
			of $\pi_1(M,\pt)$ (Lemma \ref{setupData} (2)(4)).
			In fact, $\hat{\mathcal{B}}^\circ_{D}$ contains a generating set 
			of $\pi_1(\SO(M),\mathbf{e})$ which is the preimage of
			the generating set inside $\mathcal{B}^\circ_{D}$, so
			the extension is unique; $\hat{\tau}_h(\hat{\mathcal{B}}^\circ_K)$
			contains a triangular generating set of $\pi_1(\SO(M),\mathbf{e})$
			which is the preimage of the triangular generating
			set inside $\tau_h(\mathcal{B}^\circ_K)$, so the extension exists
			by Lemma \ref{triangularRelation}.
		\end{proof}
		
		By Lemma \ref{homomorphismExtension}, we conclude that the restriction of
		$\Psi_h$ to $\hat{\mathcal{B}}^\circ_{D}$ extends uniquely to be a homomorphism
			$$\Psi_h:\,\pi_1(\SO(M),\mathbf{e})\to\ocobordism_{R,\epsilon}.$$
		Descending to the abelianization, we denote the induced homomorphism as
			$$\Psi_h^{\mathtt{ab}}:\,H_1(\SO(M);\ZZ)\to\ocobordism_{R,\epsilon}.$$
		
		\begin{remark}
			The reader should compare the definition of $\Psi_h$ with the definition
			of the operator $A_T$ in \cite[Subsection 7.1]{KM-Ehrenpreis}.
			Since $A_T$ was defined with a coefficient $\frac12$, it does not work
			in integral coefficients. In fact, the argument of Good Correction Theorem
			\cite[Theorem 3.2]{KM-Ehrenpreis} essentially implies that
			$A_T$ induces an isomorphism $\psi^{\mathtt{ab}}:H_1(S;\QQ)\to \ocobordism_{R,\epsilon}(S)$
			for any closed oriented hyperbolic surface $S$, which is
			the inverse of the homomorphism $\phi:\ocobordism_{R,\epsilon}(S)\to H_1(S;\QQ)$
			given by $\phi([\gamma]_{R,\epsilon})=[\gamma]$.
			By introducing a right-handed tripod $\mathfrak{a}_0\vee\mathfrak{a}_1\vee\mathfrak{a}_2$
			in addition to segment $\mathfrak{b}$, 
			we may get rid of the coefficient $\frac12$
			and write down an expression of $\Psi_h$ with integral coefficients. 
			However, the ambiguity of the 
			choice of $\vec{n}_{\hat{g}}$ makes it necessary to pass to $\SO(M)$
			rather than to stay in $M$.
		\end{remark}

		\subsubsection{Verifications}\label{Subsubsec-verifications}
		It remains to verify that $\Psi_h^{\mathtt{ab}}$ is the inverse of $\Phi$.
		We complete this by proving that $\Psi_h^{\mathtt{ab}}$ is the pre-inverse of $\Phi$ 
		(Lemma \ref{preInverse}),
		and that $\Psi_h$ is onto
		(Lemma \ref{spanning}).
		
		\begin{lemma}\label{preInverse}	
			For any element $\hat{g}$ in $\hat{\mathcal{B}}^\circ_{D}$, 
				$$\Phi(\Psi_h(\hat{g}))\,=\,[\hat{g}]$$
			in $H_1(\SO(M);\ZZ)$. Hence the composition
			$\Phi\circ\Psi_h^{\mathtt{ab}}$ is
			the identity transformation of $H_1(\SO(M);\ZZ)$.
		\end{lemma}
		
		\begin{proof}
			It suffices to prove
			for any $\hat{g}\in\hat{\mathcal{C}}_{300L}\cap\hat{\mathcal{B}}^\circ_{R-300L}$
			with respect to $\Psi_1$.
			Let $\hat{a}_{i,i+1}$ in $\pi_1(\SO(M),\mathbf{e})$ be 
			the $\delta$-sharp element associated to 
			$\mathfrak{a}_{i,i+1}$ for $i\in\ZZ_3$, 
			and $\hat{b}\in\pi_1(\SO(M),\mathbf{e})$ 
			be associated to $\mathfrak{b}$ (Definition \ref{sharpElement}).
			Hence $\hat{c}\hat{b}^{-1}$ is associated to $\bar{\mathfrak{b}}$.
			It is clear from the construction of $\Phi$ that 
			the the image of 
			$[\mathfrak{s}_{\hat{g}}\mathfrak{a}_{i,i+1}]_{R,\epsilon}$
			under $\Phi$ is equal to $[\hat{g}\hat{a}_{i,i+1}]$.
			Similarly, $\Phi([\mathfrak{s}_{\hat{g}}\mathfrak{b}]_{R,\epsilon})=[\hat{g}\hat{b}]$,
			and $\Phi([\mathfrak{s}_{\hat{g}}\bar{\mathfrak{b}}]_{R,\epsilon})=[\hat{g}\hat{c}\hat{b}^{-1}]$.			
			Then in $H_1(\SO(M);\ZZ)$, $\Phi(\Psi_h(\hat{g}))$ equals
			\begin{eqnarray*}
				&&[\hat{g}\hat{a}_{01}]+[\hat{g}\hat{a}_{12}]+[\hat{g}\hat{a}_{20}]
				-[\hat{g}\hat{b}]-[\hat{g}\hat{c}\hat{b}^{-1}]\\
				&=&[\hat{g}]+[\hat{a}_{01}]+[\hat{a}_{12}]+
					[\hat{a}_{20}]-[\hat{c}]\\
				&=&[\hat{g}]+[\hat{a}_{01}\hat{a}_{12}\hat{a}_{20}\hat{c}]\\
				&=&[\hat{g}],
			\end{eqnarray*}
			where $[\hat{a}_{01}\hat{a}_{12}\hat{a}_{20}\hat{c}]=0$
			because $\mathfrak{a}_0\vee\mathfrak{a}_1\vee\mathfrak{a}_2$ is right-handed,
			(Remark \ref{sharpTripodLemmaRemark}).
			This shows $\Phi(\Psi_1(\hat{g}))=[\hat{g}]$.
			The `hence' part
			is because $\hat{\mathcal{B}}^\circ_{D}$ generates $\pi_1(\SO(M),\mathbf{e})$
			(Lemma \ref{setupData} (2)).
		\end{proof}

		\begin{lemma}\label{spanning}
			For any curve $\gamma\in\ocurves_{R,\epsilon}$, 
			the $(R,\epsilon)$-panted cobordism class
			$[\gamma]_{R,\epsilon}$ is equal to an integral
			linear combination of elements in the image
			of $\hat{\mathcal{B}}^\circ_{D}$ under $\Psi_h$.
			Hence the homomorphism 
			$\Psi_h^{\mathtt{ab}}$ surjects $\ocobordism_{R,\epsilon}$.
		\end{lemma}
		
		\begin{proof}
			The `hence' part follows immediately from the main statement since
			$\hat{\mathcal{B}}^\circ_{D}$ generates
			$\pi_1(\SO(M),\mathbf{e})$ (Lemma \ref{setupData} (2)).
			It remains to prove the main statement.
			
			By Lemma \ref{finerCurves}, it suffices to
			assume $\gamma\in\ocurves_{R,10\delta}$.
			Let $\mathcal{A}$ denote the subset of $\pi_1(M,\pt)$
			consisting of elements $u$ such that
			the length of $u$ is at most $\frac{R}2+10L$,
			and that the initial and terminal directions 
			of $u$ is $(10\delta)$-close to 
			$-\vec{t}_\ini(h)$ and $\vec{t}_\ini(h)$,
			respectively. Note that $\tau_h(\mathcal{A})$ is contained in $\mathcal{C}_{300L}\cap\mathcal{B}^\circ_{R-300L}$.
			Let $\hat{\mathcal{A}}$ denote the preimage of $\mathcal{A}$ in $\pi_1(\SO(M),\mathbf{e})$.
			
			\medskip\noindent\textbf{Step 1}. We find $\hat{x}_\pm\in\hat{\tau}_h(\hat{\mathcal{A}})$,
			such that
				$$[\gamma]_{R,\epsilon}=\Psi_1(\hat{x}_-)+\Psi_1(\hat{x}_+).$$
			
			Since $\gamma\in\ocurves_{R,10\delta}$, we may bisect $\gamma$
			into an $(\frac{R}2,5\delta)$-nearly regular bigon $[\mathfrak{s}_-\mathfrak{s}_+]$
			by interpolating a pair of antipodal points with suitably chosen normal vectors.
			By Lemmas \ref{setupData} (3) and \ref{sharpCK},
			we can lift $h$ to be a $\delta$-sharp element $\hat{h}\in\pi_1(\SO(M),\mathbf{e})$
			with an associated $\partial$-framed segment $\mathfrak{h}$.
			By the Connection Principle (Lemma \ref{connectionPrinciple})
			there are oriented $\partial$-framed segments  $\mathfrak{u}_\pm$
			from $p_\ter(\mathfrak{s}_\pm)$ to $\pt$, satisfying the following:
			\begin{itemize}
				\item The length and phase of $\mathfrak{u}_\pm$ are $\delta$-close
				to $2L+I(\frac\pi2)+1$ and $0$ respectively.
				The initial direction $\mathfrak{u}_\pm$ is $\delta$-close to 
				$\pm\vec{n}_\ter(\mathfrak{s}_\pm)\times\vec{t}_\ter(\mathfrak{s}_\pm)$,
				and the initial framing of $\mathfrak{u}_\pm$ is $\delta$-close
				to $\vec{n}_\ter(\mathfrak{s}_\pm)$. The terminal direction of
				$\mathfrak{u}_\pm$ is $\delta$-close to $\vec{t}_\ini(\mathfrak{h})$, 
				and the terminal framing of $\mathfrak{u}_\pm$ is
				$\delta$-close to $\vec{n}_\ini(\mathfrak{h})$.
			\end{itemize}
			See Figure \ref{figOnto1}.
			Let $\mathfrak{x}_\pm=\bar{\mathfrak{w}}_\mp\mathfrak{s}_\pm\mathfrak{w}_\pm$,
			where $\mathfrak{w}_\pm=\mathfrak{u}_\pm\mathfrak{h}$, be the reduced concatenation.
			Then $\mathfrak{x}_+$ and $\mathfrak{x}^*_-$ are associated to
			$(10\delta)$-sharp elements $\hat{x}_+$ and $\hat{x}_-$ 
			in $\hat{\tau}_h(\hat{\mathcal{A}})$.
			
			\begin{figure}[htb]
				\centering
				\includegraphics{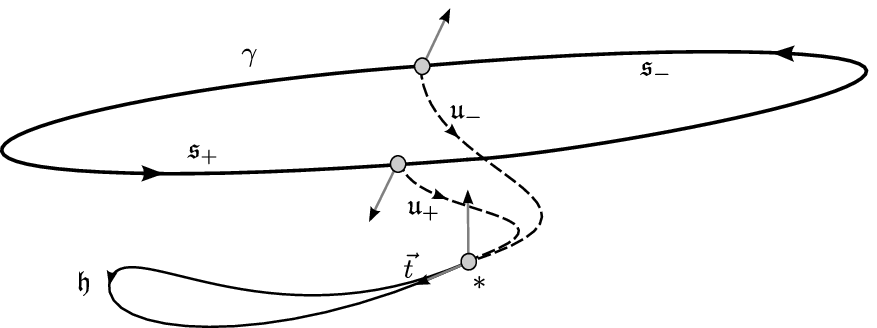}
				\caption{}\label{figOnto1}
			\end{figure}

			Choose a right-handed tripod $\mathfrak{a}_0\vee\mathfrak{a}_1\vee\mathfrak{a}_2$
			and a $\partial$-framed segment $\mathfrak{b}$ for defining $\Psi_1(\hat{x}_+)$, and 
			choose $\mathfrak{a}^*_0\vee\mathfrak{a}^*_{-1}\vee\mathfrak{a}^*_{-2}$
			with the indices understood in $\ZZ_3$ and $\bar{\mathfrak{b}}^*$ for defining 
			$\Psi_1(\hat{x}_-)$.
			Note that for $i\in\ZZ_3$,
			\begin{eqnarray*}
				&&[\mathfrak{x}_+\mathfrak{a}_{i,i+1}]_{R,\epsilon}+
				[\mathfrak{x}_-^*\mathfrak{a}^*_{-i,-i-1}]_{R,\epsilon}\\				
				&=&[\mathfrak{x}_+\mathfrak{a}_{i,i+1}]_{R,\epsilon}+
				[\mathfrak{x}_-\overline{\mathfrak{a}_{i,i+1}}]_{R,\epsilon}\\
				&=&[\mathfrak{s}_+(\mathfrak{w}_+\mathfrak{a}_{i,i+1}\bar{\mathfrak{w}}_-)]_{R,\epsilon}+
				[\mathfrak{s}_-\overline{(\mathfrak{w}_+\mathfrak{a}_{i,i+1}\bar{\mathfrak{w}}_-)}]_{R,\epsilon}\\
				&=&[\gamma]_{R,\epsilon},
			\end{eqnarray*}
			where the last equality follows from splitting $\mathfrak{s}_+\mathfrak{s}_-$ (Construction \ref{splitting}).
			Thus
				$$[\mathfrak{x}_+\mathfrak{a}_{01}]_{R,\epsilon}+
				[\mathfrak{x}_-^*\mathfrak{a}^*_{02}]_{R,\epsilon}=
				[\gamma]_{R,\epsilon};$$
				$$[\mathfrak{x}_+\mathfrak{a}_{12}]_{R,\epsilon}+
				[\mathfrak{x}_-^*\mathfrak{a}^*_{21}]_{R,\epsilon}=
				[\gamma]_{R,\epsilon};$$
				$$[\mathfrak{x}_+\mathfrak{a}_{20}]_{R,\epsilon}+
				[\mathfrak{x}_-^*\mathfrak{a}^*_{10}]_{R,\epsilon}=
				[\gamma]_{R,\epsilon};$$
			Similarly,
				$$-[\mathfrak{x}_+\mathfrak{b}]_{R,\epsilon}
				-[\mathfrak{x}_-^*\bar{\mathfrak{b}}^*]_{R,\epsilon}=
				-[\gamma]_{R,\epsilon};$$
				$$-[\mathfrak{x}_+\bar{\mathfrak{b}}]_{R,\epsilon}
				-[\mathfrak{x}_-^*\mathfrak{b}^*]_{R,\epsilon}=
				-[\gamma]_{R,\epsilon}.$$
			Summing up the five equations above shows that
			$\Psi_1(\hat{x}_+)+\Psi_1(\hat{x}_-)=[\gamma]_{R,\epsilon}.$
				
			\medskip\noindent\textbf{Step 2}. For any $\hat{z}\in\hat{\mathcal{A}}$
			of length at least $30L$
			we find $\hat{y}_\pm\in\hat{\mathcal{A}}$, such that
				$$\hat{z}=\hat{y}_-\hat{y}_+$$
			and that for the images $y_\pm,z\in\pi_1(M,\pt)$ of $\hat{y}_\pm,\hat{z}$
			respectively, both $y_\pm$ have length
			less than $\frac12\ell(z)+10L$.
			
			In fact, we may write the pointed geodesic loop as the concatenation
			of to geodesic segments $\zeta_-\zeta_+$ joint at the midpoint of $z$.
			By the Connection Principle (Lemma \ref{connectionPrinciple}),
			applied to the unframed case simply by ignoring the framings,
			there is a path $\upsilon$ from the midpoint of $z$ to $\pt$, satisfying the
			following: the length of $\upsilon$ is $\delta$-close to $2L+I(\frac{\pi}2)$;
			the initial direction of $\upsilon$ is $\delta$-closely perpendicular to $z$;
			and the terminal direction of $\upsilon$ is $\delta$-close to $\vec{t}$.
			Let $y_-$ and $y_+$ in $\pi_1(M,\pt)$ be $\zeta_-\upsilon$ and $\bar{\upsilon}\zeta_+$,
			respectively. Since $z=y_-y_+$ in $\pi_1(M,\pt)$, we may choose lifts $\hat{y}_\pm$
			of $y_\pm$ in $\pi_1(\SO(M),\mathbf{e})$ so that $\hat{z}=\hat{y}_-\hat{y}_+$.
			It is straightforward to see that $\hat{y}_\pm$ are as desired.
			
			\medskip\noindent\textbf{Step 3}. We complete the proof of the main statement.
			As mentioned above, we may assume that $\gamma\in\ocurves_{R,10\delta}$. By
			Step 1, $[\gamma]_{R,\epsilon}$ can be written as a sum of elements
			in the image of $\hat{\mathcal{A}}$ under $\Psi_h$. By 
			iterately applying Step 2,
			any element in $\Psi_h(\hat{\mathcal{A}})$ can be replaced with 
			a sum of elements of the form $\Psi_h(\hat{y})$ where 
			the image of $\hat{y}\in\hat{\mathcal{A}}$ in $\pi_1(M,\pt)$ 
			has length at most $30L$. In particular, $\hat{y}\in\hat{\mathcal{B}}_{D}$
			since $D$ is assumed to be at least $30L$ (Lemma \ref{setupData} (2)).
			Thus $[\gamma]_{R,\epsilon}$ is equal to an integral
			linear combination of elements in the image
			of $\hat{\mathcal{B}}_{D}$ under $\Psi_h$.
			This completes the proof.
		\end{proof}
		
		\subsection{Proof of Theorem \ref{theoremPantedCobordism}}\label{Subsec-proofOfTheoremPantedCobordism}
		In summary, given any oriented closed hyperbolic $3$-manifold $M$,
		and any universally small positive constant $\epsilon$,
		with the positive constant $R(\epsilon,M)$
		guaranteed by Lemma \ref{setupData} (4), suppose that
		$R$ is a positive constant greater than $R(\epsilon,M)$.
		Then the homomorphism
			$$\Phi:\,\ocobordism_{R,\epsilon}(M)\to H_1(\SO(M),\mathbf{e})$$
		constructed in Subsection \ref{Subsec-Phi} is a canonically defined
		isomorphism (Lemma \ref{wellDefinedPhi} and Subsection \ref{Subsec-theInverseOfPhi}).
		By Lemma \ref{canonicallyDefinedPhi}, for all $[L]_{R,\epsilon}\in\ocobordism_{R,\epsilon}(M)$, 
		the image of $\Phi([L]_{R,\epsilon})$ under the bundle projection is the homology class $[L]\in H_1(M;\ZZ)$.
		This completes the proof of Theorem \ref{theoremPantedCobordism}.
	
\section{Pantifying second homology classes}\label{Sec-pantifyingSecondHomologyClasses}

	In this section, we show that second homology classes of an oriented closed
	hyperbolic $3$-manifold $M$ can be represented by $(R,\epsilon)$-panted
	surfaces, as precisely stated in Theorem \ref{secondHomologyClass}.
	This will imply the absolute case
	of Theorem \ref{homologyViaPants} (1), namely,
	when the collection of curves $\mathcal{L}\subset\ocurves_{R,\epsilon}$
	is empty (Subsection \ref{Subsec-proofOfHomologyViaPants}). 
	Roughly speaking, Theorem \ref{secondHomologyClass} follows
	from inspecting the homology classes of
	the $(R,\epsilon)$-panted
	surfaces constructed in the proof of Theorem \ref{theoremPantedCobordism},
	so our argument and notations will heavily rely on 
	Section \ref{Sec-pantedCobordismGroup}.
	In particular, throughout this section, it will
	suffice to assume $\epsilon$ to be a universally
	small positive constant
	and $R$ to be a positive constant
	greater than the constant $R(\epsilon,M)$
	as guaranteed by Lemma \ref{setupData} (4).
	
	\begin{theorem}\label{secondHomologyClass}
		Let $M$ be an oriented closed hyperbolic $3$-manifold.
		For any small positive $\epsilon$ 
		and sufficiently large positive $R$ depending
		on $M$ and $\epsilon$, the following holds. 
		For any homology class $\alpha\in H_2(M;\,\ZZ)$,
		there exists an (oriented) closed $(R,\epsilon)$-panted subsurface 
		$j:F\looparrowright M$ so that $j_*[F]$ equals $\alpha$.
	\end{theorem}
	
	\begin{remark}
		There is a canonical free resolution of the integral module $\ocobordism_{R,\epsilon}$ given
		by
			$$0\longrightarrow N\longrightarrow \ZZ\,\opants_{R,\epsilon}\stackrel{\partial}\longrightarrow
			\ZZ\,\ocurves_{R,\epsilon}\longrightarrow \ocobordism_{R,\epsilon}\longrightarrow 0,$$
		where $N$ denotes the kernel of of the boundary homomorphism.
		There is also a natural homomorphism $N\to H_2(M;\,\ZZ)$ since the natural homomorphism
		$\ZZ\,\opants_{R,\epsilon}\to H_2(M,|\ocurves_{R,\epsilon}|;\ZZ)$ uniquely
		lifts to $H_2(M;\ZZ)$ restricted to $N$. Therefore, Theorem \ref{secondHomologyClass} 
		asserts that $N$ surjects $H_2(M;\ZZ)$. In this sense,
		it reveals certain finer structure
		of $(R,\epsilon)$-panted cobordisms
		in addition to Theorem \ref{theoremPantedCobordism}.
	\end{remark}
	
	The key idea of the proof of Theorem \ref{secondHomologyClass}
	is to apply a process called \emph{homologous substitution}.
	To illustrate how it works, suppose that $S$ is a connected
	oriented closed surface
	and that $f:S\to M$ is a map so that $f_*[S]$ equals $\alpha$.
	To replace $f:S\to M$ with a homologous $(R,\epsilon)$-panted subsurface
	$j:F\looparrowright M$, we endow $S$ with a triangulation with a single
	vertex $\pt$, and assume that $f$ has been homotoped so that $\pt$ is sent
	to a chosen basepoint of $M$, and that the $1$-simplices of $S$ are long 
	geodesic segments, and that the $2$-simplices of $S$ are totally geodesic
	in $M$. Following the construction of $\Psi$ in Subsection \ref{Subsec-theInverseOfPhi},
	we may replace any (oriented) $1$-simplex $e$ with an $(R,\epsilon)$-multicurve
	$L(e)$ as in the definition of $\Psi$, without passing to the $(R,\epsilon)$-panted
	cobordism class. By convention, we define $L(\bar{e})$ to be
	$\overline{L(e)}$. Moreover, we may replace any $2$-simplex $\sigma$ with an $(R,\epsilon)$-panted
	surface $F(\sigma)$, so that if $\partial\sigma$ is a cycle $e_0,e_1,e_2$, $F(\sigma)$
	will be bounded by $L(e_0)\sqcup L(e_1)\sqcup L(e_2)$.
	The $(R,\epsilon)$-panted surface $F(\sigma)$ can be obtained explicitly by
	the constructions in Lemmas \ref{wellDefinition}, \ref{triangularRelation}.
	Thus the $(R,\epsilon)$-panted surface $F$ can be obtained by naturally gluing 
	the $(R,\epsilon)$-panted surfaces $F(\sigma)$
	along the $(R,\epsilon)$-multicurves $L(e)$ on their boundary, according to the triangulation
	structure of $S$. Intuitively, it should follow from the Spine Principle (Lemma \ref{spinePrinciple})
	that there are natural isomorphisms $H_2(F(\sigma),\partial F(\sigma);\ZZ)\cong H_2(\sigma,\partial\sigma;\ZZ)$.
	Then a Mayer--Vietoris argument will imply that there is a natural
	isomorphism $H_2(F;\ZZ)\cong H_2(S;\ZZ)$ that commutes with the homomorphisms
	$f_*:H_2(S;\ZZ)\to H_2(M;\ZZ)$ and $j_*:H_2(F;\ZZ)\to H_2(M;\ZZ)$.
	In other words, the $(R,\epsilon)$-panted surface $F$ is homologous to $S$ in $M$,
	and hence represents $\alpha$ as desired. In practice, it is actually more convenient not to specify the homology class
	$\alpha\in H_2(M;\ZZ)$. Instead, we consider a triangular presentation complex
	$f:(K,\pt)\to (M,\pt)$ (cf.~Subsection \ref{Subsec-triangularPresentationComplexes})
	of $\pi_1(M,\pt)$ rather than the triangulated $(S,\pt)\to (M,\pt)$.
	Then a similar process of homologous subsitution will yield an $(R,\epsilon)$-panted
	complex $j:\mathcal{K}\looparrowright M$ 
	(a $2$-complex obtained by gluing $(R,\epsilon)$-panted surfaces
	along $(R,\epsilon)$-multicurves on the boundary, cf.~Subsection \ref{Subsec-pantedComplexes}).
	In general, there will be a natural epimorphism $H_2(\mathcal{K};\ZZ)\to H_2(K;\ZZ)$
	that commutes with $f_*$ and $j_*$. Because of
	the easy observation that $f_*:H_2(K;\ZZ)\to H_2(M;\ZZ)$ is onto,
	$j_*:H_2(\mathcal{K};\ZZ)\to H_2(M;\ZZ)$ will also be onto.
	In other words, any homology class $\alpha\in H_2(M;\ZZ)$ comes
	from some $\tilde\alpha\in H_2(\mathcal{K};\ZZ)$, so it
	is represented by some $(R,\epsilon)$-panted surface 
	obtained by a composition $F\to \mathcal{K}\to M$.
		
	The rest of this section is devoted to the 
	proof of Theorem \ref{secondHomologyClass}.
	In Subsections \ref{Subsec-triangularPresentationComplexes}, \ref{Subsec-pantedComplexes},
	we introduce some notations that we will adopt, namely, triangular presentation complexes
	and $(R,\epsilon)$-panted complexes;
	Subsection \ref{Subsec-homologousSubstitution} is the homologous substitution
	argument, which is core	of the proof;
	Subsection \ref{Subsec-proofOfTheoremSecondHomologyClass}
	completes the proof of Theorem \ref{secondHomologyClass}
	by a brief summary.
	
	Let $M$ be an oriented closed $3$-manifold. 
	The bundle of special orthonormal frames of $M$ will be denoted as $\SO(M)$.
	Fix an orthonormal frame $\mathbf{e}\,=\,(\vec{t},\vec{n},\vec{t}\times\vec{n})$
	at a fixed basepoint $\pt$ of $M$, so the special orthonormal bundle $\SO(M)$
	has a preferred basepoint $\mathbf{e}$.
	
	\subsection{Triangular presentation complexes}\label{Subsec-triangularPresentationComplexes}
		Recall that a \emph{presentation} of a group $G$
		is a pair $(\mathcal{S},\mathcal{R})$, where $\mathcal{S}$ 
		is a set of independent letters and $\mathcal{R}$ is a set of words
		in $x$ and $x^{-1}$ for $x\in\mathcal{S}$, such that
		the canonical quotient $\langle\mathcal{S}|\mathcal{R}\rangle$
		is isomorphic to $G$.
		For a presentation of $G$,
		there is a naturally associated CW $2$-complex $K$ with a single
		vertex $\pt$ as the basepoint, called the 
		\emph{presentation complex}, such that the $1$-cells are in correspondence with the
		generators, and the $2$-cells are in correspondence with the relators.
		Hence the fundamental group $\pi_1(K,\pt)$ is naturally isomorphic to $G$.
		In fact, if $(X,\pt)$ is a pointed topological space, any homomorphism
		$\pi_1(K,\pt)\to \pi_1(X,\pt)$ can be realized by a map $(K,\pt)\to (X,\pt)$.
		Combinatorially, any $2$-cell of $K$
		has polygonal boundary,
		and the number of edges is equal to the word length of the corresponding
		relator. A presentation is said to be \emph{finite} if the sets of
		generators and relators are finite. 
		
		\begin{lemma}\label{homologyOnto}
			If $(K,\pt)$ is a presentation complex of $(M,\pt)$, then the presentation map
				$$f:\,(K,\pt)\to (M,\pt)$$
			which induces the natural isomorphism on $\pi_1$ induces an epimorphism
				$$f_*:\,H_2(K;\,\ZZ)\to H_2(M;\,\ZZ).$$			
		\end{lemma}
	
		\begin{proof}
			This follows from the fact that
			$M$ is aspherical. In fact, we may obtain an Eilenberg--MacLane space
			$K'\simeq K(\pi_1(M,\pt),1)$ by attaching to $K$ cells of dimension
			greater than $2$, and we may
			extend $f$ to obtain a homotopy equivalence $f':(K',\pt)\to (M,\pt)$.
			This implies the surjectivity
			of the induced homomorphism $f_*$ on the second homology.
		\end{proof}

		\begin{lemma}\label{liftToSOM}
			If $(K,\pt)$ is a presentation complex of $(M,\pt)$, then the presentation map
				$$f:\,(K,\pt)\to (M,\pt)$$
			lifts to a map
				$$\hat{f}:\,(K,\pt)\to (\SO(M),\mathbf{e}).$$		
		\end{lemma}
	
		\begin{proof}
			This follows from the fact that $\pi_1(\SO(M))$ is the splitting
			extension of $\pi_1(M)$ by $\pi_1(\SO(3),I)\cong \ZZ_2$. 
			Moreover, the homotopy classes of lifts
			of $f$ are determined by the splittings, which are in bijection
			with $H^1(M;\ZZ_2)$ since $\pi_2(\SO(M),\mathbf{e})$ is trivial.			
		\end{proof}

		A presentation is said to be \emph{triangular}, if the word length of
		the relators are at most $3$. The complex $(K,\pt)$ associated to
		a triangular finite presentation is compact with
		only monogonal, bigonal, or triangular $2$-cells, 
		and we say that $(K,\pt)$ is \emph{triangular} and \emph{finite}.
		Note that any finite presentation gives rise to a 
		triangular finite presentation, by adding a maximal 
		collection of mutually non-intersecting
		diagonals to subdivide the $2$-cells of $K$.
		Furthermore, if we assume that $K$ minimizes the number of
		generators and the number of relators 
		in the lexicographical order among triangular finite
		presentations of $G$, it is easy to see that
		$K$ will not contain any monogonal $2$-cells,
		and that any bigonal $2$-cells of $K$ will be attached to $1$-cells 
		representing elements of order $2$.
		
		For our application, $\pi_1(M,\pt)$ has no elements of order $2$, and in fact, it
		is torsion free. It suffices to consider a specific triangular finite presentation
		of $\pi_1(M,\pt)$ associated to the triangular generating set guaranteed by
		Lemma \ref{setupData} (2). In particular,
		the associated triangular finite presentation complex $(K,\pt)$
		is a $\Delta$-complex, and
		we denote the presentation map as
			$$f:\,(K,\pt)\to (M,\pt),$$
		which is unique up to homotopy relative to the basepoint.
		
	\subsection{Panted complexes}\label{Subsec-pantedComplexes}
		By a (topological) \emph{panted complex} $\mathcal{K}$ we mean a compact 
		CW space obtained from a finite disjoint union of
		circles by attaching finitely many disjoint pairs of pants 
		via homeomorphisms from cuffs.
		Recording the pants decomposition of $\mathcal{K}$ 
		as part of data, we will refer to the defining
		circles and pairs of pants as the \emph{structure curves} and
		\emph{structure pants} of $\mathcal{K}$. 
		When a panted complex $\mathcal{K}$
		is immersed into a closed hyperbolic $3$-manifold $M$,
		we will say that the immersion 
		is \emph{$(R,\epsilon)$-panted}, if the restriction to
		each pair of pants is $(R,\epsilon)$-nearly regular up to 
		homotopy, or ambiguously, we will say
		that $\mathcal{K}$ is an \emph{$(R,\epsilon)$-panted complex}.
		Note that any connected $(R,\epsilon)$-panted surface is naturally 
		an $(R,\epsilon)$-panted complex.
		A \emph{panted map}	between two panted complexes is a map that sends each
		structural pair of pants homeomorphically onto a structural pair of pants
		of the target. 
	
		\begin{lemma}\label{pantedRepresentative}
			For any nontrivial element $\alpha\in H_2(\mathcal{K};\,\ZZ)$, 
			there is a closed oriented panted surface $F$ 
			and a panted map $F\to\mathcal{K}$,
			such that the fundamental class of $F$ in $H_2(F;\,\ZZ)$
			is sent to $\alpha$.
		\end{lemma}
		
		\begin{proof}
			Let $\mathcal{P}$ be the disjoint union of 
			structure pants of $\mathcal{K}$,
			and $\mathcal{C}$ be the disjoint union of structure curves of $\mathcal{K}$.
			The long exact sequence of homology yields
			\begin{displaymath}
			\xymatrix{
				0 \ar[r] &H_2(\mathcal{K};\ZZ) \ar[r] &H_2(\mathcal{K},\mathcal{C};\ZZ) \ar[r]^{\partial_*} 
					&H_1(\mathcal{C})\\
				& &H_2(\mathcal{P},\partial\mathcal{P};\ZZ) \ar[r]^{\partial_*} \ar[u]^{\cong} &H_1(\partial\mathcal{P};\ZZ) \ar[u]^{\chi}
			}
			\end{displaymath}
			where $\chi$ means the homomorphism 
			induced by the \emph{characteristic map}
			of the panted complex $\mathcal{K}$ 
			that identifies the cuffs of structure pants 
			with the structure curves. Now any element $\alpha\in H_2(\mathcal{K};\ZZ)$ can be identified 
			as an element $\alpha'$ of $H_2(\mathcal{P},\partial\mathcal{P};\ZZ)$ 
			in the kernel of $\chi\circ\partial_*$.
			Since $H_2(\mathcal{P},\partial\mathcal{P};\ZZ)$ has a basis formed
			by the fundamental classes $[P]$ of the components $P\subset\mathcal{P}$,
			the element $\alpha'$ naturally yields a collection of copies of
			structure pants with suitable orientations, 
			and $\chi(\partial_*[\alpha'])=0$
			implies that these copies of pants can be glued up along cuffs, 
			resulting in a closed oriented panted surface $F$.
			The naturally induced panted map $F\to \mathcal{K}$ is as desired.
		\end{proof}
	
	\subsection{Homologous substitution}\label{Subsec-homologousSubstitution}
		
		Let $(K,\pt)$ be the CW complex	associated to the triangular generating
		set guaranteed by Lemma \ref{setupData} (2),
		(cf.~Subsection \ref{Subsec-triangularPresentationComplexes}),
		and 
		$f:(K,\pt)\to (M,\pt)$ be the basepoint-preserving map associated
		to the presentation.
		Let $\tau_h$ be the conjugation provided by Lemma \ref{setupData} (3). The
		presentation conjugated by $\tau_h$ induces an isomorphism $\tau_h\circ f_{\sharp}:\,
		\pi_1(K,\pt)\to \pi_1(M,\pt)$. We denote the corresponding basepoint-preserving map
		as
			$$f^h:\,(K,\pt)\to (M,\pt).$$
		It can be topologically obtained from $f$ by pushing the image of $\pt\in K$ along the loop 
		corresponding to $h^{-1}\in\pi_1(M,\pt)$.
				
		\begin{lemma}\label{homologousSubstitution}
			There exists a CW complex $K'$ obtained from $K$ by attaching $1$-cells, 
			and a compact panted complex
			$\mathcal{K}$, and there exist maps
				$$(\mathcal{K},\emptyset)\longrightarrow (K',\pt)\longrightarrow (M,\pt)$$
			satisfying the following.
			\begin{itemize}
				\item The map $(K',\pt)\to (M,\pt)$ restricts to be the basepoint-preserving map
					$$f^h:\,(K,\pt)\to(M,\pt).$$
				\item The map $\mathcal{K}\to K'$ induces an epimorphism
					$$H_2(\mathcal{K};\ZZ)\,\to\, H_2(K';\ZZ).$$
				\item The composition $\mathcal{K}\to M$ yields an $(R,\epsilon)$-panted complex.
			\end{itemize}
		\end{lemma}

		We prove Lemma \ref{homologousSubstitution} in the rest of this subsection.

		\subsubsection{Construction of $\mathcal{K}$ and $K'$}

			We construct $\mathcal{K}$, $K'$ and maps
			$(\mathcal{K},\emptyset)\to (K',\pt)\to (M,\pt)$ following
			the construction of 
			$\Psi:\pi_1(\SO(M),\mathbf{e})\to\ocobordism_{R,\epsilon}$
			in Subsection \ref{Subsec-theInverseOfPhi}.
			Fix a lift of $f^h$ into $\SO(M)$, denoted as
				$$\hat{f}^h:\,(K,\pt)\to (\SO(M),\mathbf{e})$$
			(Lemma \ref{liftToSOM}).
			Note that by our assumption,
			$K$ is a $\Delta$-complex,
			so each $1$-cell of $K$ can be conveniently
			denoted by the element $g\in\pi_1(K,\pt)\cong \pi_1(M,\pt)$
			that it represents, and each $2$-cell of $K$ 
			can be conveniently denoted by a triangular
			relation $g_0g_1g_2=\id$ of three generators or their inverses.
			As $K$ and $f^h$ are provided from Lemma \ref{setupData}, for each $1$-cell
			$g$ of $K$, $\hat{f}^h(g)$ is a
			$\delta$-sharp element.
			Thus there
			is an associated oriented $\partial$-framed segment $\mathfrak{s}_{\hat{f}^h(g)}$,
			or simply written as $\mathfrak{s}_{\hat{g}}$ in order
			to be consistent with the notation in Subsection \ref{Subsec-theInverseOfPhi}.

			\medskip\noindent\textbf{Step 1}. For each $1$-cell $g$ of $K$,
			fix a right-handed nearly regular tripod $\mathfrak{a}_0\vee\mathfrak{a}_1\vee\mathfrak{a}_2$
			and an oriented $\partial$-framed segment $\mathfrak{b}$ as in the definition of
			$\Psi$.
			We construct a $2$-dimensional 
			$\Delta$-complex $X(g)$ and a multcurve $L(g)$, together with
			maps 
				$$(L(g),\emptyset)\to (X(g),\pt)\to (M,\pt)$$
			such that the composition yields an
			$(R,\epsilon)$-multicurve $L(g)\looparrowright M$,
			which represents the $(R,\epsilon)$-panted cobordism
			class $\Psi(\hat{g})\in\ocobordism_{R,\epsilon}$,
			where $\hat{g}\in\pi_1(\SO(M),\mathbf{e})$
			is the element $\hat{f}_{\sharp}(g)$. 
			
			The construction is as follows.
			Take the subcomplex $*\cup g$ of $K$;
			attach $1$-cells $a_{i,i+1}$ for $i\in\ZZ_3$ and $b$, corresponding to the
			carrier segments of $\mathfrak{a}_{i,i+1}$ and $\mathfrak{b}$, respectively;
			attach a simplicial $2$-cell with boundary the cycle $a_{01}a_{12}a_{20}$.
			The resulting $2$-complex will be denoted as $X(g)$.
			There is a naturally induced map $(X(g),\pt)\to (M,\pt)$
			extending $f^h|_{\pt\cup g}$. Let $L(g)\,=\,[\mathfrak{s}_{\hat{g}}\mathfrak{a}_{01}]
			\sqcup[\mathfrak{s}_{\hat{g}}\mathfrak{a}_{12}]
			\sqcup[\mathfrak{s}_{\hat{g}}\mathfrak{a}_{20}]
			\sqcup\overline{[\mathfrak{s}_{\hat{g}}\mathfrak{b}]}
			\sqcup\overline{[\mathfrak{s}_{\hat{g}}\bar{\mathfrak{b}}]}$
			be the multicurve of five components, which are the reduced cyclic concatenations
			of the defining oriented $\partial$-framed segments.
			The natural immersion $L(g)\looparrowright M$ can be homotoped to factor through
			$X(g)$ via a composition $(L(g),\emptyset)\to (X(g),\pt)\to (M,\pt)$
			in the explicit way as indicated by the construction. 
			
			We define $L(g^{-1})$ to be $\overline{L(g)}$, and $X(g^{-1})$
			for $X(g)$. Note
			that $L(g^{-1})$ is the representative of $\Phi(\hat{g}^{-1})\in\ocobordism_{R,\epsilon}$
			corresponding to the defining right-handed tripod 
			$\mathfrak{a}^*_2\vee\mathfrak{a}^*_1\vee\mathfrak{a}^*_0$
			and $\mathfrak{b}^*$ since the $\delta$-sharp element $\hat{g}^{-1}$ is associated
			to $\bar{\mathfrak{s}}^*_{\hat{g}}$.
			Thus there are also maps 
				$$(L(g^{-1}),\emptyset)\to (X(g^{-1}),\pt)\to (M,\pt),$$
			and the composition is homotopic to an $(R,\epsilon)$-multicurve 
			$L(g^{-1})\looparrowright M$, which represents $\Psi(\hat{g}^{-1})\in\ocobordism_{R,\epsilon}$.
				
			\medskip\noindent\textbf{Step 2}. For each simplicial $2$-cell $\sigma$ of $K$ 
			corresponding to a triangular relation $g_0g_1g_2=\id$,
			write 
				$$X(\partial \sigma)=X(g_0)\vee X(g_1) \vee X(g_2),$$
			where the wedge is over the basepoint $\pt$,
			and 
				$$X(\sigma)=X(\partial\sigma)\cup\sigma$$
			by identifying the copies of $g_i$ in $X(g_i)$ and $\sigma$,
			and
				$$L(\partial\sigma)=L(g_0)\sqcup L(g_1)\sqcup L(g_2).$$
			We construct a $2$-complex $Y(\sigma)$ obtained from $X(\sigma)$ by 
			attaching $1$-cells
			and a panted surface $F(\sigma)$ bounded by $L(\partial\sigma)$, together with
			maps 
				$$(F(\sigma),\emptyset)\to (Y(\sigma),\pt)\to (M,\pt),$$
			such that the composition is homotopic to an $(R,\epsilon)$-panted surface
			$F(\sigma)\looparrowright M$; moreover, the following diagram of maps
			commutes:
			\begin{displaymath}
			\xymatrix{
					F(\sigma) \ar[r]& Y(\sigma)\\
					L(\partial\sigma) \ar[r] \ar[u]^{\cup}& X(\partial\sigma) \ar[u]^{\cup}.
			}
			\end{displaymath}
			
			The construction is as follows. 
			Since $[L(g_0)]_{R,\epsilon}+[L(g_1)]_{R,\epsilon}+[L(g_2)]_{R,\epsilon}=0$
			in $\ocobordism_{R,\epsilon}$ (Lemma \ref{triangularRelation}), 
			there exists an $(R,\epsilon)$-panted surface $F(\sigma)$ with
			$\partial F(\sigma)$ equal to $L(\partial\sigma)$.
			In fact, the construction of $F(\sigma)$ relies on Lemmas 
			\ref{wellDefinition}, \ref{triangularRelation}.
			Checking the constructions there, we see that $F(\sigma)$
			can be constructed based only on the $\Delta$-complex $X(\sigma)$ and the
			map	$X(\sigma)\to M$ induced from the constructed 
			maps $X(g_i)\to M$ and the given map $f^h|:\sigma\to M$.
			Formally, we regard $X(\sigma)$ as a partially-$\Delta$
			space over $M$ where the partially-$\Delta$ structure
			is given by the entire $\Delta$-complex $X(\sigma)$,
			and the map $X(\sigma)\to M$ is as described above (Definition \ref{partiallyDeltaSpace})
			Then the construction of $F(\sigma)$ implies that there is a partially-$\Delta$
			space $X'$ over $M$ which is an extension of $X(\sigma)$, such that
			there is a commutative diagram of maps:
			\begin{displaymath}
			\xymatrix{
					F(\sigma) \ar[r]& X' \ar[rd] \\
					L(\partial\sigma) \ar[r] \ar[u]^{\cup}& X(\partial\sigma) \ar[u]^{\cup} \ar[r]&M.
			}
			\end{displaymath}
			By the Spine Principle (Lemma \ref{spinePrinciple}),
			$X'$ is $1$-spined over $X(\sigma)$, so we may replace
			$X'$ with a CW complex $Y(\sigma)$, which is obtained from $X(\sigma)$
			by attaching $1$-cells. Then $F(\sigma)$, $Y(\sigma)$ and
			the involved maps from the diagram are as desired.
			
			\medskip\noindent\textbf{Step 3}. Now
			we may naturally attach the disjoint union
			of all $F(\sigma)$ to the disjoint union of all $L(g)$
			according to the attaching maps of $K$. The result
			is a panted complex $\mathcal{K}$. Similarly,
			we may attach the disjoint union of $Y(\sigma)$
			to the disjoint union of all $X(g)$ by naturally identifying
			the copies of $X(g)$ (possibly marked by $g^{-1}$).
			The result is a CW $2$-complex $Y(K)$
			containing the subcomplex $K$. Moreover,
			there are naturally induced maps
				$$(\mathcal{K},\emptyset)\to (Y(K),\pt)\to (M,\pt).$$
			The composition is homotopic to
			an $(R,\epsilon)$-panted complex $\mathcal{K}\looparrowright M$;
			the restriction of $Y(K)\to M$ 
			to $K$ is $f^h$. 
			
			It is clear from
			the construction that $Y(K)$ deformation retracts
			relative to $K$ to a CW subspace
				$$K'\hookrightarrow Y(K),$$
			which can be obtained from $K$ by attaching $1$-cells.
			Therefore, replacing $Y(K)$ with $K'$,
			we obtain maps
				$$(\mathcal{K},\emptyset)\to (K',\pt)\to (M,\pt).$$			
			
		\subsubsection{Verifications}			
			To verify that $\mathcal{K}$, $K'$ and the maps 
			$(\mathcal{K},\emptyset)\to (K',\pt)\to (M,\pt)$
			above are as desired, 
			it suffices to prove that $\mathcal{K}\to K'$
			is surjective on the second homology, as the other
			listed properties are obviously satisfied. 
			By the construction, we may equivalently
			prove with $Y(K)$ instead of $K'$.
			
			Write $\mathcal{C}$ for the disjoint union of all $L(g)$,
			and $X(K^{(1)})$ for the wedge of all $X(g)$ over $\pt$.
			There is a commutative diagram of homomorphisms
			\begin{displaymath}
			\xymatrix{
					0 \ar[r] & H_2(\mathcal{K};\ZZ)\ar[r] \ar[d]^{\phi}
						&H_2(\mathcal{K},\mathcal{C};\ZZ) \ar[r]^{\partial_*} \ar[d]^{\phi''} &H_1(\mathcal{C};\ZZ) \ar[d]^{\phi'}\\
					0 \ar[r]& H_2(Y(K);\ZZ)\ar[r]  
						&H_2(Y(K),X(K^{(1)});\ZZ) \ar[r]^{\partial_*}  &H_1(X(K^{(1)});\ZZ) \\
					0 \ar[r]& H_2(K;\ZZ)\ar[r] \ar[u]_{\iota}^\cong  
						&H_2(K,K^{(1)};\ZZ) \ar[r]^{\partial_*} \ar[u]_{\iota''}^\cong &H_1(K^{(1)};\ZZ) \ar[u]_{\iota'}^\cup					
			}
			\end{displaymath}
			where the rows are part of the long exact sequences of
			homology, and the homomorphisms $\phi,\phi'',\phi'$
			are induced from 
			the map $(\mathcal{K},\mathcal{C})\to (Y(K), X(K^{(1)})$
			of our construction, and the homomorphisms $\iota,\iota'',\iota'$
			are induced from the natural inclusion
			$(K,K^{(1)})\hookrightarrow (Y(K),X(K^{(1)}))$.
			
			Write the quotient map defining $\mathcal{K}$ as
				$$q:\,\bigsqcup_{\sigma\subset K}\,F(\sigma)\,\to\, \mathcal{K}.$$			
			Define a homomorphism
				$$\psi'':\,H_2(K,K^{(1)};\ZZ)\to H_2(\mathcal{K},\mathcal{C};\ZZ)$$
			by assigning 
				$$\psi''([\sigma])\,=\,q_*[F(\sigma)]$$
			where $[F(\sigma)]\in H_2(F(\sigma),L(\partial\sigma);\ZZ)$
			is the fundamental class.
			Define a homomorphism
				$$\psi':\,H_1(K^{(1)};\ZZ)\to H_1(\mathcal{C};\ZZ)$$
			by assigning 
				$$\psi'([g])\,=\,q_*[L(g)]$$
			where $[L(g)]\in H_1(L(g);\ZZ)$ is the fundamental
			class. There is a commutative diagram
			\begin{displaymath}
			\xymatrix{
					H_2(\mathcal{K},\mathcal{C};\ZZ) \ar[r]^{\partial_*} &H_1(\mathcal{C};\ZZ) \\
					H_2(K,K^{(1)};\ZZ) \ar[r]^{\partial_*} \ar[u]^{\psi''} &H_1(K^{(1)};\ZZ) \ar[u]^{\psi'}.
			}
			\end{displaymath}
			
			\begin{lemma}\label{phipsi}
				With the notations above, $\phi'\circ\psi'=\iota'$ and $\phi''\circ\psi''=\iota''$.
			\end{lemma}
			
			\begin{proof}
				It is straightforward to check that $\phi'(\psi'([g]))=\iota'([g])$
				for any $1$-cell $g$ of $K$, by the construction of $L(g)$.
				Since $H_1(K^{(1)};\ZZ)$ is freely generated by $[g]$ where $g$ runs over all $1$-cells
				of $K$, we see that $\phi'\circ\psi'=\iota'$.
				
				We claim that for any $2$-cell $\sigma$ of $K$, 
				$\phi''(\psi''([\sigma]))=\iota''([\sigma])$.
				In fact, applying the discussion to the special case
				when $K$ consists of a single $2$-simplex $\sigma$ together with 
				the $1$-skeleton $\partial\sigma$, we obtain the commutative diagrams 
				\begin{displaymath}
				\xymatrix{
					0 \ar[r] &H_2(F(\sigma),L(\partial\sigma);\ZZ) \ar[r]^{\partial_*} \ar[d]^{\phi''_\sigma} 
						&H_1(L(\partial\sigma);\ZZ) \ar[d]^{\phi'_\sigma}\\
					0 \ar[r] &H_2(Y(\sigma),X(\partial\sigma);\ZZ) \ar[r]^{\partial_*}  &H_1(X(\partial\sigma);\ZZ) \\
					0 \ar[r] &H_2(\sigma,\partial\sigma;\ZZ) \ar[r]^{\partial_*} \ar[u]_{\iota''_\sigma}^\cong 
						&H_1(\partial\sigma;\ZZ) \ar[u]_{\iota'_\sigma}^\cup					}
				\end{displaymath}
				where the rows are exact sequences, and
				\begin{displaymath}
				\xymatrix{
						H_2(F(\sigma),L(\partial\sigma);\ZZ) \ar[r]^{\partial_*} &H_1(L(\partial\sigma);\ZZ) \\
						H_2(\sigma,\partial\sigma;\ZZ) \ar[r]^{\partial_*} \ar[u]^{\psi''_\sigma} &H_1(\partial\sigma;\ZZ) \ar[u]^{\psi'_\sigma}.
				}
				\end{displaymath}
				Because 
					$$\partial_*(\phi''_\sigma\circ\psi''_\sigma([\sigma]))\,=\,\phi'_\sigma\circ\psi'_\sigma(\partial_*([\sigma]))
					\,=\,\iota'_\sigma(\partial_*([\sigma]))=\partial_*(\iota''_\sigma[\sigma]),$$
				the injectivity of $\partial_*$ in this case implies that
					$$\phi''_\sigma\circ\psi''_\sigma([\sigma])\,=\,\iota''_\sigma([\sigma]).$$
				By naturality of Mayer--Vietoris sequences, it follows that
					$$\phi''\circ\psi''([\sigma])\,=\,\iota''([\sigma]),$$
				as claimed.				
				
				Since $H_2(K,K^{(1)};\ZZ)$ is
				freely generated by $[\sigma]$ where $\sigma$ runs over all $2$-cells
				of $K$, we conclude that $\phi''\circ\psi''=\iota''$.			
			\end{proof}
			
			Now an easy diagram chase will show the surjectivity
			of
				$$\phi:\,H_2(\mathcal{K};\ZZ)\,\to\, H_2(Y(K);\ZZ).$$
			In fact, we may identify $H_2(\mathcal{K};\ZZ)$ and $H_2(Y(K))$ as kernels
			of $\partial_*$ in $H_2(\mathcal{K},\mathcal{C};\ZZ)$ and $H_2(Y(K),X(K^{(1)};\ZZ)$,
			respectively. If $\alpha\in H_2(Y(K),X(K^{(1)};\ZZ)$ vanishes under $\partial_*$,
			$\beta=\psi''\circ(\iota'')^{-1}(\alpha)$ 
			in $H_2(\mathcal{K},\mathcal{C};\ZZ)$ also vanishes under $\partial_*$,
			using the injectivity of $\iota'$.
			Moreover, $\phi''(\beta)=\alpha$ by Lemma \ref{phipsi}.
			This implies that $\phi''$ is surjective between the kernels of $\partial_*$,
			or in other words, $\phi$ is surjective.
			
			This completes the verification, and hence completes the proof
			of Lemma \ref{homologousSubstitution}.
			
	\subsection{Proof of Theorem \ref{secondHomologyClass}}\label{Subsec-proofOfTheoremSecondHomologyClass}
		We summarize the proof of Theorem \ref{secondHomologyClass} as follows.
		Let $M$ be a closed oriented hyperbolic $3$-manifold. By Lemma \ref{homologousSubstitution},
		there exists a CW complex $K'$ obtained from a presentation complex $K$ of $\pi_1(M,\pt)$
		by attaching $1$-cells, and a compact panted complex
		$\mathcal{K}$, together with maps
				$$(\mathcal{K},\emptyset)\longrightarrow (K',\pt)\longrightarrow (M,\pt),$$
		satisfying the listed properties. In particular,
		the composed map $\mathcal{K}\to M$ is an $(R,\epsilon)$-panted complex
		in $M$. It is surjective on the second homology by Lemmas \ref{homologousSubstitution}
		and \ref{homologyOnto}.
		In other words, any element $\alpha$ of $H_2(M;\ZZ)$ can be lifted to be an element $\tilde\alpha$
		of $H_2(\mathcal{K};\ZZ)$.
		Moreover, $\tilde\alpha$ can be represented
		by a panted surface via a panted map $F\to\mathcal{K}$ (Lemma \ref{pantedRepresentative}),
		which yields an $(R,\epsilon)$-panted surface via the map $\mathcal{K}\to M$.
		Therefore, the $(R,\epsilon)$-panted surface $F\looparrowright M$ is an representative
		of $\alpha$ as desired. This completes the proof of Theorem \ref{secondHomologyClass}.

\section{Panted connectedness between curves}\label{Sec-pantedConnectedness}
	In this section, we show that in a closed hyperbolic $3$-manifold $M$,
	the collection of $(R,\epsilon)$-curves $\ocurves_{R,\epsilon}$ are
	\emph{$(R,\epsilon)$-panted connected} in the sense of the following
	Proposition \ref{pantedConnectedness}. This will be applied to show
	that an ubiquitous measure of 
	$(R,\epsilon)$-pants $\mu\in\meas(\opants_{R,\epsilon})$
	is irreducible in the proof of Theorem \ref{homologyViaPants} (Subsection
	\ref{Subsec-proofOfHomologyViaPants}).
	
	\begin{proposition}\label{pantedConnectedness}
		Let $M$ be a closed hyperbolic $3$-manifold.
		For any universally small positive $\epsilon$ and any sufficiently large 
		positive $R$ depending
		on $M$ and $\epsilon$, the following holds. 
		For any two curves $\gamma_0,\gamma_1\in\ocurves_{R,\epsilon}$,
		there exists a connected $(R,\epsilon)$-panted subsurface $j:F\looparrowright M$,
		and $\partial F$ contains two
		components homotopic to $\gamma_0$ and $\gamma_1$, respectively.
	\end{proposition}
	
	The proof of Proposition \ref{pantedConnectedness} follows from an easy construction
	using the Connection Principle (Lemma \ref{connectionPrinciple}).
	However, we leave it as a separate section as its statement has
	certain independent geometric interest.	
	
	\begin{proof}
		We say that two curves $\gamma_0,\gamma_1\in\ocurves_{R,\epsilon}$ are 
		\emph{$(R,\epsilon)$-panted connected} if there exists a connected
		$(R,\epsilon)$-panted subsurface $j:F\looparrowright M$ as in the conclusion
		of the proposition. 
		This defines 
		an equivalence relation between curves in $\ocurves_{R,\epsilon}$.
		Suppose that $\epsilon$ universally small and $R$ is sufficiently large,
		for instance, as guaranteed by Lemma \ref{setupData}.
				
		By splitting (Construction \ref{splitting}), every curve in $\ocurves_{R,\epsilon}$
		is $(R,\epsilon)$-panted connected to a curve in $\ocurves_{R,\frac{\epsilon}{10000}}$.
		It suffices to show that any two curves in $\ocurves_{R,\frac{\epsilon}{10000}}$
		are $(R,\epsilon)$-panted connected.
		
		Let $\gamma_0,\gamma_1$ be any two curves in $\ocurves_{R,\frac{\epsilon}{10000}}$.
		By interpolating a pair of points of distance $\frac{R}4$ 
		on $\gamma$ with suitably assigned normal vectors,
		we may decompose $\gamma_0$ into a $(1,\frac{\epsilon}{10000})$-tame bigon
		$[\mathfrak{a}_-\mathfrak{a}_+]$ with $\mathfrak{a}_-$ of length $\frac{R}4$.
		Similarly, we decompose 
		$\gamma_1$ into a $(1,\frac{\epsilon}{10000})$-tame bigon 
		$[\mathfrak{b}_-\mathfrak{b}_+]$ with $\mathfrak{b}_+$ of length $\frac{R}4$.
		By the Connection Principle (Lemma \ref{connectionPrinciple}),
		there exist oriented $\partial$-framed segments
		$\mathfrak{s}$ and $\mathfrak{t}$ of length $(\frac{\epsilon}{10000})$-close
		to $\frac{R}4$ and phase $(\frac{\epsilon}{10000})$-close
		to $0$, so that $\mathfrak{a}_-,\mathfrak{s},\mathfrak{b}_+,\mathfrak{t}$
		form a $(100,\frac{\epsilon}{100})$-tame cycle.
		Let $\gamma'$ be the reduced cyclic concatenation
		$[\mathfrak{a}_-\mathfrak{s}\mathfrak{b}_+\mathfrak{t}]$,
		then $\gamma'\in\ocurves_{R,\frac{\epsilon}{100}}$ by the Length and
		Phase Formula (Lemma \ref{lengthAndPhaseFormula}).
		Furthermore, 
		$[\mathfrak{a}_-\mathfrak{a}_+]$ and $[\mathfrak{a}_-(\mathfrak{s}\mathfrak{b}_+\mathfrak{t})]$
		form an $(\frac{\epsilon}{100})$-swap pair, and 
		$[\mathfrak{b}_-\mathfrak{b}_+]$ and $[(\mathfrak{t}\mathfrak{a}_-\mathfrak{s})\mathfrak{b}_+]$
		form an $(\frac{\epsilon}{100})$-swap pair (Definition \ref{swapPair}).
		By swapping (Construction \ref{swapping}), we see that
		$\gamma_0$ is $(R,\epsilon)$-panted connected with $\gamma'$,
		and that $\gamma'$ is $(R,\epsilon)$-panted connected with
		$\gamma_1$. Thus
		$\gamma_0$ and $\gamma_1$ are $(R,\epsilon)$-panted connected.
		This completes the proof.
	\end{proof}
	
	An application of Proposition \ref{pantedConnectedness} is that we can replace any
	$(R,\epsilon)$-panted surface with a connected one with the same boundary without changing
	its relative homology class.
	
	\begin{lemma}\label{connectedPantedSurface}
		If $M$ is an oriented closed hyperbolic $3$-manifold and $(R,\epsilon)$ are constants
		so that $\ocurves_{R,\epsilon}$ is $(R,\epsilon)$-panted connected
		in the sense of Proposition \ref{pantedConnectedness}, then for any oriented compact
		$(R,\epsilon)$-panted subsurface $j:F\looparrowright M$ bounded by
		an $(R,\epsilon)$-multicurve $L$, there exists an oriented compact
		connected $(R,\epsilon)$-panted subsurface $j':F'\looparrowright M$ bounded by $L$
		such that $j_*[F]$ equals $j'_*[F']$ in $H_2(M,L;\ZZ)$. 
	\end{lemma}
	
	\begin{proof}
		By induction, it suffices to prove the case when $F$ has only two components
		$F_1$ and $F_2$. 
		 
		Without loss of generality, we may assume that each $F_i$ has a nonseparating
		glued cuff $c_i$ of its pants structure. In fact, this is automatically true if 
		$F_i$ is closed. If some $F_i$ is has a nonempty boundary component $c$, there is
		an $(R,\epsilon)$-pants $P$ with a boundary component $c'$ homotopic to $c$.
		Take two oppositely oriented copies $P_\pm$ of $P$ and glue up along the two cuffs other
		that $c'_\pm$. Denote the resulting $(R,\epsilon)$-panted surface as $Q$,
		so $\partial Q$ is $c'_+\sqcup c'_-$. 
		If $P_+$ is has the same orientation as that of $P$,
		we glue up $Q$ and $F_i$ identifying $c'_-$ with $\overline{c}$. The resulting
		$(R,\epsilon)$-panted surface $F'_i$ has the same boundary as that of $F_i$ up to homotopy,
		and $F'_i$ is homologous to $F_i$ in $M$ relative to their boundary.
		After replacing $F_i$ with $F'_i$, each component of $F$ has
		a nonseparating glued cuff $c$ as claimed.
		
		Now suppose that $c_i\subset F_i$ is a nonseparating glued cuff for each $F_i$.
		Let $E_i$ be the $(R,\epsilon)$-panted surface obtained by cutting $F_i$ along $c_i$,
		so $\partial E_i$ has two components $c_{i+}$ and $c_{i-}$ homotopic to $c$ and
		its orientation reversal, respectively.
		By Proposition \ref{pantedConnectedness}, there exists an $(R,\epsilon)$-panted
		surface $W$ with boundary, so that there are two components 
		$d_1$ and $d_2$ of $\partial W$ homotopic to $c_1$ and $c_2$, respectively.
		Take a copy $W_+$ of $W$ and a copy $W_-$ of the orientation reversal of $W$.
		Denote the components of $\partial W_\pm$ corresponding to $d_i$ as $d_{i\pm}$.  
		We glue up $W_\pm$ along the opposite pairs of boundary components 
		other than $d_{i\pm}$, and glue them with $E_i$ by identifying $d_{i\pm}$
		with $c_{i\mp}$, respectively. The resulting $(R,\epsilon)$-panted surface $F'$
		has the same boundary as that of $F$, and $F'$ is homologous to
		$F$ in $M$ relative to their boundary. Since $F'$ is connected by the
		construction, we see that $F'$ is a connected $(R,\epsilon)$-panted 
		surface as desired.
	\end{proof}
		
\section{Bounded quasi-Fuchsian subsurfaces}\label{Sec-boundedQuasiFuchsianSubsurfaces}
	In this section, we prove Theorem \ref{homologyViaPants} by
	applying Theorem \ref{theoremPantedCobordism} and Propositions \ref{secondHomologyClass},
	\ref{pantedConnectedness} (Subsection \ref{Subsec-proofOfHomologyViaPants});
	we prove Theorem \ref{main-qfSurface} by applying Theorems \ref{gluing}
	and \ref{homologyViaPants}, following the methodology of 
	Section \ref{Sec-methodology}; we prove Theorem \ref{main-pantedSurface} 
	by applying Theorems \ref{theoremPantedCobordism}, \ref{secondHomologyClass} 
	and Proposition \ref{pantedConnectedness}. 
	Subsection \ref{Subsec-descriptionOfTheProblem}
	contains some remarks about the formulation of Theorem \ref{main-qfSurface}

	\subsection{Description of the problem}\label{Subsec-descriptionOfTheProblem}
		Generally speaking, the construction problem
		of geometrically finite surface subgroups in a Kleinian group
		is concerned about the existence of such surface subgroups subject
		to various conditions. For instance, one may ask about the existence
		specifying the boundary of the subsurface, or requiring it to represent
		some homology class. We will allow ourselves to pass to finite covers of the designated
		boundary loops or positive multiples of the homology class,
		and we will only look for immersed subsurfaces rather than embedded ones.
		In many motivating applications,
		these are the usual assumptions under similar circumstances.
		On the other hand, we will only consider closed hyperbolic $3$-manifolds.
		This is due to the restriction of our current techniques, and
		it would certainly be interesting to study
		the construction problem for cusped hyperbolic $3$-manifolds.
		
		Let $M$ be a closed hyperbolic $3$-manifold. Suppose $\gamma_1,\cdots,\gamma_m$
		are $\pi_1$-injectively immersed loops $\gamma_i:S^1\looparrowright M$,
		and let $L:\sqcup^mS^1\looparrowright M$ be their disjoint union:
			$$L=\gamma_1\sqcup\cdots\sqcup\gamma_m.$$
		The relative homology of the mapping cone $H_*(M\cup_L\vee^mD^2,\vee^mD^2;\ZZ)$ is well defined, depending only
		on the free homotopy classes of the loops. Without loss of generality,
		we may hence assume that they are embedded.
		Identifying $L$ with its image, the relative homology of the cone becomes
		$H_*(M,L;\ZZ)$ by excision. If
			$$j:(F,\partial F)\looparrowright (M,L)$$
		is an immersion of
		an oriented compact surface $F$, then $F$ naturally represents a
		relative homology class
			$$j_*[F,\partial F]\,\in\, H_2(M,L;\,\QQ),$$
		where the rational coefficient is taken since we are not interested in
		the torsion. Moreover, it is clear that the restriction of $j$
		on $\partial F$ is a covering if $F$ has no disk or sphere component
		and if $j$ is $\pi_1$-injective on each component of $F$.
		
		With the notations above, the construction problem that we are concerned about in this paper
		is the following:
		
		\begin{question}
			For any element $\alpha\in H_2(M,L;\,\QQ)$, is there a positive multiple
			of $\alpha$ represented by a connected oriented surface
			$F$ and a $\pi_1$-injective quasi-Fuchsian immersion
			$j$?
		\end{question}
		
		The answer is affirmative as stated in Theorem \ref{main-qfSurface}. 
		
	\subsection{Proof of Theorem \ref{homologyViaPants}}\label{Subsec-proofOfHomologyViaPants}
		We summarize the proof of Theorem \ref{homologyViaPants}.
		Note that the `furthermore' part follows immediately
		from the main statements because the boundary homomorphism
		$\partial$ is integral coefficiented over the natural basis
		of $\meas(\opants_{R,\epsilon})$ and $\meas(\ocurves_{R,\epsilon})$.
				
		To prove the first statement, it suffices to find an $(R,\epsilon)$-panted surface
		representing any class $\alpha\in H_2(M,|\mathcal{L}|;\QQ)$ up to a scalar multiple.
		We may assume without loss of generality that $\alpha$ is integral.
		Under the boundary homomorphism
			$$\partial: H_2(M,|\mathcal{L}|;\ZZ)\to H_1(|\mathcal{L}|;\ZZ),$$
		$\partial{\alpha}$ can be represented by a multicurve $L$ all of whose components
		are carried by components of $|\mathcal{L}|$. 
		Under the composition of the 
		canonical isomorphism $\ocobordism_{R,\epsilon}\cong H_1(\SO(M);\ZZ)$ (Theorem \ref{theoremPantedCobordism}) 
		and	the projection $H_1(\SO(M);\ZZ)\to H_1(M;\ZZ)$,
		$[L]_{R,\epsilon}$ is sent to $0$ since $L$ is a boundary.
		This means $2[L]_{R,\epsilon}$ is $0$ in $\ocobordism_{R,\epsilon}$,
		so there is an $(R,\epsilon)$-panted surface $F_L$ with boundary $2L$.
		Let
			$$\beta\,=\,2\alpha-[F],$$
		in $H_2(M,|\mathcal{L}|;\ZZ)$. Since $\partial\beta=0$ in $H_1(|\mathcal{L}|;\ZZ)$,
		$\beta$ can be regarded as an element of $H_2(M;\ZZ)$. By Theorem \ref{secondHomologyClass},
		$\beta$ is represented by an $(R,\epsilon)$-panted surface $F_\beta$.
		Therefore $2\alpha$ is represented by the $(R,\epsilon)$-panted surface $F_L\sqcup F_\beta$.
		
		To prove the second statement, 
		recall that by \cite[Theorem 3.4, cf.~Section 4]{KM-surfaceSubgroup},
		there is a measure of $(R,\epsilon)$-nearly regular pants
		$\mu\in\meas(\opants_{R,\epsilon})$ which is $(R,\frac{\epsilon}2)$-nearly 
		evenly footed (Definition \ref{gluingConditions}),
		such that $\partial\mu$ is positive on any curve $\gamma\in\ocurves_{R,\epsilon}$.
		Note that the existence of $\mu$ 
		relies on the assumption that $\epsilon$ is bounded by the injectivity radius of $M$,
		\cite[Subsection 4.4]{KM-surfaceSubgroup}.		
		Let $\mu_1$ be the sum $\mu+\bar{\mu}$, where $\bar{\mu}(\{\Pi\})=\mu(\{\bar\Pi\})$,
		so $\mu_1\in\bmeas(\opants_{R,\epsilon})$.
		Let $\mu_2\in\bmeas(\opants_{R,\epsilon})$ be the sum of all $\Pi$ for all $\Pi\in\opants_{R,\epsilon}$.
		For some sufficiently large positive integer $N$, let
			$$\mu_0\,=\,N\mu_1+\mu_2.$$
		Then $\mu_0\in\bmeas_{R,\epsilon}$ is ubiquitous,
		$(R,\epsilon)$-nearly evenly footed. It is certainly rich as $\partial^\flat\mu_0$ vanishes
		in this case. By Proposition \ref{pantedConnectedness}, $\mu_0$ is irreducible. In fact,
		suppose otherwise that $\mu_0$ were the sum $\mu'+\mu''$, such that $\partial\mu'$ and $\partial\mu''$
		have disjoint supports on $\ocurves_{R,\epsilon}$. 
		Then for $\gamma',\gamma''\in\ocurves_{R,\epsilon}$ lying in the
		supports of $\partial\mu'$ and $\partial\mu''$ respectively,
		$\gamma'$ and $\gamma''$ cannot appear simultaneously on an $(R,\epsilon)$-panted
		surface $F$ whose pants are all from the support of $\mu_0$. However,
		$\mu_0$ is ubiquitous so the support of $\mu_0$ is $\opants_{R,\epsilon}$.
		Thus we reach a contradiction 
		since $\gamma'$ and $\gamma''$ should be $(R,\epsilon)$-panted connected
		in the sense of the conclusion of Proposition \ref{pantedConnectedness}.
		Therefore, $\mu_0\in\bmeas(\opants_{R,\epsilon})$ is a measure as claimed in the 
		second statement of Theorem \ref{homologyViaPants}.
		
		It remains to prove that if $\xi\in\zmeas(\opants_{R,\epsilon},|\mathcal{L}|)$,
		then for some sufficiently large integer $m$,
			$$\xi'=\xi+m\mu_0$$
		in $\zmeas(\opants_{R,\epsilon},|\mathcal{L}|)$ is ubiquitous, irreducible,
		$(R,\epsilon)$-nearly evenly footed, and rich. 
		Technically, one may assume 
		here that $\mu_0$ is $(R,\frac{\epsilon}2)$,
		and such a $\mu_0$ can be achieved by
		using $\frac\epsilon2$ instead of $\epsilon$ 
		in the construciton above.
		It is clear that $\xi'$	is ubiquitous and irreducible as so is $\mu_0$. 
		On the other hand, $\mu_0$ being ubiquitous also implies that $\partial^\sharp\mu_0$ 
		is positive on $\vistorus_\gamma$ for any curve $\gamma$, so when $m$
		is sufficiently large, the normalization of the measure $\partial^{\sharp}\xi'$
		can be $(\frac{\epsilon}{2R})$-equivalent to 
		the normalization of $\partial^\sharp\mu_0$ restricted to $\vistorus_\gamma$,
		for all $\gamma\in\ocurves_{R,\epsilon}$. Hence 
		$\xi'$ is $(R,\epsilon)$-nearly evenly footed. It is also 
		clear that $\xi'$ is rich if $m$ is so large that
		$\partial\xi'(\{\gamma\})$ is less than, for instance, $\frac{m}3\partial\mu_0(\{\gamma\})$
		for all $\gamma\in\ocurves_{R,\epsilon}$.
		This completes the proof of the second statement,
		and hence the proof of Theorem \ref{homologyViaPants}.
				
	\subsection{Proof of Theorem \ref{main-qfSurface}}
		We derive Theorem \ref{main-qfSurface} from Theorems \ref{homologyViaPants}
		and \ref{gluing} as follows.
		
		Let $M$ be a closed hyperbolic $3$-manifold, and $L\subset M$
		be the union of finitely many mutually disjoint, $\pi_1$-injectively embedded
		loops. 
		
		\begin{lemma}\label{reductionToOrientable}
			If Theorem \ref{main-qfSurface} is true
			when $M$ is orientable, 
			it holds in the general case as well.
		\end{lemma}
		
		\begin{proof}
			Assume that Theorem \ref{main-qfSurface} is true
			for the orientable case.
			If $M$ is not orientable, 
			we consider an orientable $2$-fold cover $\kappa:\tilde{M}\to M$,
			and let $\tilde{L}$ be the preimage of $L$.
			By excision, $H_2(M,L;\QQ)$ is isomorphic to $H_2(N;\QQ)$ where
			$N$ is the compact $3$-manifold with tori boundary obtained 
			from $M$ with $L$ removed, and similarly, 
			$H_2(\tilde{M},\tilde{L};\QQ)$ is isomorphic to $H_2(\tilde{N};\QQ)$
			where $\tilde{N}$ is the $2$-fold cover of $N$ obtained from $\tilde{M}$
			with $\tilde{L}$ removed.
			Because $H_2(\tilde{N};\QQ)$ surjects $H_2(N;\QQ)$ under the covering,
			the same holds for $H_2(\tilde{M},\tilde{L};\QQ)$ and $H_2(M,L;\QQ)$.
			Therefore, for any element $\alpha\in H_2(M,L;\QQ)$, 
			we may take an element $\tilde{\alpha}\in H_2(\tilde{M},\tilde{L};\QQ)$
			which is projected to be $\alpha$. The orientable
			case implies that a positive integral multiple of $\tilde\alpha$
			can be represented by an orientable compact $\pi_1$-injectively immersed
			quasi-Fuchsian surface $\tilde{j}:(F,\partial F)\looparrowright (\tilde{M},\tilde{L})$,
			so $\kappa\circ\tilde{j}:(F,\partial F)\looparrowright (M,L)$ 
			represents a positive integral multiple
			of $\alpha$ as desired.
		\end{proof}
			
		By Lemma \ref{reductionToOrientable}, we may assume without loss of generality
		that $M$ is an oriented closed hyperbolic $3$-manifold.
		
		\begin{lemma}\label{goodCurveCovers}
			With the notations of Theorem \ref{main-qfSurface},
			for any constant $\epsilon>0$, and for any constant $\hat{R}>0$,
			there exist some $R>\hat{R}$, 
			such that every component of $L$ has
			a finite cyclic cover which is homotopic to
			a curve of $\ocurves_{R,\epsilon}$.
		\end{lemma}
		
		\begin{proof}
			Let $c_1,\cdots,c_n$ be the components of $L$. Let
			$\ell_k\in(0,+\infty)$ be the length of $c_k$,
			and $\varphi_k\in\RR\,/\,2\pi\ZZ$ 
			be the phase of $c_k$. Then for any positive integer $m_k$,
			the length and phase of the $m_k$-fold cyclic cover of 
			$c_k$ are $m_k\ell_k$ and $m_k\varphi_i$, respectively.
			We must show that for any positive constant $\epsilon$ and $\hat{R}$
			as above, there exists $R>\hat{R}$ and positive integers
			$m_1,\cdots,m_n$, such that 
			the complex numbers $m_k(\ell_k+\varphi_k\imunit)$
			are all $\epsilon$ close to the real number $R$.
			
			Consider the complex torus $T$ obtained by $\CC^n$
			modulo the lattice spanned over $\ZZ$
			by $2\pi\imunit\,\vec{e}_k$,
			and $(\ell_k+\varphi_k\imunit)\,\vec{e}_k$, for $k=1,\cdots,n$,
			where $\vec{e}_1,\cdots,\vec{e}_n$ denotes 
			the standard basis. The origin of $\CC^n$ is projected to be
			the basepoint $\pt$ of $T$. The diagonal
			ray 
				$$\vec{v}:[0,+\infty)\to\CC^n$$
			defined by $\vec{v}(r)=r(\vec{e}_1+\cdots+\vec{e}_n)$
			is projected to be a ray 
				$$v:[0,+\infty)\to T$$ 
			start from $v(0)=\pt$. Then to find $R$ and $m_1,\cdots,m_n$
			as above, it suffices to show that for any positive constants
			$\epsilon$ and $\hat{R}$, there exists
			$R>\hat{R}$ so that $v(R)$ is $\epsilon$-close to $\pt$.
			However, this is a well known fact, which can be
			derived easily from dynamics
			of geodesic flow on a Euclidean torus, 
			so we omit the details.			
		\end{proof}
		
		Choose a sufficiently small positive $\epsilon$ depending on $M$
		and a sufficiently large $R$
		depending on $M$ and $\epsilon$ so that Lemma \ref{goodCurveCovers}, 
		and Theorems \ref{gluing} and \ref{homologyViaPants} can be applied.
		By Lemma \ref{goodCurveCovers}, there is a finite cover
		$\tilde{L}$ of $L$ with all components homotopic 
		to curves of $\ocurves_{R,\epsilon}$. 
		Let $|\mathcal{L}|\subset|\ocurves_{R,\epsilon}|$ be
		all the unoriented curve 
		classes which are realized by some 
		component of $\tilde{L}$.
		
		Consider the essential case
		when $\tilde{L}$ does not have two components that
		are homotopic to each other up to the orientation reversion.
		In this case, $|\mathcal{L}|$ is in correspondence with components of
		$\tilde{L}$. The relative homology
		$H_2(M,|\mathcal{L}|;\ZZ)$ defined
		in Subsection \ref{Subsec-homologyViaPants}
		is naturally isomorphic to $H_2(M,\tilde{L};\ZZ)$ 
		(cf.~Subsection \ref{Subsec-descriptionOfTheProblem}), and that there
		is a natural homomorphism $H_2(M,\tilde{L};\ZZ)\to H_2(M,L;\ZZ)$
		induced by the covering $\tilde{L}\to L$.
		We may realize any homology class
		$\alpha\in H_2(M,L;\ZZ)$ as follows.
		Passing to some large multiple of $\alpha$,
		we may lift $\alpha$ to be an element 
			$$\tilde{\alpha}\in H_2(M,|\mathcal{L}|;\ZZ).$$
		Then by Theorem \ref{homologyViaPants} (1), 
		there exists a measure $\mu\in\zmeas(\opants_{R,\epsilon},|\mathcal{L}|)$
		which is mapped onto $\tilde\alpha$, and by Theorem \ref{homologyViaPants} (2),
		we may assume that $\mu$ is ubiquitous,
		irreducible, $(R,\epsilon)$-nearly evenly footed, and rich. 
		Then there exists an oriented, connected, compact,
		$\pi_1$-injectively immersed quasi-Fuchsian subsurface: 
				$$j:\,F\looparrowright M,$$
		which is $(R,\epsilon)$-nearly regularly
		panted subordinate to a positive integral multiple of $\mu$,
		by Theorem \ref{gluing}. It is clear from the construction that (up to homotopy) $[F]$
		can be regarded as an element in $H_2(M,L;\ZZ)$ representing a positive multiple of $\alpha$.
		
		From the essential case above one can derive
		the general case when $\tilde{L}$ may have components that
		are homotopic to each other up to the orientation reversion.
		We will only sketch the argument, as the tricks used below 
		should be easy and less important. 
		
		Let $\tilde{L}_0$ be a maximal subunion of
		components of $\tilde{L}$ so that $\tilde{L}_0$ 
		does not have two components that
		are homotopic to each other up to the orientation reversion.
		As in the essential case, $|\mathcal{L}|$ is in correspondence with components of
		$\tilde{L}$, so $H_2(M,|\mathcal{L}|;\ZZ)$ can be realized
		as $H_2(M,\tilde{L}_0;\ZZ)$. There is also a natural homomorphism
		$H_2(M,\tilde{L}_0;\ZZ)\to H_2(M,\tilde{L};\ZZ)$ induced by the inclusion
		$\tilde{L}_0\subset \tilde{L}$. It is easy to verify that
		$H_2(M,\tilde{L};\ZZ)$ is generated by $H_2(M,\tilde{L}_0;\ZZ)$
		together with all $[A]$ where $A\looparrowright M$ 
		is any annulus between two homotopic components of $\tilde{L}_0$.
		If the boundary of any such annulus $A$ is a curve $\gamma\in\ocurves_{R,\epsilon}$,
		we may take a null-homologous $(R,\epsilon)$-panted surface $E$
		which has a glued cuff $c$ homotopic to $\gamma$. Such an $E$ can be
		obtained, for example, by gluing pairs of pants prescribed by 
		an ubiquitous $\mu_0\in\bmeas(\opants_{R,\epsilon})$.
		Cutting $E$ along $c$ and homotoping 
		the two boundary components to the two components
		of $\partial A$ gives rise to an $(R,\epsilon)$-panted surface $E_A$
		such that $[E_A]$ equals $[A]$ as in $H_2(M,\tilde{L};\ZZ)$.
		Now for any homology class $\alpha\in H_2(M,L;\QQ)$, we may pass to 
		a positive integral multiple of $\alpha$ and lift it as $\tilde{\alpha}
		\in H_2(M,\tilde{L};\QQ)$.
		From the above $\tilde{\alpha}$ equals the sum of 
		some $\tilde{\alpha}_0\in H_1(M,\tilde{L}_0;\QQ)$ together with
		some positive rational multiple $[E_A]$. Possibly after passing to
		further positive multiple, the essential case
		implies that $\tilde{\alpha}_0$
		can be represented by an $(R,\epsilon)$-panted surface $F_0$ with
		$\partial F_0$ mapped to $\tilde{L}_0$. 
		We may also represent the difference term
		$\tilde{\alpha}-\tilde{\alpha}_0$ 
		by a union of $(R,\epsilon)$-panted surfaces $E_{A_1},\cdots,E_{A_s}$
		as above. Thus the union 
		$F=F_0\cup E_{A_1}\cup\cdots\cup E_{A_s}$ 
		is an $(R,\epsilon)$-panted surface representing $\tilde{\alpha}$
		in $H_2(M,\tilde{L};\QQ)$. We may assume that $F_0$ and each $E_{A_1},\cdots,E_{A_s}$
		to be obtained by a $(R,\epsilon)$-nearly unit-shearing gluing
		of a collection of $(R,\epsilon)$-pants prescribed by
		a ubiquitous, irreducible, $(R,\epsilon)$-nearly evenly footed,
		and rich measure, then modifying the gluing by a hybriding argument
		(cf.~Lemma \ref{hybriding})
		will yield a connected $\pi_1$-injectively immersed quasi-Fuchsian 
		surface $F'$, which still represents $\tilde{\alpha}$.
		This completes the argument of the general case,
		and hence completes the proof of Theorem \ref{main-qfSurface}.

	\subsection{Proof of Theorem \ref{main-pantedSurface}}
		We derive Theorem \ref{main-pantedSurface} from Theorems \ref{theoremPantedCobordism},
		\ref{secondHomologyClass} and Proposition \ref{pantedConnectedness} as follows.
		We point out that as those results rely only on the constructions
		of $\partial$-framed segments (Section \ref{Sec-basicConstructions}),
		the input from dynamics necessary for the proof
		is only	the mixing property of the frame flow on the closed hyperbolic $3$-manifold $M$,
		but not the fact that the mixing rate is exponential.
		
		The invariant $\sigma$ can be defined for any null-homologous $(R,\epsilon)$-multicurve
		as
			$$\sigma(L)\,=\,\Phi([L]_{R,\epsilon}),$$
		where $\Phi:\ocobordism_{R,\epsilon}\to H_1(\SO(M);\ZZ)$ is the canonical isomorphism 
		by Theorem \ref{theoremPantedCobordism}. Note that since $L$ is null homologous,
		$\sigma(L)$ lies in the canonical submodule of $H_1(\SO(M);\ZZ)$ coming from the center
		$\ZZ_2$ of $\pi_1(\SO(M))$. Hence $\sigma(L)$ has well defined value in $\ZZ_2$.
		
		It follows that $\sigma(L_1\sqcup L_2)$ equals $\sigma(L_1)+\sigma(L_2)$
		because $\Phi$ is a homomorphism. It also follows that $\sigma(L)$ vanishes 
		if and only if $L$ is the boundary of an $(R,\epsilon)$-panted subsurface $F$ of $M$.
		Moreover, Proposition \ref{pantedConnectedness} implies that we may
		assume $F$ to be connected (Lemma \ref{connectedPantedSurface}).
		
		It remains to show the last statement in the conclusion of Theorem \ref{main-pantedSurface}.
		Let $L$ be an $(R,\epsilon)$-multicurve
		with vanishing $\sigma(L)$. Fix an $(R,\epsilon)$-panted surface $F_0$ bounded by
		$L$, and denote the relative homology class of $F_0$ as $\alpha_0\in H_2(M,L;\ZZ)$.
		For any homology class $\alpha\in H_2(M,L;\ZZ)$ with $\partial \alpha$ equal to
		$[L]\in H_1(L;\ZZ)$, there exists some $\beta\in H_2(M;\ZZ)$ such that 
		$\alpha=\alpha_0+\beta$. By Theorem \ref{secondHomologyClass},
		$\beta$ can also be represented by a closed $(R,\epsilon)$-panted
		subsurface $E$. Thus we may take $F$ to be $F_0\sqcup E$ so that $F$
		is an oriented compact $(R,\epsilon)$-panted subsurface of $M$ representing $\alpha$.
		By applying Lemma \ref{connectedPantedSurface} again, we may substitute
		$F$ with another oriented compact connected
		$(R,\epsilon)$-panted subsurface $F'$ representing $\alpha$, as desired.
		This completes the proof of Theorem \ref{main-pantedSurface}.

\section{Conclusions}\label{Sec-conclusions}
	In conclusion, we are able to construct homologically interesting connected
	immersed nearly geodesic nearly regularly panted subsurfaces in a closed hyperbolic $3$-manifold $M$ 
	by knowning a finite presentation of its fundamental group. 
	The existence of plenty of nearly regular pairs
	of pants in $M$ is a consequence of the exponential mixing property
	of the frame flow, and is the essential reason 
	for the connectedness and the $\pi_1$-injective quasi-Fuchsian property. 
	Even if we did not know the mixing rate, 
	the Connection Principle can still be deduced from the mixing property, 
	so homologically interesting connected $(R,\epsilon)$-panted subsurfaces
	can still be constructed.
		
	We propose a few further questions regarding generalization of results from this paper.
	
	\begin{question}\label{changeCoefficient}
		Is it possible to generalize Theorem \ref{main-qfSurface} to other coefficients?
		For example, if $\mathbb{F}$ is any field,
		does every homology class $\alpha\in H_2(M,L;\,\mathbb{F})$
		have represented an $\mathbb{F}$-oriented compact $\pi_1$-injectively immersed 
		quasi-Fuchsian subsurface?
	\end{question}
	
	It seems that our argument can be modified without difficulty to confirm
	Question \ref{changeCoefficient} when $\mathbb{F}$
	is any field of characteristic other than $2$. However, the $\ZZ_2$
	coefficient case is not clear since the subsurface constructed
	might be non-orientable, so	an unoriented
	$(R,\epsilon)$-panted cobordism theory needs to be developed.
	
	\begin{question}\label{changeDimension}
		Is it possible to generalize Theorems \ref{theoremPantedCobordism}
		and \ref{secondHomologyClass} to other dimensions?
		In particular, can we define and determine the $(R,\epsilon)$-panted cobordism group
		$\ocobordism_{R,\epsilon}(M)$ 
		for any oriented closed hyperbolic manifold $M$?
	\end{question}
	
	We expect that Theorem \ref{theoremPantedCobordism}
	should hold for all dimensions at least $3$. 
	In dimension $2$, it seems that $\ocobordism_{R,\epsilon}(S)$ 
	should be a split extension of $H_1(S;\ZZ)$ by $\ZZ_2$.
	This is basically because 
	the special orthonormal frame bundle $\SO(S)$ 
	should be replaced with the special orthonormal frame bundle of
	a stabilization $T(S)\oplus \epsilon^1$ of the tangent bundle $T(S)$,
	in order that Lemma \ref{canonicalLiftLemma} holds.
	
	The following two questions are much more difficult but significant. 
	To answer Question \ref{gfSubmanifolds}, we expect a notion of good basic
	pieces playing the role of nearly regular pairs of pants. 
	To answer Question \ref{cuspCase}, we need a modified version of 
	the Connection Principle since the mixing property of frame flow
	no longer holds.
	
	\begin{question}\label{gfSubmanifolds}
		How to construct (homologically interesting) 
		connected $\pi_1$-injectively immersed geometrically finite submanifolds
		in a closed hyperbolic manifold $M$?
	\end{question}
	
	\begin{question}\label{cuspCase}
		How to construct (homologically interesting) connected
		$\pi_1$-injectively immersed quasi-Fuchsian subsurfaces
		in a cusped hyperbolic $3$-manifold $M$ of finite volume?
	\end{question}

\bibliographystyle{amsalpha}

\end{document}